\documentclass[notitlepage,11pt,reqno]{amsart}  
\usepackage[foot]{amsaddr}
\usepackage[numbers,sort]{natbib}
\usepackage{amssymb,nicefrac,bm,upgreek,mathtools,verbatim,yhmath}

\usepackage[mathscr]{eucal}
\usepackage{dsfont} 	

\usepackage[margin=10pt,font=small,labelfont=bf]{caption}

\usepackage{xcolor}

\makeatletter
\renewcommand{\@setref}[3]{%
  \ifx#1\relax
    \textcolor{red}{\textbf{[??]}}%
    \G@refundefinedtrue
  \else
    \expandafter#2#1\null
  \fi
}
\makeatother

\usepackage{hyperref}

\usepackage[margin=1in]{geometry} 

\allowdisplaybreaks  
\usepackage{color}  
\definecolor{dmagenta}{rgb}{.4,.1,.5}       
\definecolor{dblue}{rgb}{.0,.0,.5}     
\definecolor{mblue}{rgb}{.0,.0,.8}     
\definecolor{ddblue}{rgb}{.0,.0,.4}            
\definecolor{dred}{rgb}{.6,.0,.0}   
\definecolor{dgreen}{rgb}{.0,.5,.0}  
\definecolor{Eeom}{rgb}{.0,.0,.5}

\usepackage[normalem]{ulem}
\usepackage{algorithm}
\newfloat{algorithm}{t}{lop}

\newtheorem{lemma}{Lemma}[section]
\newtheorem{theorem}{Theorem}[section]
\newtheorem{proposition}{Proposition}[section]
\newtheorem{corollary}{Corollary}[section]

\theoremstyle{definition}

\newtheorem{assumption}{Assumption}[section]

\theoremstyle{remark}
\newtheorem{remark}{Remark}[section]
\numberwithin{equation}{section}

\hypersetup{
  colorlinks=true,
  citecolor=dblue,
  linkcolor=dblue,
  frenchlinks=false,
  pdfborder={0 0 0},
  naturalnames=false, 
  hypertexnames=false,
  breaklinks}
\usepackage[capitalize,nameinlink]{cleveref}

\crefname{section}{Section}{Sections}
\crefname{subsection}{Section}{Sections}
\crefname{condition}{Condition}{Conditions}
\crefname{hypothesis}{Hypothesis}{Conditions}
\crefname{assumption}{Assumption}{Assumptions} 
\crefname{lemma}{Lemma}{Lemmas} 
  
\crefformat{equation}{\textup{#2(#1)#3}}
\crefrangeformat{equation}{\textup{#3(#1)#4--#5(#2)#6}}
\crefmultiformat{equation}{\textup{#2(#1)#3}}{ and \textup{#2(#1)#3}}
{, \textup{#2(#1)#3}}{, and \textup{#2(#1)#3}}
\crefrangemultiformat{equation}{\textup{#3(#1)#4--#5(#2)#6}}%
{ and \textup{#3(#1)#4--#5(#2)#6}}{, \textup{#3(#1)#4--#5(#2)#6}}%
{, and \textup{#3(#1)#4--#5(#2)#6}}

\Crefformat{equation}{#2Equation~\textup{(#1)}#3}
\Crefrangeformat{equation}{Equations~\textup{#3(#1)#4--#5(#2)#6}}
\Crefmultiformat{equation}{Equations~\textup{#2(#1)#3}}{ and \textup{#2(#1)#3}}
{, \textup{#2(#1)#3}}{, and \textup{#2(#1)#3}}
\Crefrangemultiformat{equation}{Equations~\textup{#3(#1)#4--#5(#2)#6}}%
{ and \textup{#3(#1)#4--#5(#2)#6}}{, \textup{#3(#1)#4--#5(#2)#6}}%
{, and \textup{#3(#1)#4--#5(#2)#6}}

\crefdefaultlabelformat{#2\textup{#1}#3}

\newcommand{\cB}{{\mathcal{B}}}

\newcommand{\sV}{\mathscr{V}}
\newcommand{\Lg}{\mathcal{L}}

\newcommand{\sQ}{{\mathscr{Q}}}

\newcommand{\sS}{{\mathscr{S}}}

\newcommand{\cX}{{\mathcal{X}}}

\newcommand{\frB}{\mathfrak{B}}

\newcommand{\tP}{\mathscr{P}}

\newcommand{\beql}[1]{\begin{equation}\label{#1}}

\newcommand{\beq}{\begin{displaymath}}
\newcommand{\eeqno}{\end{displaymath}}
\newcommand{\eeq}{\end{equation}}

\newcommand{\E}{\mathbb{E}}
\newcommand{\PP}{\mathbb{P}}

\newcommand{\RR}{\mathds{R}}
\newcommand{\NN}{\mathds{N}}

\newcommand{\D}{\mathrm{d}}

\newcommand{\Uadm}{\mathfrak{U}}
\newcommand{\Um}{\mathfrak{U}_{\mathrm{M}}}
\newcommand{\Usm}{\mathfrak{U}_{\mathrm{SM}}}

\newcommand{\Ind}{\mathds{1}}   
\newcommand{\Cc}{\mathcal{C}}

\newcommand{\transp}{^{\mathsf{T}}}

\newcommand{\grad}{\nabla}

\newcommand{\calP}{\mathcal{P}}

\newcommand{\cC}{\mathcal{C}}

\newcommand{\cF}{\mathcal{F}}

\newcommand{\frC}{\mathfrak{C}}

\newcommand{\wV}{\widetilde V}

\newcommand{\Upu}{\beta}

\newcommand{\vt}{\vartheta}

\makeatletter
\DeclareRobustCommand\widecheck[1]{{\mathpalette\@widecheck{#1}}}
\def\@widecheck#1#2{%
    \setbox\z@\hbox{\m@th$#1#2$}%
    \setbox\tw@\hbox{\m@th$#1%
       \widehat{%
          \vrule\@width\z@\@height\ht\z@
          \vrule\@height\z@\@width\wd\z@}$}%
    \dp\tw@-\ht\z@
    \@tempdima\ht\z@ \advance\@tempdima2\ht\tw@ \divide\@tempdima\thr@@
    \setbox\tw@\hbox{%
       \raise\@tempdima\hbox{\scalebox{1}[-1]{\lower\@tempdima\box
\tw@}}}%
    {\ooalign{\box\tw@ \cr \box\z@}}}
\makeatother

\usepackage{accents}
\newlength{\dhatheight}

\newcommand{\bU}{\mathbb{U}}

\newcommand{\Wsm}{\mathfrak{W}_{\mathrm{SM}}}
\newcommand{\Wm}{\mathfrak{W}_{\mathrm{M}}}

\newcommand{\mynegspace}{\hspace{-0.12em}}
\newcommand{\lvvvert}{\rvert\mynegspace\rvert\mynegspace\rvert}
\newcommand{\rvvvert}{\rvert\mynegspace\rvert\mynegspace\rvert}
\DeclarePairedDelimiter{\vvvert}{\lvvvert}{\rvvvert}

\let\oldtocsection=\tocsection
\let\oldtocsubsection=\tocsubsection
\let\oldtocsubsubsection=\tocsubsubsection
\renewcommand{\tocsection}[2]{\hspace{0em}\oldtocsection{#1}{#2}}
\renewcommand{\tocsubsection}[2]{\hspace{1em}\oldtocsubsection{#1}{#2}}
\renewcommand{\tocsubsubsection}[2]{\hspace{2em}\oldtocsubsubsection{#1}{#2}}

\newcommand{\ttl}{\Large Exponential rate of convergence of relative value iteration \\[5pt] algorithms 
 for  ergodic controls of diffusions  }

\begin{document}
\title[]{\ttl}

\author{Sumith Reddy Anugu$^\dagger$}
\author{Guodong Pang$^\ddag$}

\address{$^\dagger$Institut F\"ur Mathematik, Technische Universit\"at Ilmenau, Ilmenau, Germany 98693}
\email{sumith-reddy.anugu@tu-ilmenau.de}
\address{$^\ddag$Department of Computational Applied Mathematics and Operations Research,
George R. Brown School of Engineering and Computing,
Rice University,
Houston, TX 77005}
\email{gdpang@rice.edu}


\date{\today}

\maketitle

\allowdisplaybreaks

\begin{abstract}  
In this paper, we investigate the rate of convergence of the relative value iteration (RVI) algorithms for diffusions in $\RR^d$ under both the conventional ergodic cost (CEC) and ergodic risk-sensitive cost (ERSC) criteria, and under the uniform exponential stability condition. 
The existing RVI algorithms for the CEC and ERSC problems solve the associated initial value Hamilton-Jacobi-Bellman type equations whose solutions are shown to converge asymptotically to the corresponding optimal values. However, the rates of convergence for such algorithms have remained open. 
This paper proposes discrete-time implementations for the RVI algorithms based on slight modifications of the associated PDEs, and proves that the rates of convergence of these RVI algorithms are exponential under a weighted sup-norm.  These implementations have discrete-time iterates that can be explicitly expressed as recursive systems. The difference between these iterates and  the desired value function  in the CEC case can then  be expressed in terms of  the associated Markov kernels. Similarly, this can be done for the 
 logarithms of the corresponding iterates and desired value function in the ERSC case in terms of the  associated Markov kernels for the extended diffusion.  As a result, we are able to prove the desirable contraction properties in order to establish the exponential rate of convergence by making use of a weighted semi-norm in which Markov kernel acts a contraction. 

\end{abstract}

\medskip

\section{Introduction}

Ergodic control of diffusions under the conventional ergodic cost (CEC) criterion  and under the ergodic risk-sensitive cost (ERSC) criterion  have been extensively studied in the literature; see for example, the book \cite{arapostathis2012ergodic} and survey \cite{biswas2022survey}. 
In particular, under the condition that the controlled diffusion is uniformly stable, the well-posedness (existence and uniqueness of solutions) of the Hamiltonian-Jacobi-Bellman (HJB) equations,  and the characterizations of optimal controls via these HJB equations  are well established for CEC problems \cite[Chapter~3]{arapostathis2012ergodic} and for ERSC problems in  \cite{FM95,arapostathis2019strict,AB20,ABBK20}.  The solutions to these HJB equations comprise of two quantities:  the  value function and the optimal value (due to the form of the respective HJB equations, the value function is unique up to an additive constant in the CEC case and unique up to a multiplicative constant in the ERSC case). In general, explicit solutions to the HJB equations (resulting in concrete optimal controls) are difficult to obtain, except in very structured problems such as linear-quadratic problems. 
Hence, in order to find the optimal controls for these problems, one resorts to computational algorithms. Two such common algorithms  include  relative value iteration (RVI) algorithms and policy iteration algorithms. In this paper we focus on RVI algorithms only, and
 for the study of policy iteration algorithms, we 
 refer the reader to \cite{arapostathis1993discrete,meyn1997policy} (for the CEC criterion in discrete-time Markov chains - DTMCs), \cite{arapostathis2012} (for the CEC criterion in diffusions),   \cite{biswas2022ergodic,borkar2002risk} (for the ERSC criterion in DTMCs), 
  \cite{arapostathis2021policy} (for the ERSC criterion in diffusions), and \cite{biswas2022survey} (for a survey of results under the ERSC criterion). 

The brief overview of the structure of the RVI algorithms is as follows:  the iterates of these algorithms approximate the value function and the optimal value directly, and after many iterates we use these approximations, to construct a control which is nearly optimal. To do this, one fixes a reference point from the state space and treats the iterate evaluated at the reference point as an approximation of the optimal value. Then, one constructs a particular kind of recursion  such that the associated fixed point equation matches with the respective HJB equation.  The iterates of this recursion are the approximations of the value function. 

For ergodic control of diffusions, the RVI algorithm is first studied in Arapostathis et al.  \cite{arapostathis2012relative} (see also \cite{arapostathis2017correction}), which involves solving a nonlinear parabolic PDE given an initial value function, see \eqref{eq-rvi-org-0}.  
Under the uniform stability condition,  its solution is shown to converge to the optimal value function, and one also obtains the optimal cost/value and optimal stationary Markov controls as a result. This RVI algorithm was also shown to converge under the near-monotonicity condition in \cite{arapostathis2014convergence}.  
For ERSC control of diffusions, Arapostathis and Borkar \cite{arapostathis2020relative} studied the multiplicative RVI algorithm, 
which requires solving  an initial-value nonlinear parabolic PDE that is constructed directly from the multiplicative HJB equation, see \eqref{eq-rvi-org-gamma}. 
Under the positive recurrence condition of the associated ground diffusion (a special case of the extended diffusion, see \eqref{eq-Z-cont}), the convergence is established only within a neighborhood of the solution to the risk-sensitive
HJB equation, while under the uniform exponential stability condition,
the global convergence of the multiplicative RVI algorithm starting
from any positive initial condition is established. 
See also \cite{hmedi2023global} on the global convergence result under the near-monotone condition. 
However, the rates of convergence of these RVI algorithms have remained open (see \cite{arapostathis2019open}).

We recall that the implementation of the RVI algorithm requires a numerical solution to the associated initial-value PDE problem. Exact numerical schemes are irrelevant in the proofs for the convergence of the PDE solution for the RVI algorithm  in   \cite{arapostathis2012relative}. The proofs rely heavily on analytical properties of the associated PDEs, and the theory of monotone dynamical systems, using the theory of reverse martingales.
However, we find that this method is not amenable for establishing the rate of convergence.

In this paper we propose discrete-time implementations of the slightly modified RVI algorithms under both the CEC and ERSC criteria and under the uniformly  exponential stability of the controlled diffusions, and show that their rates of convergence are exponential in an appropriate sense.  For the CEC problems,  the direct implementations of the RVI algorithms are via~\eqref{eq-rvi-org-0}, which solves the associated PDEs ``once'' over a large time period, and solutions are sampled at the discretized times. 
We propose a modified implementation of the RVI algorithms given in Algorithms~\ref{alg-rvi-mod-0}. 
Specifically,
in each iterate time window, the associated PDEs are solved using the solutions at the end of the previous time period as the initial values. 
Consequently,  we obtain an  explicit recursive system for the outputs at the discretized end times of the iterate time windows. We observe that the difference between the  iterates and the desired value function  can be expressed via Markov kernels by applying It{\^o}'s formula; see~\eqref{eq-V-conv-11} and~\eqref{eq-V-conv-12}, and also,~\eqref{eq-sketch-1} in the uncontrolled case for illustration purpose. 
 To prove the rate of convergence, a key step is to establish 
the contraction property of the associated Markov kernels under an appropriate weighted semi-norm, for which we apply the  results from  Hairer and Mattingly  \cite{hairer2011}.  Specifically, this requires checking two conditions (i) a Foster-Lyapunov inequality and (ii) a minorization condition for the Markov kernel (see Theorem 3.1 of \cite{hairer2011}, summarized in Proposition~\ref{prop-contract} for convenience). Under the uniform  Foster-Lyapunov condition (Assumption~\ref{a-main}), we are able to  verify these conditions (see Corollaries~\ref{cor-lyap-drift} and ~\ref{cor-minor}). We choose to work with the aforementioned weighted semi-norm instead of the conventional weighted sup-norm because  using the weighted semi-norm, we can establish the convergence of the iterates of the value function to the desired value function and then proceed to use this to  analyze the convergence of the value iterates to the optimal value. In contrast, this kind of ``decoupling'' does not seem to follow upon working with weighted sup-norm.

  We use the contraction property of Markov kernels to establish that the iterations of the RVI algorithm converge to the desired value function of the CEC problem under the appropriately chosen weighted semi-norm, at an exponential rate; see Theorem~\ref{thm-rvi-0}. 
   This is done by observing that the difference between iterates  of the value function of our modified RVI algorithm and the desired value function is bounded (from above and below) by integrals of the difference of the preceding iterates and the desired functions over Markov kernels (associated with appropriate choice of Markov controls), modulo additive terms that are constants; see Equations~\eqref{eq-V-conv-11} and~\eqref{eq-V-conv-12}. From here, using the aforementioned contraction property, we conclude that the iterates of the value function in our algorithm converge exponentially to the desired value function in the weighted semi-norm.  Next,  we show that the sequence of  value function iterates evaluated at the reference point converges to the desired optimal value at an exponential rate, in Lemma~\ref{lem-L-conv}. Finally,  using the relation between weighted semi-norm and sup-norm (see Lemma~\ref{lem-compare}), we conclude that the value function iterates of our modified algorithm converge to the desired value function in the weighted sup-norm, at an exponential rate; see Corollary~\ref{cor-norm-conv}.

   \smallskip
   
 The proof for the rate of convergence of the RVI algorithm the ERSC problem is much more involved. In this case, we assume  a stronger form of Foster-Lyapunov condition (Assumption~\ref{a-main-gamma}) and also restrict the  initial condition of the RVI algorithm to non-negative continuous functions with appropriate growth condition. Under these conditions, we establish that the logarithm of the  iterates of our modified RVI algorithm converges to the logarithm of the desired value function in the weighted semi-norm in Theorem~\ref{thm-rvi-gamma}. We start with transforming the recursive system, multiplicative HJB equation  and 
 the Lyapunov condition into the corresponding additive forms by an exponential transformation in Sections~\ref{sec-add-eq-hjb-gamma} and~\ref{sec-add-a-main-gamma}, respectively. The recursive system and the resulting HJB equation have an additional maximization operation over an auxiliary control set. That corresponds to a CEC problem for an extended diffusion, with the auxiliary control in the drift,  and a quadratic (in terms of the auxiliary control) penalty function in the running cost. From this setup, it might appear that the methodology developed for the CEC problem could be easily adapted for this, but unfortunately, this is not the case and the verification of the Lyapunov and  minorization conditions (for the Markov kernel associated with the extended diffusion) in Proposition~\ref{prop-contract} is substantially more difficult to establish. 
 This is because unlike the CEC problem, the  verification of these conditions involves obtaining the appropriate moment bounds (uniform in the iteration number) in terms of the auxiliary control. However, these bounds are difficult to obtain.

 To overcome this difficulty, we introduce an intermediate  family of iterates (via~\eqref{eq-widehat}) which are  identical to the family of  iterates of our modified RVI algorithm, modulo an additive constant term depending only on the iteration number. The reason for such a construction is three-fold: (1) In terms of the weighted semi-norm, the desired value function is equidistant from the iterates of our modified RVI algorithm and  the intermediate family of iterates (see~\eqref{eq-const-series});    (2) The aforementioned moments bounds (uniform in iteration number) on the auxiliary controls associated with the intermediate  iterates is easier to establish in this case; (3) The difference between the logarithm of the intermediate  iterates and the logarithm of the desired value function is bounded (from above and below) by integrals of the difference of the preceding intermediate iterates and the desired functions over Markov kernels associated with appropriate choice of Markov controls for the extended diffusion, modulo constant additive terms (independent of the iteration number); see Equations~\eqref{eq-mk-exp-ub} and~\eqref{eq-mk-exp-lb}. We begin by proving the one-step contraction under the assumption that certain weighted    sup-norm between the intermediate iterates and the desired value function is strictly less than $1$.     In light of the above reasons, we establish the aforementioned moments bounds for the extended diffusion for a specific choice of  control-auxiliary control pairs.     Then, for the extended diffusion under the aforementioned control pairs, we verify the aforementioned  Foster-Lyapunov condition (in Lemma~\ref{lem-exp-conv}) and  the minorization condition  (in Lemma~\ref{lem-exp-minor}) for the applicability of Proposition~\ref{prop-contract}. 
 From hereon, following the arguments from the CEC case, we establish that if the scale of the weight of the weighted sup-norm is chosen small enough,  the weighted semi-norm of the difference between the logarithms of the  intermediate iterates and the desired value function satisfies a one-step contraction property. We then show that the above condition also holds for the next iterate; hence  by repeating the argument,   the one-step contraction property is shown to be  preserved for all the iterates, given that the first iterate satisfies the aforementioned property on the weighted sup-norm.   It immediately follows that the logarithm of the iterates of our RVI algorithm converge to that of the desired value function under the weighted semi-norm at an exponential rate. 
  Using the arguments from the CEC case, we proceed to show that the exponential convergence, in fact, also holds under the weighted sup-norm.

 \smallskip

To the best of our knowledge,  this is the first work concerning the analysis of rate of convergence of RVI algorithms for ergodic control of Markov processes under either CEC or ERSC criteria beyond the case of finite-state space. 
We refer  to \cite{white1963dynamic,federgruen1978contraction,schweitzer1988contraction} for the geometric rate of convergence analysis of White's RVI algorithm for DTMCs under the CEC criterion, and \cite{bielecki1999risk,cavazos2003value} for that of direct extension of White's RVI algorithm for DTMCs under the ERSC criterion. 
Both lines of work make use of the contraction property of the risk neutral or risk sensitive Bellman operators in order to prove the rate of convergence in the span-norm.
 We also refer to \cite{bertsekas1998new} for the geometric rate of convergence of Jacobi-like RVI algorithm for DTMCs under the  CEC criterion, and  \cite{anugu2025new} for that of the corresponding Jacobi-like RVI algorithms under the  ERSC criterion. Both works show the rate of convergence in certain weighted sup-norms, by exploiting either the (global or local) contraction properties of the associated Bellman operators. Like these works, the choice of the norms for the convergence also plays a significant role in our study, which uses the weighted sup-norm (that is closely related to a weighted semi-norm) inspired by \cite{hairer2011}. 
For DTMCs with countable state space, we refer to \cite{aviv1999value, hordijk1974convergence,cavazos1996value,cavaozs1996value1} under the CEC criterion and \cite{borkar2002risk} under the ERSC criterion, for the convergence of the RVI algorithms, but no rate of convergence is established; see the same results in \cite{arapostathis2020relative} for DTMCs with general Polish space. 

We believe that our framework may be employed and further developed to study rates of convergence of RVI algorithms in CEC and ERSC formulations for DTMCs with countable and more general state spaces, as well as for continuous-time Markov processes with general state spaces.  
In the context of diffusions,  the methodology could be potentially extended to study the rate of convergence of the value iteration algorithms for other types of control problems, by exploiting properties of the associated Markov kernels.
For example, a RVI algorithm is developed for zero-sum  stochastic  differential games in \cite{arapostathis2013relative}, and our methodology may be extended to study its rate of convergence. It would be interesting to investigate these in future work.

\subsection{Notation} In the entire paper, we set $(\Omega, \mathcal{F},\mathcal{F}_t,\PP)$ as the underlying filtered probability space. For any $k$ and $m$, $\cC^m(\RR^k)$ is the set of $m$--times differentiable functions on $\RR^k$. Let $\RR^k_T\doteq [0,T]\times \RR^k$ and $\cC^{l,m}(\RR^k_T)$ denoted the set of functions on $\RR^k_T$ that are $l$--times differentiable in the time coordinate and $m$--times differentiable in the space coordinate on every open subset of $\RR^k_T$.  For any measurable space $(\cX,\cB)$, $\calP(\cX)$ is the set of probability measures on $(\cX,\cB)$. For a non-negative   $\cB$-measurable function $\mathscr{F} :\cX\rightarrow \RR$ and $\beta>0$, we set   
\begin{align}\label{def-norm} &\qquad\qquad\|f\|_{\Upu,\mathscr{F}}\doteq \sup_{x\in \cX} \frac{|f(x)|}{1+\Upu \mathscr{F}(x)} 
\end{align}
and define 
\begin{align*}
\cC_{\beta,\mathscr{F}}(\cX)&\doteq \Big\{f:\cX\rightarrow \RR: f \text{ is $\cB$-measurable and }  \|f\|_{\Upu,\mathscr{F}}<\infty\Big\}\,.\end{align*} The space $\cC_{\beta,\mathscr{F}}(\cX)$ is equipped with topology induced by $\|\cdot\|_{\Upu,\mathscr{F}}$. The Lebesgue measure on $\RR^k$ is denoted by $m$, \emph{i.e.,} $m(A)=\int_A \D x$, for  every Borel measurable set $A\subset \RR^k$. For a Borel set $\frB\subset \RR^k$, $\widehat \tau_\frB$ denotes the first time the underlying process enters the set $\frB$.

\medskip

\section{Controlled Diffusion Model} 
We consider a $\RR^d$--valued controlled diffusion $X=\{X_t: t\ge 0\}$ given as the strong solution to  
\begin{equation} \label{eqn-X}
X_t = {X}_0 + \int_0^t b(X_s, U_s)\D s + \int_0^t\Sigma(X_s) \D W_s\,.
\end{equation}
Here, the process  $U$ (referred to as control) is assumed to take values in a compact metric space $\bU$ and the coefficients $b$ and $\Sigma$ satisfy the assumptions below. 
\begin{assumption}\label{a-regularity} 
\begin{enumerate}
\item[]
\item[(i)] (Local Lipschitz continuity) $b:\RR^d\times \bU\rightarrow \RR^d$ and $\Sigma:\RR^d\rightarrow\RR^{d\times d}$ are continuous and for every $R>0$, there exists $C_R>0$ such that  
\begin{align*}
\|b(x,u)-b(y,u)\|+ \|\Sigma(x)-\Sigma(y)\|\leq C_R\|x-y\|, \text{ for  $x,y\in B_R$,  and $u \in \bU$}\,.
\end{align*}
\item[(ii)] (Linear growth in drift) There exists a constant $C_b>0$ such that 
$$ \|b(x,u)\|^2  \leq C_b(1+\|x\|^2), \text{ for $x\in \RR^d$ and $u \in \bU$\,.}$$
\item [(iii)](Non-degeneracy and boundedness of diffusion coefficient) There exists $\sigma>0$ such that 
$$ \sigma^{-1} \|z\|^2\leq z\transp\Sigma(x)\big(\Sigma(x)\big)\transp z\leq \sigma \|z\|^2, \text{ for $z\in \RR^d$  and $x \in \RR^d$}\,.$$
\end{enumerate}
\end{assumption}
 For simplicity, assume that  ${X}_0={x}$ is a deterministic constant.

A $\bU$--valued control process $U$ is said to be admissible if it satisfies the following: if $U_t=U_t(\omega)$ is jointly measurable in $(t,\omega)\in \RR^+\times \Omega$ and  for every $0\leq s< t$, $W_t-W_s$ is independent of the completed filtration (with respect to $(\Omega, \cF,\PP)$) generated by $\{X_0,U_r, W_r: r\leq s\} $. The set of all such controls is denoted by $\Uadm$.
Let $\Um\subset \Uadm$ be the set of Markov controls,  \emph{i.e.,} every $u\in \Um$ is of the form $u_t=v(t,X_t)$, for some Borel measurable function $v:\RR_+\times \RR^d\rightarrow \bU$. With slight abuse of notation, we write $v(t,\cdot)$ to denote a generic Markov control. If $v(t,\cdot)$ is independent of $t$, then $v$ is referred to as a stationary Markov control. Let $\Usm\subset \Um$ denote the set of stationary Markov controls.  For every $u\in\bU$, we denote the generator $\Lg^{u}:\Cc^{2}(\RR^{d})\mapsto\Cc(\RR^{d})$ of the controlled diffusion ${X}$ as 
\begin{align}\label{eq-gen}
\Lg^{u} f(x) & \doteq   \sum_{i=1}^d b_i(x,u) \frac{\partial}{\partial x_i} f(x) + \frac{1}{2}\sum_{i,j=1}^d A_{ij}(x) \frac{\partial^2}{\partial x_i \partial x_j} f(x)\,, 
\end{align}
where $A(x)\doteq \Sigma(x)\Sigma(x)\transp$. In the following, whenever we are dealing with a generic admissible control, we denote it by $U$ and a generic stationary Markov control is denoted by $v$. 

Let $r:\RR^d\times\bU\rightarrow \RR_+$ be the running cost that is continuous  and for every $R>0$, there exists $K_R>0$ such that
\begin{align*}
|r(x,u)-r(y,u)|\leq K_R\|x-y\|, \text{ for  $x,y\in B_R$,  and $u \in \bU$}\,.
\end{align*}
 To keep the expressions concise, we let 
$$  \Lg^{v} f(x) \doteq   \sum_{i=1}^d b_i(x,v(x)) \frac{\partial}{\partial x_i} f(x) + \frac{1}{2}\sum_{i,j=1}^d A_{ij}(x) \frac{\partial^2}{\partial x_i \partial x_j} f(x)\,, $$
whenever the underlying control is $v\in \Usm$. In what follows, we emphasize that  the initial condition is $X_0=x$ and  that the underlying controls are $U\in \Uadm$ and $v\in \Usm$ by writing expectation $\E$ as  $\E_x^{U}$ and $\E_x^v$, respectively. 

\begin{remark}From hereon, concerning quantities like the value function, the optimal value, iterates of an algorithm and so on, the superscript  is always used to denote iteration number. The subscript on these quantities is used to denote either the case of conventional ergodic control problem (in which case, we use subscript $0$; it is the content of Section~\ref{sec-gamma-0}) or the case of  ergodic risk sensitive control problem with risk sensitivity parameter $\gamma>0$ (in which case, we use subscript $\gamma$; this is the content of Section~\ref{sec-gamma}). 
\end{remark}

\section{Rate of Convergence of RVI Algorithm in CEC Problems}

\subsection{Conventional ergodic control problem}\label{sec-gamma-0}

The conventional ergodic control (CEC)  cost function is given by
\begin{equation*} 
J_0(x, U) \doteq \limsup_{T\to\infty} \frac{1}{T}  \E_x^U\left[ \int_0^T r({X}_t, U_t )\D t \right]\,.
\end{equation*}
The associated CEC cost minimization problem is given by
\begin{equation}\label{def-inf-u}
{\Lambda}^*_0({x}) \doteq  \inf_{U\in \Uadm}  J_0({x}, U)\, \quad \text{ and } \quad \Lambda^*_0\doteq \inf_{x\in \RR^d} \Lambda_0(x)\,.
\end{equation}
In addition, let 
\begin{equation}\label{def-inf-sm}
\Lambda_{\text{SM},0}^*({x}) \doteq \inf_{v\in \Usm}  J_0({x}, v) \quad \text{ and } \quad  \Lambda^*_{\text{SM},0}\doteq \inf_{x\in \RR^d}\Lambda_{\text{SM},0}(x)\,.
\end{equation}
 It is easy to see that $\Lambda^*_0\leq \Lambda^*_{\text{SM},0}$. We now state a Foster-Lyapunov condition, which along with Assumption~\eqref{a-regularity} implies that the controlled diffusion $X_t$ is uniformly exponentially stable.   
\begin{assumption}\label{a-main} There exist an inf-compact $\cC^2(\RR^d) $ function $\sV\geq 0$ and constants $\lambda_0,\lambda_1 >0$  such that for every $x\in \RR^d$ and $u\in \bU$, we have
\begin{align}\label{eq-erg-0}
\Lg^u \sV(x)\leq -\lambda_0\sV(x)+\lambda_1\,.
\end{align}
\end{assumption}
We now recall the well-known existing result (which is \cite[Theorem 3.7.12]{arapostathis2012ergodic}) concerning the associated HJB equation and complete characterization of the optimal stationary Markov controls.

\begin{theorem}\label{thm-hjb-0} Suppose Assumptions~\ref{a-regularity} and~\ref{a-main} hold and the running cost $r$ satisfies 
\begin{align}\label{eq-r-growth}
\limsup_{\|x\|\to\infty} \sup _{u\in \bU} \frac{|r(x.u)|}{ \sV(x)}<\infty\,.
\end{align}
 Then, there exists a function $V^*_0\in \cC^2(\RR^d)\cap \cC_{\Upu,\sV}(\RR^d) $, unique up to an additive constant,  such that 
\begin{align}\label{eq-hjb-0}
\min_{u\in \bU}\big[\Lg^uV^*_0(x) +r(x,u)\big]=\Lambda^*_0\,.
\end{align}
Moreover, $\Lambda^*_0=\Lambda_{\text{SM},0}^*$ and $v\in \Usm$ is optimal if and only if  
$$ \Lg^{v(x)} V^*_0(x)+r\big(x,v(x)\big)= \min_{u\in \bU}\big[\Lg^uV^*_0(x) +r(x,u)\big], \text{ a.e. $x\in \RR^d$.}$$
\end{theorem}
\begin{remark}
Since $V^*_0$ is unique up to an additive constant, we can fix the value of $V^*_0$ at any chosen  point to a chosen value. In particular, we choose the state $0$ and the corresponding value to be  $0$ for convenience, \emph{i.e.,} $V^*_0(0)=0$.
\end{remark}

\subsection{Numerical computation of $(V^*_0,\Lambda^*_0)$}
The objective  is to numerically compute $(V^*_0,\Lambda^*_0)$. We do this by building on the existing RVI algorithm  from \cite[Pg. 1891]{arapostathis2012relative}: the approximation to $(V^*_0,\Lambda^*_0)$ is given by $(\widetilde V(t,x)-\widetilde V(t,0), \widetilde V(t,0))$, for large $t$, where  $\widetilde V(t,x)$ is the  function in $\cC^{1,2}(\RR_T^d)$ for every  $T>0$,  that is the unique classical 
solution to the PDE: 
\begin{align}\label{eq-rvi-org-0}
\partial_t\widetilde V(t,x)= \min_{u\in \bU}\big[\Lg^u\widetilde V(t,x) +r(x,u)\big] - \widetilde V(t,0),\quad \widetilde V(0,x)=V^0(x),
\end{align} 
where $V^0\in \cC^2(\RR^d)\cap \cC_{\Upu,\sV}(\RR^d)$. In Theorem 2.3 of \cite{arapostathis2012relative}, the convergence of $\widetilde V(t,\cdot)$ as $t\to\infty$ is investigated, which is restated in  Proposition~\ref{prop-conv-0}  below.   
\begin{proposition}\label{prop-conv-0}
Suppose Assumptions~\ref{a-regularity} and~\ref{a-main} hold, the running cost $r$ satisfies~\eqref{eq-r-growth} and for  $\Upu>0$ and $V^0\in \cC^2(\RR^d)\cap\cC_{\Upu,\sV}(\RR^d)$. Then, there exists a unique function $\widetilde V\in \cC_{\Upu,\sV}(\RR^d_T)\cap \cC^{1,2}(\RR^d_T)$, for all $T>0$ that satisfies~\eqref{eq-rvi-org-0}. Additionally as $t\to\infty$, $\widetilde V(t,0)$ converges to $\Lambda^*_0$  and $\widetilde V(t,\cdot)$  converges to $V^*_0(\cdot)+ \Lambda^*_0$ uniformly on the compact sets of $\RR^d$. 
\end{proposition}

The above proposition is the reason behind referring to~\eqref{eq-rvi-org-0} as an algorithm - solving this PDE numerically, gives us the approximations of $V^*$ and $\Lambda^*$, as $t\to\infty$.
It is clear that even though the above RVI algorithm is given in terms of  continuous time $t\geq 0$, we can easily obtain a `discrete' version of the existing RVI algorithm  as follows: 
 fix $\tau>0$ and choose $n^*\in \NN$, with an initialization  $ V^0 \in \cC^2(\RR^d)\cap\cC_{\Upu,\sV}(\RR^d)$, for $0\leq t\leq n^* \tau $, solve the PDE in \eqref{eq-rvi-org-0}, and for  $n\leq n^*$ , set $\widetilde V^n(\cdot)= \widetilde V(n\tau,\cdot)\in \cC^2(\RR^d)\cap\cC_{\Upu,\sV}(\RR^d)$, and then output $\big(\widetilde V^n-\widetilde V^n(0),  \widetilde V^n(0)\big)$ as the approximation of  $(V^*_0,\Lambda^*_0)$.

\setcounter{algorithm}{0}

 In the following, we provide a modified `version' of the RVI algorithm. Fix $\delta,\tau>0$ such that $\delta\tau<1$; see Remark~\ref{rem-delta-tau} for further discussion regarding this choice.

 \setcounter{algorithm}{0}
\renewcommand{\thealgorithm}{RN}

\begin{algorithm}[H]
	\caption{Modified `discrete' version of RVI Algorithm}\label{alg-rvi-mod-0}
	\begin{enumerate}
	\item [(i)] Choose $\delta,\tau>0$ and $n^*\in \NN$.
		\item [(ii)] Input: initialize with $n=0$ and $V^0 \in \cC^2(\RR^d)\cap\cC_{\Upu,\sV}(\RR^d)$.
		\item [(iii)] Update: for $0\leq t\leq \tau$, solve the PDE below 
		\begin{align}\label{eq-rvi-mod-0}
\partial_t V(t,x)= \min_{u\in \bU}\big[\Lg^uV(t,x)+ r(x,u)\big]- \delta V^n(0),\quad V(0,x)= V^n(x)\,. 
\end{align}
		\item [(iv)] Set $n\leftarrow n+1$ and $V^n= V(\tau,\cdot)  \in \cC^2(\RR^d)\cap\cC_{\Upu,\sV}(\RR^d)$.
		\item [(v)] While $n\leq n^*$, repeat Steps (iii) and (iv). 
		\item [(vi)] Output: the approximate of $(V^*_0,\Lambda^*_0)$ is $\big(V^n-V^n(0), \delta V^n(0)\big)$. 
	\end{enumerate}
\end{algorithm}
\begin{remark}\label{rem-delta-tau}
Although the above algorithm is well-defined for all choices of $\delta>0$ and $\tau>0$, it is not clear if iterates $\big(V^n-V^n(0), \delta V^n(0)\big)$ of the above algorithm with an arbitrary choice of $(\delta,\tau)$ converge to  $(V^*_0,\Lambda^*_0)$. In Theorem~\ref{thm-rvi-0}, we will see that in addition to Assumptions~\ref{a-regularity} and~\ref{a-main}, if we have  $\delta\tau<1$, then we can guarantee that the above convergence holds and is, in fact, exponential under $\|\cdot\|_{\beta,\sV}$. In other words, $\tau>0$ can be chosen arbitrarily positive and $\delta>0$ is chosen such that $\delta\tau<1$. Since the expression for $V^{n+1}$ in~\eqref{eq-expression} only involves $\delta$ \emph{via.} $\delta V^n(0)$,   this term need not necessarily be very small even if $\delta$ is.  However, the impact of the choice of $\tau$ on the numerical implementation is as follows:  the computational effort needed to compute the next iterate from the current iterate is less (more, resp.) when  $\tau$ is small (large, resp.). On the other hand, according to Theorem~\ref{thm-rvi-0}, the number of iterations needed to approximate $(V^*_0,\Lambda^*_0)$ with a given accuracy is more for smaller $\tau$.

\end{remark} 
\begin{remark}\label{rem-compare} There are two main differences between the existing RVI algorithm and Algorithm~\ref{alg-rvi-mod-0}:  (i) the terms  $\widetilde V(t,0)$ and $\delta V^n(0)$ in~\eqref{eq-rvi-org-0} and~\eqref{eq-rvi-mod-0}, respectively and (ii) in the existing RVI algorithm, we solve~\eqref{eq-rvi-org-0} only once over the interval $[0,n^*\tau]$ with initial condition $V^0$, whereas in Algorithm~\ref{alg-rvi-mod-0}, we solve~\eqref{eq-rvi-mod-0} over $[0,\tau]$ for $n^*$ times with successive initial conditions $V^0, V^1,\ldots, V^{n-1}$ for $n<n^*$.
\end{remark}

The following result guarantees the well-posedness of~\eqref{eq-rvi-mod-0} in Step (iii) above.  Recall that $\RR_\tau^d= [0,\tau]\times \RR^d$. 
\begin{proposition}\label{prop-ext-0}
Suppose Assumptions~\ref{a-regularity} and~\ref{a-main} hold. For $\Upu>0$, there exists a unique solution $ V\in \cC_{\Upu,\sV}(\RR^d_\tau)\cap \cC^{1,2}(\RR^d_\tau)$ to~\eqref{eq-rvi-mod-0} with $ V(0,x)=f(x)\in\cC^2(\RR^d)\cap\cC_{\Upu,\sV}(\RR^d) $. 
\end{proposition}
\begin{proof} The proof uses arguments in  the proof of Lemma 4.1 in \cite{arapostathis2012relative} where the existence of a unique solution $\widecheck V\in \cC_{\Upu,\sV}(\RR^d_\tau)\cap \cC^{1,2}(\RR^d_\tau)$ is established, which satisfies 
\begin{align}\label{eq-sol}
\partial_t \widecheck V(t,x)= \min_{u\in \bU}\big[\Lg^u\widecheck V(t,x)+ r(x,u)\big]- g(t),\quad \widecheck V(0,x)= f(x)\in\cC^2(\RR^d)\cap\cC_{\Upu,\sV}(\RR^d),
\end{align}
for a bounded continuous function $g:[0,\tau]\rightarrow \RR^d$.  We next set $g(t)= \delta V^n(0)$ to obtain the desired result.
\end{proof}

Before we move on to the analysis of the convergence of the above algorithm, we discuss its advantage over the existing RVI algorithm.  From the first of part of the  Proposition~\ref{prop-conv-0}, we know that for every $t>0$, there exists a map $\widetilde \sS_{t}:\cC^2(\RR^d)\cap\cC_{\Upu,\sV}(\RR^d)\rightarrow \cC^{2}(\RR^d)\cap \cC_{\Upu,\sV}(\RR^d)$ that satisfies 
$$\widetilde   V(t,x)= (\widetilde \sS_{t} V^0)(x)\,.$$
However, 
a more explicit form of the map $\widetilde \sS_{t}$ is not apparent. This issue makes the analysis of the RVI algorithm in~\eqref{eq-rvi-org-0}, more challenging.  

  In contrast, the map analogous to $\widetilde\sS_{t}$ can be expressed in a more explicit form that is amenable to the convergence analysis that will follow.
 From Proposition~\ref{prop-ext-0}, we know for $0\leq t\leq \tau $,  there exists a map $\sS_{t}: \cC^2(\RR^d)\cap \cC_{\Upu,\sV}(\RR^d)\rightarrow \cC^{2}(\RR^d)\cap \cC_{\Upu,\sV}(\RR^d) $ that satisfies 
 $$ V(t,x)=(\sS_{t} f)(x), \quad \text{for $0\leq t\leq \tau$}$$
 with $V(t,x)$ being the solution of~\eqref{eq-rvi-mod-0} and $V(0,x)=f(x)$. In particular, we have  $V(\tau,x)= (\sS_{\tau} f)(x)$. To see that the map $\sS_{\tau}$ can be expressed in a more explicit form,  we apply It\^o's formula to $V(\tau-s,X_s)$ on $[0,t]$ with $t\leq \tau$, we obtain for $0\leq t\leq \tau$,
 \begin{align}\label{eq-vn-min}
 V(t,x)&= \min_{U\in \Uadm}\E^U_x\Big[\int_0^t \big(r(X_s,U_s)-\delta f(0)\big)\D s + f(X_t)\Big] \doteq (\sS_{t} f)(x)\,. 
 \end{align}
 To get the second equality, we match the right hand side of the first line with the definition of $\sS_{t}$. In particular, we can conclude that 
 \begin{align}\label{def-Sn} (\sS_{\tau} f)(x)= \min_{U\in \Uadm}\E^U_x\Big[\int_0^\tau \big(r(X_t,U_t)-\delta f(0)\big)\D t + f(X_\tau)\Big]\,.\end{align}
Hence, in each iteration $n=0,1,2,..., n^*-1$, we have \begin{align}\label{eq-Sn}V^{n+1}(x) = (\sS_{\tau} V^n)(x).\end{align} By repeated iteration, we can write for $n\ge 1$,  \[V^n(x) =   (\sS^n_{\tau} V^0)(x),\]
  where $\big(\sS_{\tau}^n f\big)(x)\doteq (\sS_{\tau}^{n-1} \sS_{\tau} f\big)(x)$, for every $n>2$ with $\big(\sS_{\tau}^2 f)(x)\doteq \big(\sS_{\tau} (\sS_{\tau} f)\big)(x)$.
  Finally, the output of Algorithm~\ref{alg-rvi-mod-0} gives the approximation  of $(V^*_0,\Lambda^*_0)$ as follows: 
  \[
\big(V^n(\cdot)- V^n(0), \delta V^n(0)\big) =  \Big((\sS^n_{\tau} V^0)(\cdot)- (\sS^n_{\tau} V^0)(0), \delta (\sS^n_{\tau} V^0)(0)\Big)\,. 
  \]
We note that, in comparison with the existing RVI algorithm, in each iteration, it requires to solve the PDE~\eqref{eq-rvi-mod-0} in Step (iii), but the computation is only over the interval $[0,\tau]$, with an updated initial condition $V(0,x) = V^n(x)$.

 \begin{remark}\label{rem-int}
 Before we proceed further, we intuitively verify that  for the $V^n$ defined above, whenever it converges, the limit of $\big(V^n-V^n(0),\delta V^n(0)\big)$ is $(V^*_0, \Lambda^*_0)$ . 
To do this, we suppose that $V^n$ is  convergent, \emph{i.e.,} $V^n$ is independent of $n$, for large $n$. Then, from Theorem~\ref{thm-hjb-0}, it is clear that $\delta V^n(0)$ approaches $\Lambda^*_0$ and that $V^n$ approaches $V^*_0$, up to an additive constant. But as $V^*_0(0)=0$ and $\delta V^n(0)$ approaches $\Lambda^*_0$, we can intuitively conclude that $V^n(\cdot)-V^n(0)$ approaches $V^*_0(\cdot)$. 
\end{remark}

 \subsection{First main result} 
 The notations  $\|\cdot\|_{\beta,\sV}$ and $\vvvert{\cdot}_{\beta,\sV}$ used below are defined~\eqref{def-norm} and~\eqref{def-norm-w}, respectively. Here, $\beta>0$ and $\sV$ is the inf-compact function from Assumption~\ref{a-main}.    Define 
 $$ \Phi^n(x)\doteq V^n(x)-V^*_0(x)-\delta^{-1} \Lambda^*_0\,.$$
 As $V^*_0(0)=0$, we have $\Phi^n(0)= V^n(0)-\delta^{-1}\Lambda^*_0$. Recall that $m(\cdot)$ is the Lebesgue measure on $\RR^d$.

 \begin{theorem}\label{thm-rvi-0}
Suppose Assumptions~\ref{a-regularity} and~\ref{a-main} hold, the running cost $r$ satisfies~\eqref{eq-r-growth} and $0<\delta\tau<1$. Then,  for $R>2\lambda_0\lambda_1^{-1}$, we have 
\footnote{Here and in the following, to avoid any confusion with the superscript $n$ denoting the $n$th iterate, we write the $n$th power in the rates of convergence as $(\kappa)^n$, $(\tilde\kappa)^n$, $(1-\delta\tau)^n$ etc.} 
\begin{align*}
\vvvert{\Phi^n}_{\Upu,\sV}&\leq (\kappa) ^n \vvvert{\Phi^0}_{\Upu,\sV},\\
|\Phi^n(0)|&\leq  (\widetilde \kappa)^n\bigg(| \Phi^0(0)| + \frac{C_0}{e\kappa_{\max} \log (\widetilde \kappa \kappa_{\max}^{-1})}  \vvvert{\Phi^0}_{\Upu,\sV}\bigg), \end{align*}
for every $n\geq 1$, where, $C_0\doteq 2+2\Upu\sV(0)+\Upu\lambda_1\tau$, 
\begin{align}\label{def-kappa} 0<\kappa\doteq \big(1+\lambda_3-\alpha_{R,\tau}m(K)\big) \vee \Big(\frac{2+\lambda_2\Upu R}{2+\Upu R}\Big)&<1,\quad  \kappa_{\max}\doteq  \max\{\kappa, 1-\delta\tau\}<\widetilde \kappa<1\,,
\end{align}
with
\begin{align}\label{def-upu} \lambda_2\doteq e^{-\lambda_0\tau} + \frac{2\lambda_1(1-e^{-\lambda_0\tau})}{\lambda_0 R}<1,\quad 0<\lambda_3 &<\alpha_{R,\tau}m(K)<1,\quad  \Upu\doteq \frac{\lambda_3 \lambda_0}{\lambda_1(1-e^{-\lambda_0\tau})}. 
\end{align}
Here, $\alpha_{R,\tau}$ and the compact set $K\subset \RR^d$ are from Corollary~\ref{cor-minor}.

\end{theorem} 
\begin{corollary}\label{cor-norm-conv}
Under the conditions of Theorem~\ref{thm-rvi-0} and with the constants there, we have 
\begin{align}\label{eq-contraction-2} \|\Phi^n\|_{\Upu,\sV}\leq C (\widetilde  \kappa)^n\|\Phi^0\|_{\Upu,\sV}, \end{align}
where, \[C\doteq  \bigg( 3+\Upu\sV(0) +  \frac{(2+2\Upu\sV(0)+\Upu\lambda_1\tau)}{e\kappa_{\max} \log (\widetilde \kappa \kappa_{\max}^{-1})}\bigg).\]
\end{corollary}

 The proofs of both the results above are deferred to Section~\ref{sec-proof-1}.
\subsection{ Sketch of the proof of Theorem~\ref{thm-rvi-0}}
Here, we  sketch the arguments involved in the proof of Theorem~\ref{thm-rvi-0}. To prove the theorem, it suffices to show that  for $n\geq 1$, \begin{align}\label{eq-sketch-V}
\vvvert{\Phi^n}_{\Upu,\sV}&\leq \kappa  \vvvert{\Phi^{n-1}}_{\Upu,\sV}\\\label{eq-sketch-lambda}
|\Phi^n(0)|&\leq (1-\delta\tau) |\Phi^{n-1}(0)| + C_0  (\kappa)^n \vvvert{\Phi^0}_{\Upu,\sV} \,.\end{align}   
Indeed,  a repeated application of the above inequalities and simple algebraic computations then provide us the result. To understand the methodology of the proof,  we consider the simpler problem where the control set $\bU=\{ u^*\}$, a singleton set. In other words, our simpler problem becomes uncontrolled and hence, the infimum operation over $\bU$ in~\eqref{def-inf-u} and~\eqref{def-inf-sm} 
and  minimization operation in~\eqref{eq-hjb-0},~\eqref{eq-rvi-org-0},~\eqref{eq-rvi-mod-0} and~\eqref{eq-vn-min} become redundant. In this case, we discuss the convergence of iterates $V^n$ of Algorithm~\ref{alg-rvi-mod-0}.  Recall that $V^n$ in terms of $V^{n-1}$ is given by~\eqref{def-Sn} and~\eqref{eq-Sn}: for $x\in \RR^d$,
\begin{align}\label{eq-expression}V^{n+1}(x)=\E^{u^*}_x\Big[\int_0^\tau \big(r(X_t,u^*)-\delta V^{n-1}(0)\big)\D t + V^{n-1}(X_\tau)\Big]\,.
\end{align}
From Remark~\ref{rem-int}, we have intuitively seen that $V_0^*+\delta^{-1}\Lambda^*_0$ is the fixed point of the above iteration, \emph{i.e.,} 
\begin{align*}V_0^{*}(x)=\E^{u^*}_x\Big[\int_0^\tau \big(r(X_t,u^*)-\Lambda^*_0\big)\D t + V_0^{*}(X_\tau)\Big]\,.
\end{align*}
From the above two displays, it is clear that 
\begin{align} V^n(x)-V_0^*(x)&= \E_x^{u^*}\Big[ V^{n-1}(X_\tau)-V_0^{*}(X_\tau)\Big] + \tau\big( \delta V^{n-1}(0)-\Lambda^*_0\big)\nonumber\\\label{eq-sketch-1}
&= \Big(\tP^{u^*}_\tau\big(V^{n-1}-V_0^*\big)\Big)(x) + \tau\big( \delta V^{n-1}(0)-\Lambda^*_0\big)\,  .\end{align}
Here, $\tP_\tau^{u^*}$ the Markov kernel associated with $X$ under $u^*$; see~\eqref{def-markov-ker-x} for the definition. To the above display, we add a constant  $\widetilde c_{u^*}-\delta^{-1}\Lambda^*_0$ (where $\widetilde c_u^*= c_{\tP^{u^*}_\tau\Phi^{n-1}}$ which is given in~\eqref{def-c*}), divide by $1+\Upu \sV(x)$ on both sides and use Lemma~\ref{lem-compare} to get
\begin{align*} \frac{V^n(x)-V_0^*(x)- \tau\big( \delta V^{n-1}(0)-\Lambda^*_0\big) +\widetilde c_{u^*} -\delta^{-1}\Lambda^*_0}{1+\Upu \sV(x)} &= \frac{\Big(\tP^{u^*}_\tau\big(V^{n-1}-V_0^*\big)\Big)(x) +\widetilde c_{u^*}-\delta^{-1}\Lambda^*_0}{1+\Upu \sV(x)}\\
&\leq \vvvert{\tP^{u^*}_\tau\Phi^{n-1}}_{\Upu,\sV}\,. \end{align*}
To get the second line, we use the definition of $\Phi^{n-1}$. Again, using Lemma~\ref{lem-compare} and the sub-optimality of the constant $- \tau\big( \delta V^{n-1}(0)-\Lambda^*_0\big) +\widetilde c_{u^*}$ with respect to the infimum in Lemma~\ref{lem-compare}, 
we get 
\begin{align}
\vvvert{\Phi^n}_{\Upu,\sV}\leq \vvvert{\tP^{u^*}_\tau\Phi^{n-1}}_{\Upu,\sV}\,.
\end{align}
From here, using Proposition~\ref{prop-contract}, in conjunction with Corollaries~\ref{cor-lyap-drift} and~\ref{cor-minor}, we obtain the desired contraction,  \emph{i.e.,} 
\begin{align}\label{eq-sk-1}
\vvvert{\Phi^n}_{\Upu,\sV}\leq  \kappa \vvvert{\Phi^{n-1}}_{\Upu,\sV}\,.
\end{align}
In terms of $\Phi^n$, ~\eqref{eq-sketch-1} becomes
\begin{align}\label{eq-sketch-2}
\Phi^n(x)= -\tau\delta \Phi^{n-1}(0) + \big(\tP^{u^*}_\tau\Phi^{n-1}\big)(x)\,.
\end{align}
From the repeated iteration of~\eqref{eq-sk-1}, we have
$$\vvvert{\Phi^n}_{\Upu,\sV}\leq (\kappa)^n  \vvvert{\Phi^0}_{\Upu,\sV}\,.$$
Using the definition of $\vvvert{\cdot}_{\Upu,\sV}$ and the sub-optimality of the pair $(x,0)$ with respect to the supremum in the definition of $\vvvert{\cdot}_{\Upu,\sV}$, we get
 \begin{align*}
\Phi^{n-1}(0) - &(\kappa)^{n-1}\Big(2+\Upu \sV(x)+\Upu \sV(0)\Big)\vvvert{\Phi^0}_{\Upu,\sV}\\
&\leq \Phi^{n-1}(x) \leq \Phi^{n-1}(0) + (\kappa)^{n-1}\Big(2+\Upu \sV(x)+\Upu \sV(0)\Big)\vvvert{\Phi^0}_{\Upu,\sV}\,.
 \end{align*}
 
 Applying $\tP^{u^*}_\tau$ (evaluated at $x$) to the above display and substituting~\eqref{eq-sketch-2} into the resulting equation gives us
 \begin{align*}
\Phi^{n-1}(0) - &(\kappa)^{n-1}\Big(\tP_\tau^{u^*}\big(2+\Upu \sV(x)+\Upu \sV(0)\big)\Big)(x)\vvvert{\Phi^0}_{\Upu,\sV}\\
&\leq \big(\tP_\tau^{u^*}\Phi^{n-1}\big)(x) \leq \Phi^{n-1}(0) + (\kappa)^{n-1}\Big(\tP_\tau^{u^*}\big(2+\Upu \sV(x)+\Upu \sV(0)\big)\Big)(x)\vvvert{\Phi^0}_{\Upu,\sV}\\
\implies \Phi^{n-1}(0) - &(\kappa)^{n-1}\Big(\tP_\tau^{u^*}\big(2+\Upu \sV(x)+\Upu \sV(0)\big)\Big)(x)\vvvert{\Phi^0}_{\Upu,\sV}\\
&\leq \Phi^n(x)+\tau\delta \Phi^{n-1}(0) \leq \Phi^{n-1}(0) + (\kappa)^{n-1}\Big(\tP_\tau^{u^*}\big(2+\Upu \sV(x)+\Upu \sV(0)\big)\Big)(x)\vvvert{\Phi^0}_{\Upu,\sV}\,.
 \end{align*} 
Evaluating the second  line for $x=0$, bounding the terms involving $\tP^{u^*}_\tau$ by $C_0$ and following simple algebraic computation, we obtain
\begin{align} 
|\Phi^n(0)|\leq (1-\delta\tau)|\Phi^{n-1}(0)|+ C_0 (\kappa)^n \vvvert{\Phi^0}_{\Upu,\sV}\,.
\end{align}
This proves~\eqref{eq-sketch-lambda}.

The extension of  the above argument to the controlled case (where we relax the assumption that $\bU$ is a singleton set), involves an appropriate choice of Markov controls for minimization operations involved in~\eqref{eq-hjb-0} and~\eqref{eq-rvi-mod-0}.

\subsection{Proof of Theorem~\ref{thm-rvi-0}} \label{sec-proof-1}
The theorem follows from Lemmas~\ref{lem-V-conv} and~\ref{lem-L-conv} that are stated and proved in what follows. Below, we establish a few auxiliary results used in the proof of Lemmas~\ref{lem-V-conv} and~\ref{lem-L-conv}. 

\begin{lemma}\label{lem-simple-bound} Suppose Assumptions~\ref{a-regularity} and~\ref{a-main} hold. Then, the following hold:
\begin{enumerate}
\item [(i)] For any $U\in \Uadm$ and $t>0$, we have 
$$ \E^U_x\Big[\int_0^t r(X_s,U_s)\D s- \Lambda^*_0t\Big]\geq V^*_0(x)-\E^U_x\big[V^*_0(X_t)\big]\,.$$
\item [(ii)] For $t>0$ and $v^*\in \Usm$ that satisfies 
\begin{align*} \min_{u\in \bU}\big[\Lg^uV^*_0(x) +r(x,u)\big]&= \Lg^{v^*(x)}V^*_0(x) +r(x,v^*(x)),\quad \text{for a.e. $x\in \RR^d$}, \end{align*}
we have 
\[
 \E^{v^*}_x\Big[\int_0^t r\big(X_s,v^*(X_s)\big)\D s- \Lambda^*_0t\Big]= V^*_0(x)-\E^{v^*}_x\big[V^*_0(X_t)\big]\,.\] 
\end{enumerate}
\end{lemma}
\begin{proof} Fix $U\in \Uadm$. Applying It\^o's formula to $V^*_0(X_t)$ and using~\eqref{eq-hjb-0}, we obtain
\begin{align}
\E_x^U\big[V^*_0(X_t)\big]-V^*_0(x)\geq  \E_x^U\Big[\int_0^t \big(r(X_t,U_t)-\Lambda^*_0\big)\D t\Big]\,.
\end{align}
Rearranging the above display proves the lemma. Similarly, part (ii) can be proved.
\end{proof}
 
In the following, we apply Proposition~\ref{prop-contract} in our case.  To do this, we introduce the following notation: for any $v\in \Um$, we define an associated Markov kernel $\tP_\tau^{v}:\RR^d\times \cB(\RR^d)\rightarrow [0,1]$ as 
\begin{align}\label{def-markov-ker-x}
\tP_\tau^{v}(x,A)= \E_x^v\big[\Ind_{A}(X_\tau)\big]\,.
\end{align}
In other words, $(\tP_\tau^vf)(x)$ for a Borel measurable function $f:\RR^d\rightarrow\RR$ is the expectation of $\E_x^v[f(X_\tau)]$, where $X$ is the solution to~\eqref{eqn-X} with $U_t=v(t,X_t)$.  
$\tP^{v}_\tau$ is clearly a Markov kernel.  

Next we show that $\tP_\tau^v$ satisfies~\eqref{eq-drift} of Proposition~\ref{prop-contract}, for every $v\in \Um$. 
\begin{corollary}\label{cor-lyap-drift} Suppose Assumption~\ref{a-main} holds. Then, for $0<\eta=\eta(\tau,\lambda_0)\doteq  e^{-\lambda_0\tau}<1$ and $K=K(\tau,\lambda_0,\lambda_1)\doteq   \lambda_1 \lambda_0^{-1} (1-e^{-\lambda_0 \tau})$, 
\begin{align*}
\big(\tP^v_\tau \sV\big)(x) \leq \eta \sV(x)+ K, \text{ for every $x\in \RR^d$ and $v\in\Um$.}
\end{align*}

\end{corollary}
\begin{proof}
Fix $v\in \Um$. Applying It\^o's formula to $e^{\lambda_0 t}\sV(X_t)$ up to $t=\tau$ with $X_t$ under $v\in \Um$, we obtain 
\begin{align*}
e^{\lambda_0 \tau}\E_x^v\Big[\sV(X_\tau)\Big]&= \sV(x) + \E_x^v\Big[\int_0^\tau \Big( e^{\lambda_0 s}\Lg^v \sV(X_s)  + \lambda_0 e^{\lambda_0 s}\sV(X_s) \Big) \D s\Big]\\
&\leq  \sV(x) + \E_x^v\Big[\int_0^\tau \lambda_1e^{\lambda_0s}\D s\Big]\\
&\leq \sV(x) + \frac{\lambda_1 (e^{\lambda_0 \tau}-1)}{\lambda_0}\,.
\end{align*}
To get the second line, we use Assumption~\ref{a-main}. From the definition of  $\tP^v_\tau$, we obtain the desired result. 
\end{proof}

\begin{corollary}\label{cor-minor} Suppose Assumption~\ref{a-regularity} holds and $v\in \Um$. Define $\sV_R\doteq \{x\in \RR^d: \sV(x)\leq R\}$. Then, for every  $R>0$ and a compact set $K\subset \RR^d$, there exists a constant $\alpha_{R,\tau}>0$  such that  $0< \alpha_{R,\tau}m(K)  <1$  and
\begin{align} \label{eq-p-minor}\inf_{x\in \sV_R} \tP_\tau ^v(x,A)\geq \alpha_{R,\tau}m(K) \nu_K(A), \text{ for every Borel set $A\subset \RR^d$}\end{align}
where $\nu_K(\cdot)\doteq \frac{m(\cdot \cap K)}{m(K)}$ and $m$ is the Lebesgue measure on $\RR^d$. Moreover, the above constants are independent of $v\in \Um$ and $\tP^v_\tau$ satisfies~\eqref{eq-minor} of Proposition~\ref{prop-contract}.
\end{corollary}

\begin{proof}
 To begin with, we note that from \cite[Theorem 1.2]{menozzi2021}, for every $v\in \Um$, there exists  a function $p^v_\tau (x,y)$, continuous in $x,y\in \RR^d$, such that  
$$ \tP^v_\tau(x,A)= \int_{A} p^v_\tau(x,y)\D y, \text{ for every Borel $A\subset \RR^d$}\,.$$
Moreover, \cite[Theorem 1.2]{menozzi2021} also helps us conclude that for $R>0$, there exists a constant $\alpha_{R,\tau}>0$ such that 
$$ \inf_{x\in \sV_R} p_\tau^v(x,y)\geq \alpha_{R,\tau}\,.$$
Now choose any compact set $K\subset \RR^d$ and any Borel set $A\subset K$. From the above display, we can see that for $R>0$,
\begin{align*} \inf_{x\in \sV_R} \tP_\tau ^v(x,A) &= \int_A p^v_\tau(x,y)\D y\geq \alpha_{R,\tau}  m(A)= \alpha_{R,\tau}m(K)\frac{m(A)}{m(K)}\,. 
 \end{align*}
From the definition of $\nu_K$, we obtain~\eqref{eq-p-minor}. The rest of the proof follows from the independence of $\alpha_{R,\tau}$ on $v\in \Um$. 
\end{proof}

\subsubsection{Convergence of $\Phi^n$ to $0$ in $\vvvert{\cdot}_{\Upu,\sV}$}
\begin{lemma}\label{lem-V-conv} Suppose Assumptions~\ref{a-regularity} and~\ref{a-main} hold. Let $\kappa$ and $\Upu$   be as defined in~\eqref{def-kappa} and~\eqref{def-upu}, respectively. Then, for every $n\geq 1$, 
$$ \vvvert{\Phi^n}_{\Upu,\sV} \leq (\kappa)^n \vvvert{\Phi^0}_{\Upu,\sV}\,.$$
Moreover, for any  real-valued sequence $\{\phi^n:n\geq 0\}$, we have
\begin{align}\label{eq-V-conv-f}
\vvvert{\Phi^n+ \phi^n}_{\Upu,\sV} \leq (\kappa)^n \vvvert{\Phi^0+\phi^0}_{\Upu,\sV}\,.
\end{align}
\end{lemma}
\begin{proof} Fix $x\in \RR^d$, $n\geq 1$ and recall $\Phi^n(x)= V^n(x)-V^*_0(x)-\delta^{-1}\Lambda^*_0$.  Then, we choose $v\in \Um$ such that it is the minimizer of~\eqref{eq-rvi-mod-0} - we remark that this chosen $v$ depends on $n$, but we suppress its dependence on $n$.  Applying It\^o's formula to $V(\tau-t,X_t)$ up to $t=\tau$ with $X_t$ under $\bar v(t,\cdot)\doteq v(\tau-t,\cdot)$, we have
\begin{align}\label{eq-markov-ker-0}
V^n(x)= \E_{x}^{\bar v}\Big[\int_0^\tau \Big(r\big(X_s,\bar v(s,X_s)\big)-\delta V^{n-1}(0)\Big) \D s+ V^{n-1}(X_\tau)\Big]\,.
\end{align} 
Recall that $V^n(x)=V(\tau,x)$ with $V(0,x)=V^{n-1}(x)$.
Similarly, applying It\^o's formula to $V^*(X_t)$ up to $t=\tau$ with $X_t$ under $\bar v(t,\cdot)\doteq v(\tau-t,\cdot)$, we have
\begin{align*}
V^*(x)\leq \E_{x}^{\bar v}\Big[\int_0^\tau \Big(r\big(X_s,\bar v(s,X_s)\big)-\Lambda^*_0\Big) \D s+ V^*(X_\tau)\Big]\,.
\end{align*} 
Recall that $V^*$ satisfies~\eqref{eq-hjb-0}.
From the above two displays and the definition of $\Phi^n$, we have
\begin{align}\nonumber
\Phi^n(x)&\leq  -\E_{x}^{\bar v}\Big[\int_0^\tau \big(\delta V^{n-1}(0)-\Lambda^*_0\big) \D s+ V^*(X_\tau)-V^{n-1}(X_\tau)\Big]-\delta^{-1}\Lambda^*_0\\\nonumber
&\leq -\delta \tau \Phi^{n-1}(0)+ \E_{x}^{\bar v}\Big[ \Phi^{n-1}(X_\tau)\Big] \\\label{eq-V-conv-11}
&=- \delta \tau \Phi^{n-1}(0)+ \big(\tP^{\bar v}_\tau \Phi^{n-1}\big)(x)\,.\end{align}
The last line above follows from the definition of $\tP_\tau^{\bar v}$. Let   
\begin{align}\label{eq-c-v}c_{\bar v}\doteq \inf_{y\in \RR^d}\Big(\vvvert{\tP^{\bar v}_\tau \Phi^{n-1}}_{\Upu,\sV}\big(1+\Upu\sV(y)\big)-\tP^{\bar v}_\tau \Phi^{n-1}(y)\Big).\end{align}  Adding $c_{\bar v}$ and dividing by $1+\Upu\sV(x)$ on both sides of the above display, we have 
\begin{align*}
\frac{\Phi^n(x)+\widetilde c_{\bar v}}{1+\Upu\sV(x)}\leq  \frac{\big(\tP^{\bar v}_\tau \Phi^{n-1}\big)(x)+c_{\bar v}}{1+\Upu\sV(x)}&\leq \sup_{y\in \RR^d}\frac{\big|\big(\tP^{\bar v}_\tau \Phi^{n-1}\big)(y)+c_{\bar v}\big|}{1+\Upu\sV(y)}\\
&= \|\tP^{\bar v}_\tau \Phi^{n-1}+c_{\bar v}\|_{\Upu,\sV}\\
&= \vvvert{\tP^{\bar v}_\tau \Phi^{n-1}}_{\Upu,\sV}\,.
\end{align*}
Here, $\widetilde c_{\bar v}= c_{\bar v}+ \delta \tau  \Phi^{n-1}(0)$. In the above, to get the second line, we take supremum over $x$; to get the third line, we use the definition of $\|\cdot\|_{\Upu,\sV}$ and to get the last line, we use Lemma~\ref{lem-compare}. Therefore, as $x$ is arbitrary, we have shown that 
\begin{align}\label{eq-V-conv-1}
\inf_{x\in \RR^d}\frac{\Phi^n(x)+\widetilde c_{\bar v}}{1+\Upu\sV(x)}\leq \sup_{x\in \RR^d}\frac{\Phi^n(x)+\widetilde c_{\bar v}}{1+\Upu\sV(x)}\leq \vvvert{ \tP^{\bar v}_\tau \Phi^{n-1}}_{\Upu,\sV}\,.
\end{align}

We now proceed to obtain a reverse inequality analogous to the above display. The proof follows exactly along the same lines as above. To that end, we choose $v^*\in \Usm$ such that it is the minimizer of~\eqref{eq-hjb-0}. From It\^o's formula to $V(\tau-t,X_t)$ up to $t=\tau$ with $X_t$ under $v^*$, we have
\begin{align}
V^n(x)\leq  \E_{x}^{ v^*}\Big[\int_0^\tau \Big(r\big(X_s,v^*(X_s)\big)-\delta V^{n-1}(0)\Big) \D s+ V^{n-1}(X_\tau)\Big]\,.
\end{align} 
Similarly, applying It\^o's formula to $V^*(X_t)$ up to $t=\tau$ with $X_t$ under $v^*$, we have
\begin{align}
V^*(x)= \E_{x}^{v^*}\Big[\int_0^\tau \Big(r\big(X_s,v^*(X_s)\big)-\Lambda^*_0\Big) \D s+ V^*(X_\tau)\Big]\,.
\end{align} 
From the above two displays and the definition of $\phi^n$, we have
\begin{align}\label{eq-V-conv-12}
\Phi^n(x)&\geq -  \delta \tau \Phi^{n-1}(0)+\big(\tP^{v^*}_\tau \Phi^{n-1}\big)(x)\,.\end{align}
Analogous to the previous case, let  \[c_{v^*}\doteq \inf_{y\in \RR^d}\Big(\vvvert{\tP^{ v^*}_\tau \Phi^{n-1}}_{\Upu,\sV}\big(1+\Upu\sV(y)\big)-\big(\tP^{ v^*}_\tau \Phi^{n-1}\big)(y)\Big).\] Adding $c_{v^*}$ and dividing by $1+\Upu\sV(x)$ on both sides of the above display, we have 
\begin{align*}
\frac{\Phi^n(x)+\widetilde c_{v^*}}{1+\Upu\sV(x)}\geq  \frac{\big(\tP^{v^*}_\tau \Phi^{n-1}\big)(x)+c_{ v^*}}{1+\Upu\sV(x)}&\geq -\sup_{y\in \RR^d}\frac{\big|\big(\tP^{v^*}_\tau \Phi^{n-1}\big)(y)+c_{v^*}\big|}{1+\Upu\sV(y)}\\
&= -\|\tP^{v^*}_\tau \Phi^{n-1}+c_{ v^*}\|_{\Upu,\sV}\\
&=- \vvvert{ \tP^{v^*}_\tau \Phi^{n-1}}_{\Upu,\sV}\,.
\end{align*}
Here, $\widetilde c_{v^*}= c_{ v^*}+ \delta \tau \Phi^{n-1}(0)$. In the above, as earlier, to get the second inequaltiy, we take supremum over $x$ and the definition of $\Phi^n$; to get the first equality, we use the definition of $\|\cdot\|_{\Upu,\sV}$ and to get the last line, we use Lemma~\ref{lem-compare}. Therefore, as $x$ is arbitrary, we have shown that 
\begin{align}\label{eq-V-conv-2}
\sup_{x\in \RR^d}\frac{\Phi^n(x)+\widetilde c_{v^*}}{1+\Upu\sV(x)}\geq \inf_{x\in \RR^d}\frac{\Phi^n(x)+\widetilde c_{v^*}}{1+\Upu\sV(x)}\geq -\vvvert{\tP^{v^*}_\tau \Phi^{n-1}}_{\Upu,\sV}\,.
\end{align}

We next show that 
\begin{align}\label{eq-V-conv-3}
\|\Phi^n\|_{\Upu,\sV}\leq \max\Big\{\vvvert{ \tP^{\bar v}_\tau \Phi^{n-1}}_{\Upu,\sV},  \vvvert{\tP^{ v^*}_\tau \Phi^{n-1}}_{\Upu,\sV} \Big\}\,.
\end{align}
To do this, we first observe that for any $f:\RR^d\rightarrow \RR$,
\begin{align}\label{eq-simp-obs} \|f\|_{\Upu,\sV}=\max\Big\{\sup_{x\in \RR^d} \frac{f(x)}{1+\Upu\sV(x)}, -\inf_{x\in \RR^d} \frac{f(x)}{1+\Upu\sV(x)}\Big\}\,.\end{align}
Also, we first suppose that $\widetilde c_{\bar v}\geq \widetilde c_{v^*}$. Then, we have 
\begin{align}\nonumber
\|\Phi^n +\widetilde c_{v^*}\|_{\Upu, \sV}&= \max\Big\{ \sup_{x\in \RR^d}\frac{\Phi^n(x)+\widetilde c_{v^*}}{1+\Upu\sV(x)}, -\inf_{x\in \RR^d} \frac{\Phi^n(x)+\widetilde c_{v^*}}{1+\Upu\sV(x)}\Big\}\\\nonumber
& \leq  \max\Big\{ \sup_{x\in \RR^d}\frac{\Phi^n(x)+\widetilde c_{\bar v}}{1+\Upu\sV(x)}, -\inf_{x\in \RR^d} \frac{\Phi^n(x)+\widetilde c_{v^*}}{1+\Upu\sV(x)}\Big\}\\\label{eq-V-conv-4}
&\leq   \max\Big\{ \vvvert{ \tP^{\bar v}_\tau \Phi^{n-1}}_{\Upu,\sV} , \vvvert{ \tP^{ v^*}_\tau \Phi^{n-1}}_{\Upu,\sV}\Big\}\,.
\end{align}
In the above, to get the first line, we use~\eqref{eq-simp-obs}; to get the second line, we use the fact that $\widetilde c_{\bar v}\geq \widetilde c_{v^*}$, and to get the first term inside the maximum in the third line, we use~\eqref{eq-V-conv-1} and to get the second term inside the maximum, we use~\eqref{eq-V-conv-2}.

Now suppose $\widetilde c_{\bar v}< \widetilde c_{v^*}$. Then, computations analogous to those leading up to~\eqref{eq-V-conv-4} give us
\begin{align}\nonumber
\|\Phi^n +\widetilde c_{\bar v}\|_{\Upu, \sV}&= \max\Big\{ \sup_{x\in \RR^d}\frac{\Phi^n(x)+\widetilde c_{\bar v}}{1+\Upu\sV(x)}, -\inf_{x\in \RR^d} \frac{\Phi^n(x)+\widetilde c_{\bar v}}{1+\Upu\sV(x)}\Big\}\\\nonumber
& \leq  \max\Big\{ \sup_{x\in \RR^d}\frac{\Phi^n(x)+\widetilde c_{\bar v}}{1+\Upu\sV(x)}, -\inf_{x\in \RR^d} \frac{\Phi^n(x)+\widetilde c_{ v^*}}{1+\Upu\sV(x)}\Big\}\\\label{eq-V-conv-5}
&\leq   \max\Big\{ \vvvert{ \tP^{\bar v}_\tau \Phi^{n-1}}_{\Upu,\sV} , \vvvert{ \tP^{ v^*}_\tau \Phi^{n-1}}_{\Upu,\sV}\Big\}\,.
\end{align}
From Lemma~\ref{lem-compare}, we know that $\vvvert{f}_{\Upu,\sV}\leq \|f+c\|_{\Upu,\sV}$, for every $c\in \RR$. Therefore, combining this with~\eqref{eq-V-conv-4} and~\eqref{eq-V-conv-5} gives us
\begin{align}\label{eq-V-conv-6}
\vvvert{\Phi^n}_{\Upu,\sV}\leq  \max\Big\{ \vvvert{ \tP^{\bar v}_\tau \Phi^{n-1}}_{\Upu,\sV} , \vvvert{ \tP^{ v^*}_\tau \Phi^{n-1}}_{\Upu,\sV}\Big\}\,.
\end{align}
From Corollaries~\ref{cor-lyap-drift} and~\ref{cor-minor}, Proposition~\ref{prop-contract} and with  $\kappa$ and $\Upu$  as defined in~\eqref{def-kappa} and~\eqref{def-upu}, respectively, we have
$$ \vvvert{ \tP^{v}_\tau \Phi^{n-1}}_{\Upu,\sV}\leq \kappa \vvvert{\Phi^{n-1}}_{\Upu,\sV},$$
for every $v\in\Um$.  This gives us 
$$ \vvvert{\Phi^n}_{\Upu,\sV}\leq \kappa  \vvvert{  \Phi^{n-1}}_{\Upu,\sV}\,.$$

From here,  the first part of the  lemma follows immediately. 
Finally, to get~\eqref{eq-V-conv-f}, we use the fact that for any function $g$ and  constant $c$,  $\vvvert{g+c}_{\Upu,\sV}=\vvvert{g}_{\Upu,\sV}$. This completes the proof.
\end{proof}

\subsubsection{Convergence of $\Phi^n(0)$ to $0$}

It is important to note that Lemma~\ref{lem-V-conv} does not necessarily imply that $\Phi^n(0)$ converges to zero, as $\vvvert{\cdot}_{\Upu,\sV}$ is only a semi-norm. From the definition of $\vvvert{\cdot}_{\Upu,\sV}$, we can conclude that 
\begin{align}\label{eq-est-v-v^*}
 \sup_{0\neq x\in \RR^d}\frac{|\Phi^n(x)- \Phi^n(0)|}{2+\Upu \sV(x)+\Upu \sV(0)}\leq \sup_{x\neq y} \frac{|\Phi^n(x)- \Phi^n(y)|}{2+\Upu \sV(x)+\Upu \sV(y)}=\vvvert{\Phi^n}_{\Upu,\sV}\leq \kappa^n \vvvert{\Phi^0}_{\Upu,\sV}\,.
\end{align}
To obtain the first inequality, we use the sub-optimality of $(x,y)=(x,0)$ (with $x\neq 0$) with respect to the supremum over $(x,y)$ such that $x\neq y$. From the above, we can also conclude that  $\Phi^n$ converges to zero, at an exponential rate, if we can show that $\Phi^n(0)$   converges to zero, at an exponential rate. This is the content of the next lemma. 
\begin{lemma}\label{lem-L-conv}
Suppose Assumptions~\ref{a-regularity} and~\ref{a-main} hold and that $0<\delta \tau<1$. Then,
with $\widetilde \kappa$ as defined in~\eqref{def-kappa}, we have
$$ | \Phi^n(0)|\leq (\widetilde \kappa)^n\bigg(| \Phi^0(0)| + \frac{C_0}{e\kappa_{\max} \log (\widetilde \kappa \kappa_{\max}^{-1})} \vvvert{\Phi^0}_{\Upu,\sV}\bigg), \text{ for every $n\geq 1$}\,.$$ Recall   $\kappa_{\max}$ from~\eqref{def-kappa} and recall that   $C_0\doteq 2+2\Upu\sV(0)+\Upu\lambda_1\tau$ with $\Upu$ as defined in~\eqref{def-upu}.
\end{lemma}
\begin{proof} 
 In terms of $\Phi^n$, rearranging~\eqref{eq-V-conv-11} and~\eqref{eq-V-conv-12}, we have 
\begin{align*}
\big(\tP^{ v^*}_\tau \Phi^{n-1}\big)(x)&\leq\Phi^n(x) + \delta \tau\Phi^{n-1}(0) \leq \big(\tP^{ \bar v}_\tau \Phi^{n-1}\big)(x)\,. 
\end{align*}
From~\eqref{eq-est-v-v^*}, for any $x\in \RR^d$ we have 
 \begin{align*}
\Phi^{n-1}(0) - &(\kappa)^{n-1}\Big(2+\Upu \sV(x)+\Upu \sV(0)\Big)\vvvert{\Phi^0}_{\Upu,\sV}\\
&\leq \Phi^{n-1}(x) \leq \Phi^{n-1}(0) + (\kappa)^{n-1}\Big(2+\Upu \sV(x)+\Upu \sV(0)\Big)\vvvert{\Phi^0}_{\Upu,\sV}\,.
 \end{align*}
Define $\bar \sV_\Upu(x)\doteq 2+\Upu \sV(x)+\Upu \sV(0)$. From the above two displays, we obtain
\begin{align*}
\Phi^{n-1}(0)-(\kappa)^{n-1}\vvvert{\Phi^0}_{\Upu,\sV}\E_0^{v^*} \big[\bar \sV_\Upu (X_\tau)\big]&\leq \Phi^{n}(0) + \delta\tau \Phi^{n-1}(0)  \\
& \leq \Phi^{n-1}(0)+(\kappa)^{n-1}\vvvert{\Phi^0}_{\Upu,\sV}\E_0^{ \bar  v} \big[\bar \sV_\Upu (X_\tau)\big]\,.
\end{align*}
Using Assumption~\ref{a-main} and It\^o's formula, we can conclude that  for any $v\in \Um$,
\begin{align*} \E_0^v\big[\bar \sV(X_\tau)\big] &= 2+\Upu \E_0^v\big[\sV(X_\tau)\big] +\Upu \sV(0) \\
&\leq  2+\Upu \big(\sV(0) +\lambda_1\tau\big) +\Upu \sV(0)= 2+2\Upu\sV(0)+\Upu\lambda_1\tau= C_0\,. \end{align*}
To summarize, we have shown that 
\begin{align*}
(1-\delta \tau)\Phi^{n-1}(0)- C_0 (\kappa)^{n-1} \vvvert{\Phi^0}_{\Upu,\sV}\leq \Phi^n(0)\leq (1-\delta \tau)\Phi^{n-1}(0)+ C_0(\kappa)^{n-1}\vvvert{\Phi^0}_{\Upu,\sV}\,.
\end{align*}
Therefore, iterating the above upper bound once gives us
\begin{align*}
\Phi^n(0)&\leq (1-\delta \tau)\Phi^{n-1}(0)+ C_0(\kappa)^{n-1}\vvvert{\Phi^0}_{\Upu,\sV}\\
&\leq (1-\delta \tau)\big((1-\delta \tau)\Phi^{n-2}(0)+ C_0(\kappa)^{n-2}\vvvert{\Phi^0}_{\Upu,\sV}\big)+ C_0(\kappa)^{n-1}\vvvert{\Phi^0}_{\Upu,\sV}\\
&\leq (1-\delta\tau)^2 \Phi^{n-2}(0) + 2 ( \kappa_{\max})^{n-1} C_0\vvvert{\Phi^0}_{\Upu,\sV}\,.
\end{align*}
From here, we can see that as we iterate up to $n-1$ times we get
$$ \Phi^n(0)\leq (1-\delta\tau)^n \Phi^0(0)+n(  \kappa_{\max})^{n-1}C_0\vvvert{\Phi^0}_{\Upu,\sV}\,.$$
Similarly, we get 
$$ \Phi^n(0)\geq (1-\delta\tau)^n \Phi^0(0)- n(  \kappa_{\max})^{n-1}C_0\vvvert{\Phi^0}_{\Upu,\sV}\,.$$
From the above two displays, we have
\begin{align*}
|\Phi^n(0)|&\leq (1-\delta\tau)^n |\Phi^0(0)|+n(  \kappa_{\max})^{n-1}C_0\vvvert{\Phi^0}_{\Upu,\sV}\\
&\leq (1-\delta\tau)^n |\Phi^0(0)|+n(  \frac{\kappa_{\max}}{\widetilde \kappa})^{n} (\widetilde \kappa)^{n} \frac{C_0}{\kappa_{\max}}\vvvert{\Phi^0}_{\Upu,\sV}\\
&\leq (1-\delta\tau)^n |\Phi^0(0)|+ (\widetilde \kappa)^{n} \frac{C_0}{e\kappa_{\max} \log (\widetilde \kappa \kappa_{\max}^{-1})}  \vvvert{\Phi^0}_{\Upu,\sV}\,. 
\end{align*}
To get the last line, we use the fact that maximum value of  $xa^x$ on $\RR$, for $0<a<1$ is $-(e\log a)^{-1}$. This completes the proof of the lemma as $(1-\delta\tau)<\widetilde\kappa $.
\end{proof}
\begin{proof}[Proof of Corollary~\ref{cor-norm-conv}]  
Fix $n\geq 1$. Let  $\kappa$ and $\Upu$   be as defined in ~\eqref{def-kappa} and~\eqref{def-upu}, respectively. Also, let  $c_n \doteq \inf_{x\in \RR^d} \Big(\vvvert{\Phi^n}_{\Upu,\sV}\big(1+\Upu\sV(x)\big) -\Phi^n(x)\Big)$.
Suppose $c_n\geq 0$. We obtain 
\begin{align*}
\|\Phi^n\|_{\Upu,\sV}
  & \leq \|\Phi^n+c_n\|_{\Upu,\sV}+  \|c_n\|_{\Upu,\sV} \\ 
  &\leq \|\Phi^n+c_n\|_{\Upu,\sV}+  c_n\\
&= \vvvert{\Phi^n}_{\Upu,\sV} +c_n\\
&\leq (\kappa)^n  \vvvert{\Phi^0}_{\Upu,\sV}  + \vvvert{\Phi^n}_{\Upu,\sV}\big(1+\Upu\sV(0)\big) -\Phi^n(0)\\
&\leq (\kappa)^n  \vvvert{\Phi^0}_{\Upu,\sV}  + \vvvert{\Phi^n}_{\Upu,\sV}\big(1+\Upu\sV(0)\big) +\big|\Phi^n(0)\big|\\
&\leq \big( 2+\Upu\sV(0)\big)(\kappa)^n \vvvert{\Phi^0}_{\Upu,\sV} + (\widetilde \kappa)^n\Big(| \Phi^0(0)| + \frac{C_0}{e\kappa_{\max} \log (\widetilde \kappa \kappa_{\max}^{-1})} \vvvert{\Phi^0}_{\Upu,\sV}\Big)\\
&\leq \Big( 2+\Upu\sV(0) +  \frac{C_0}{e\kappa_{\max} \log (\widetilde \kappa \kappa_{\max}^{-1})} \Big)(\kappa)^n \vvvert{\Phi^0}_{\Upu,\sV}  + (\widetilde \kappa)^n| \Phi^0(0)|\\
&\leq  \Big( 3+\Upu\sV(0) +  \frac{C_0}{e\kappa_{\max} \log (\widetilde \kappa \kappa_{\max}^{-1})} \Big)(\widetilde \kappa)^n \|\Phi^0\|_{\Upu,\sV}\\
&\leq \Big( 3+\Upu\sV(0) +  \frac{(2+2\Upu\sV(0)+\Upu\lambda_1\tau)}{e\kappa_{\max} \log (\widetilde \kappa \kappa_{\max}^{-1})}\Big)(\widetilde \kappa)^n \|\Phi^0\|_{\Upu,\sV} \\
&=C( \widetilde \kappa)^n \|\Phi^0\|_{\Upu,\sV}  \,.
\end{align*}
In the above, to get the second inequality, we use the definition of $\|\cdot\|_{\Upu,\sV}$, the non-negativity of $\sV$ and  the fact that $c_n$ is a constant; to get the first equality, we apply Lemma~\ref{lem-compare}; to get the third inequality, we use the sub-optimality of $0$ in the definition of $c_n$; to get the fourth inequality, we use the fact that $-z\leq |z|$, for $z\in \RR$; to get the fifth inequality, we use Lemmas~\ref{lem-V-conv} and~\ref{lem-L-conv}; to get the seventh inequality, we use the fact that $f(0)\leq \|f\|_{\Upu,\sV}$.

Similarly, we can argue that the above conclusion (in particular, the last inequality of the above display) still holds when $c_n<0$. This completes the proof of Theorem~\ref{thm-rvi-0}.
\end{proof}

\medskip

\section{Rate of Convergence of RVI Algorithm in ERSC Problems}

\subsection{Ergodic risk sensitive control problem} \label{sec-gamma}
In this section, we are   interested in optimizing the following criterion: for a fixed risk sensitive parameter $\gamma>0$,
\begin{equation*} 
J_\gamma(x, U) \doteq \limsup_{T\to\infty} \frac{1}{\gamma T}\log   \E_x^U\left[\exp\Big( \gamma\int_0^T r({X}_t, U_t )\D t\Big) \right]\,.
\end{equation*}
More precisely, the problem of interest is  
\begin{equation*}
{\Lambda}_\gamma^*({x}) \doteq  \inf_{U\in \Uadm}  J_\gamma({x}, U)\, \quad \text{ and } \quad \Lambda_\gamma^*\doteq \inf_{x\in \RR^d} \Lambda_\gamma(x)\,.
\end{equation*}
 In this case, we assume the following stronger version of the Foster-Lyapunov condition than the one in~\eqref{eq-erg-0}.
\begin{assumption}\label{a-main-gamma} There exist an inf-compact  $\cC^2(\RR^d) $ function $\widetilde \sV\geq 1$, an inf-compact $\cC(\RR^d)$ function $l\geq 1$ and  $\bar l>0$ such that for every $x\in \RR^d$ and $u\in \bU$,
\begin{align*}
\Lg^u \widetilde \sV(x)\leq \big(\bar l -l(x)\big)\widetilde \sV(x)\,.
\end{align*}
Moreover, $ l(\cdot)-\gamma \max_{u\in \bU} r(\cdot,u)$ is inf-compact. 
\end{assumption}

We now recall the well-known existing result (which is \cite[Theorem 4.1]{arapostathis2019strict}) concerning the associated multiplicative HJB equation and complete characterization of the optimal stationary Markov controls associated with ERSC problem defined above.
\begin{theorem}\label{thm-hjb-gamma} Suppose Assumptions~\ref{a-regularity} and~\ref{a-main-gamma} hold. Then, there exists a  function $\widetilde V^*_\gamma$ in $\cC^2(\RR^d)$, unique up to a multiplicative constant,  such that 
\begin{align}\label{eq-hjb-gamma}
\min_{u\in \bU}\big[\Lg^u\widetilde V^*_\gamma(x) +\gamma r(x,u)\widetilde V^*_\gamma(x)\big]=\gamma \Lambda^*_\gamma \widetilde V^*_\gamma(x)\,.
\end{align}
Moreover, 
\begin{enumerate}
\item[(i)]$v\in \Usm$ is optimal if and only if  
\begin{align}\label{eq-minimizer} \Lg^{v(x)} \widetilde V^*_\gamma(x)+\gamma r\big(x,v(x)\big)\widetilde V^*_\gamma(x)= \min_{u\in \bU}\big[\Lg^u\widetilde V^*_\gamma(x) +\gamma r(x,u)\widetilde V^*_\gamma(x)\big], \text{ a.e. $x\in \RR^d$.}\end{align}
\item [(ii)] For any $v\in \Usm$ that satisfies~\eqref{eq-minimizer}, we have
$$ \widetilde V^*_\gamma(x)= \E_x^v\Big[\exp\big(\gamma \int_0^{\widehat \tau_\frB}\big(r(X_t,v(X_t))-\Lambda^*_\gamma\big)\D t\big) \widetilde V^*_\gamma(X_{\widehat \tau_\frB})\Big], \quad\text{for $x\in \frB^c$,}$$
where $\frB$ is an open ball in $\RR^d$ and $\widehat \tau_\frB$ is the associated first hitting time of $\frB$.
\end{enumerate}
\end{theorem}
\begin{remark}
Since $\widetilde V^*_\gamma$ is unique up to a multiplicative constant, we can fix the value of $\widetilde V^*_\gamma$ at any chosen  point to a chosen value. In particular, we choose the state $0$ and the corresponding value to be $1$ for convenience,  \emph{i.e.,} $\widetilde V^*_\gamma (0)=1$.
\end{remark}

 \subsection{Numerical computation of $(\widetilde V^*_\gamma,\Lambda^*_\gamma)$}
  We first recall the existing  algorithm from \cite[Section 3.3]{arapostathis2020relative} below: the approximation to $(\widetilde V^*_\gamma,\Lambda^*_\gamma)$ is given by $(\widetilde V_\gamma(t,x)(\widetilde V_\gamma(t,0))^{-1},  \gamma ^{-1}\widetilde V_\gamma(t,0))$, for large $t$, where  $\widetilde V_\gamma(t,x)$ is the  function in $\cC^{1,2}(\RR_T^d)$ for every $T>0$,  
  that is the unique classical 
  solution to the PDE: 
\begin{align}\label{eq-rvi-org-gamma}
\partial_t\widetilde V_\gamma(t,x)= \min_{u\in \bU}\big[\Lg^u\widetilde V_\gamma(t,x) +\gamma r(x,u)\widetilde V_\gamma(t,x)\big] - \widetilde V_\gamma(t,0) \widetilde V_\gamma(t,x) \,, \quad \widetilde V_\gamma(0,x)=\wV_\gamma^0(x),
\end{align} 
where $\widetilde V_\gamma^0\in \cC^2(\RR^d)\cap \cC_{\Upu,\sV}(\RR^d)$. In Theorem 3.4 of \cite{arapostathis2020relative}, the convergence of $\widetilde V_\gamma(t,\cdot)$ as $t\to\infty$ is investigated, which is restated in  Proposition~\ref{prop-conv-gamma}  below.

\begin{proposition}\label{prop-conv-gamma}
Suppose Assumptions~\ref{a-regularity} and~\ref{a-main-gamma} hold. 
Also, suppose $\widetilde V_\gamma^0\in \cC^2(\RR^d)$ satisfies that $\widetilde V_\gamma^0$ is bounded from below away from $0$ and $\|\widetilde V^0_\gamma\|_{\Upu,\sV}<\infty$. 
 Then, there exists a unique function $\widetilde V_\gamma\in \cC_{\Upu,\sV}(\RR^d_T)\cap \cC^{1,2}(\RR^d_T)$, for all $T>0$ that satisfies~\eqref{eq-rvi-org-gamma}. Additionally, there exists a $\Gamma=\Gamma(\widetilde V^0_\gamma)>0$ such that as $t\to\infty$, $\widetilde V_\gamma(t,0)$ converges to $\gamma\Lambda^*_\gamma$  and $\widetilde V_\gamma(t,\cdot)$  converges to $\Gamma V^*_\gamma(\cdot)$ uniformly on compact sets of $\RR^d$.\end{proposition}

We obtain a natural `discrete' version of the existing risk-sensitive  RVI algorithm as follows: 
fix $\tau>0$ and choose $n^*\in \NN$, with an initialization  $ \widetilde V^0_\gamma \in \cC^2(\RR^d)\cap\cC_{\Upu,\sV}(\RR^d)$, for $0\leq t\leq n^* \tau $, solve the PDE in \eqref{eq-rvi-org-gamma}, and for $n\leq n^*$, set 
$\widetilde V_\gamma^n(\cdot)= \widetilde V_\gamma(n\tau,\cdot)\in \cC^2(\RR^d)\cap\cC_{\Upu,\sV}(\RR^d)$, and then output $\Big(\frac{\widetilde V_\gamma ^n}{\widetilde V^n_\gamma(0)},  \gamma^{-1} \widetilde V_\gamma ^n(0)\Big)$ as an approximation of $(\widetilde V_\gamma ^*,\Lambda_\gamma ^*)$.

\setcounter{algorithm}{0}

 In the following, we provide a modified `discrete' version of the risk-sensitive RVI algorithm: fix $\delta,\tau>0$.

 \setcounter{algorithm}{0}
\renewcommand{\thealgorithm}{RS}
\begin{algorithm}[H]
	\caption{Modified `discrete' version of RVI Algorithm}\label{alg-rvi-mod-gamma}
	\begin{enumerate}
	\item [(i)] Choose $\delta,\tau>0$ and $n^*\in \NN$.
		\item [(ii)] Input: initialize with $n=0$ and $\widetilde V^0_\gamma:\RR^d \rightarrow \RR^+$.
		\item [(iii)] Update: for $0\leq t\leq \tau$, solve the PDE below 
		\begin{align}\label{eq-rvi-mod-gamma}
\partial_t \widetilde V_\gamma (t,x)= \min_{u\in \bU}\big[\Lg^u\widetilde V_\gamma (t,x)+ \gamma r(x,u)\widetilde V_\gamma (t,x)\big]- \big( \widetilde V^n_\gamma (0)\big)^\delta \widetilde V_\gamma (t,x),\quad \widetilde V_\gamma (0,x)= \widetilde V_\gamma^n(x)\,.
\end{align}
		\item [(iv)] Set $n\leftarrow n+1$ and $\widetilde V_\gamma ^n= \widetilde V_\gamma (\tau,\cdot)$.
		\item [(v)] While $n\leq n^*$, repeat Steps (iii) and (iv). 
		\item [(vi)] Output: the approximate of $(\widetilde V_\gamma ^*,\Lambda_\gamma ^*)$ is $\Big(\frac{\widetilde V_\gamma ^n}{\widetilde V^n_\gamma(0)},  \gamma^{-1}\big(\widetilde V_\gamma ^n(0)\big)^\delta\Big)$.

	\end{enumerate}
\end{algorithm}
\begin{remark} Comments analogous to Remarks~\ref{rem-delta-tau} and~\ref{rem-compare} are applicable in the risk-sensitive case as well. Moreover, a discussion (analogous to that in the case of CEC problem) concerning the advantage of Algorithm~\ref{alg-rvi-mod-gamma} over the existing RVI algorithm, in terms of the convergence analysis can also be made in the risk-sensitive case.

\end{remark}

 \begin{remark}\label{rem-int-gamma}
  Before we proceed further, we again intuitively verify that $\widetilde V_\gamma^n$ defined above, whenever it converges, the limit of $\Big((\widetilde V^n_\gamma(0))^{-1}\widetilde V_\gamma^n,\gamma^{-1} \big(\widetilde V_\gamma^n(0)\big)^\delta \Big)$ is $(\widetilde V^*_\gamma,  \Lambda^*_\gamma)$. 
Suppose $\widetilde V_\gamma^n$ is  convergent, \emph{i.e.,} $\widetilde V_\gamma^n$ is independent of $n$, for large $n$. Then, from Theorem~\ref{thm-hjb-0}, it is clear that $ \gamma^{-1}\big(\widetilde V_\gamma^n(0)\big)^\delta$ approaches $\Lambda^*_\gamma$ and that $\widetilde V_\gamma^n$ approaches $\widetilde V^*_\gamma$, up to a multiplicative constant. But as $\widetilde V^*_\gamma(0)=1$ and $ \gamma^{-1}\big(\widetilde  V_\gamma^n(0)\big)^\delta$ approaches $\Lambda^*_\gamma$, we can intuitively conclude that $(\widetilde V^n_\gamma(0))^{-1}\widetilde V_\gamma^n(\cdot)$ approaches $\widetilde V^*_\gamma(\cdot)$. 
\end{remark}

The following result guarantees the well-posedness of~\eqref{eq-rvi-mod-gamma} in Step (iii) above.  The proof of this result follows the arguments in the proof of \cite[Lemma 4.1]{arapostathis2012relative} and hence, we omit it. Recall that $\RR_\tau^d= [0,\tau]\times \RR^d$.
\begin{proposition}\label{prop-ext-gamma}
Suppose Assumptions~\ref{a-regularity} and~\ref{a-main-gamma} hold. For  $n\in \NN$, there exists a unique solution $\widetilde V_\gamma\in  \cC^{1,2}(\RR^d_\tau)$ to~\eqref{eq-rvi-mod-gamma} with $\widetilde V_\gamma(0,x)=\widetilde V_\gamma^n(x)$.   Moreover, for every $n$,  if $\widetilde V_\gamma^n>0$, then $\widetilde V_\gamma^{n+1}>0$.
\end{proposition}
\begin{remark}
We remark the key differences in the existing risk-sensitive RVI algorithm and Algorithm~\ref{alg-rvi-mod-gamma}.  From the first part of Proposition~\ref{prop-conv-gamma}, we know that for $t>0$, there exists a map $\widetilde \sS_{\gamma,t}:\cC^2(\RR^d)\cap\cC_{\Upu,\sV}(\RR^d)\rightarrow \cC^{2}(\RR^d)\cap \cC_{\Upu,\sV}(\RR^d)$ that satisfies $\widetilde   V_\gamma (t,x)= (\widetilde \sS_{\gamma, t} V_\gamma ^0)(x)$.  
Just as in the risk neutral case, we do not have the knowledge of  a more explicit form of the map $\widetilde \sS_{\gamma, t}$ which makes the analysis of large $t$ behavior of~\eqref{eq-rvi-org-0} difficult.  

On the other hand, using Proposition~\ref{prop-ext-gamma} and arguing as we did in risk-neutral case, we can express $\widetilde V_\gamma^n$, the iterates of Algorithm~\ref{alg-rvi-mod-gamma} as follows: for $n\geq 1$, \[\widetilde V_\gamma^{n+1}= \sS_{\gamma,\tau} \widetilde V_\gamma^n(x)\]with  \begin{align*}
\sS_{\gamma,\tau} f(x) \doteq \min_{U\in \Uadm}\E^U_x\Big[\exp\Big(\int_0^\tau \big(r(X_s,U_s)-(f(0))^\delta\big)\D s\Big)  f(X_\tau)\Big]\,.
\end{align*}
The above form follows from a standard application of It\^o's formula. Denoting 
\[\sS_{\gamma,\tau}^nf(x)=\sS_{\gamma,\tau}^{n-1} \sS_{\gamma,\tau} f(x), \quad \text{ with} \quad \sS^0_{\gamma,\tau}f(x)=f(x),\]
 we can clearly see that 
  \[
\Big(\frac{V_\gamma^n(\cdot)}{ V_\gamma^n(0)}, \gamma^{-1} \big(V_\gamma^n(0)\big)^\delta\Big) =  \Big(\frac{(\sS^n_{\gamma,\tau} V_\gamma^0)(\cdot)}{ (\sS^n_{\gamma,\tau} V_\gamma^0)(0)},  \gamma^{-1} \big((\sS^n_{\gamma,\tau} V_\gamma^0)(0)\big)^\delta\Big)\,. 
  \]

In contrast to  the existing risk-sensitive RVI algorithm, we obtain the iterates of Algorithm~\ref{alg-rvi-mod-gamma} by repeated solving PDE~\eqref{eq-rvi-mod-gamma} in Step (iii) over the interval $[0,\tau]$.
\end{remark}

\subsection{Second main result} The following is the main result of this paper which is the convergence of $\Big(\frac{\widetilde V^n_\gamma }{\widetilde V^n_\gamma(0)},  \gamma^{-1} \big(\widetilde V^n_\gamma (0)\big)^\delta\Big)$ to $(\widetilde V_\gamma^*,\Lambda_\gamma^*)$.  Since $\widetilde V^n_\gamma,\widetilde V^*_\gamma>0$, we show that $\widetilde V^n_\gamma \big(\widetilde V^n_\gamma(0)\widetilde V^*_\gamma\big)^{-1} \to 1$ (or equivalently,  $ \log \widetilde V_\gamma^n - \log \big(\widetilde V^n_\gamma(0)\widetilde V_\gamma^*\big)\to 0$) in an appropriate norm and $\big(\widetilde V_\gamma^n(0)\big)^\delta$ approaches $\gamma\Lambda_\gamma^*$.  Although Assumption~\ref{a-main-gamma} guarantees the well-posedness  on the HJB equation (which is~\eqref{eq-hjb-gamma}) and characterization of stationary Markov controls in Theorem~\ref{thm-hjb-gamma}, it is not strong enough for the analysis in this paper. Hence, we  enforce the stronger  assumption below to obtain the rate of convergence of Algorithm~\ref{alg-rvi-mod-gamma}.

 \begin{assumption}\label{a-main-gamma-conv} There exist an inf-compact  $\cC^2(\RR^d) $ function $\widetilde \sV\geq 1$ and an inf-compact $\cC(\RR^d)$ function $l\geq 1$  such that the following conditions hold: 
\begin{enumerate}
\item [(i)] There exists some $\bar l>0$ such that for every $x\in \RR^d$ and $u\in \bU$,
\begin{align}\label{a-main-gamma-Lg}
\Lg^u \widetilde \sV(x)\leq \big(\bar l -l(x)\big)\widetilde \sV(x)\,.
\end{align}
\item [(ii)] For some   $0<\theta<1$, $\theta l(\cdot)-\gamma \max_{u\in \bU} r(\cdot,u)$ is inf-compact. 

\item [(iii)] There exists $\lambda >0$ such that $l(\cdot)-\lambda\log\widetilde \sV(\cdot)$ is inf-compact.
\end{enumerate}
\end{assumption}
It is clear that Assumption~\ref{a-main-gamma-conv} implies Assumption~\ref{a-main-gamma}.

\begin{remark}\label{rem-k-inf-c}
From Assumptions~\ref{a-main-gamma-conv}(ii)-(iii), it is clear that there exist constants $k_r$ and $k_\sV$ such that 
\begin{align}\label{eq-inf-c-1} \gamma\big( r(x,u)-\Lambda^*_\gamma\big)\leq \theta l(x)+ k_r, \quad \text{and}\quad \lambda\log \widetilde\sV(x)\leq l(x)+k_\sV\,.\end{align}
\end{remark}
\begin{remark}\label{rem-lyap-mult} From~\eqref{a-main-gamma-Lg} in Assumption~\ref{a-main-gamma-conv}(i), applying It\^o's formula to 
$ \exp\big(\int_0^t (l(X_s)-\bar l)\D s\big) \widetilde \sV(X_t)$ with $v\in \Um$ gives us 
\begin{align}\label{eq-lyap-ineq}
\E_x^v\Big[\exp\big(\int_0^t \big(l(X_s)-\bar l)\D s\big) \widetilde \sV(X_t)\Big]\leq \widetilde \sV(x),\quad \text{for $x\in \RR^d$ and $ t\geq 0$.}
\end{align}
\end{remark}
In the rest of the paper, we set $$ \sV\doteq \log \widetilde \sV,\quad V^n_\gamma\doteq \log \widetilde V^n_\gamma, \quad V^*_\gamma\doteq \log \widetilde V^*_\gamma, \quad \widehat \Lambda^*_\gamma \doteq  \delta^{-1} \log\Big(\gamma \Lambda^*_\gamma\Big)\quad \text{and}\quad \Phi^n_\gamma \doteq  V_\gamma^n-V_\gamma^*- \widehat \Lambda^*_\gamma\,.$$

 \begin{theorem}\label{thm-rvi-gamma}
Suppose Assumptions~\ref{a-regularity} and~\ref{a-main-gamma-conv} hold. Also, suppose that $V_\gamma^0>0$ is continuous and   for some  $0<\beta<1-\theta$, $ \|\Phi^0_\gamma\|_{\beta,\sV}< 1$.
Then, the following holds: for small enough $\beta>0$, $\tau$ large enough and for $\delta>0$ such that $$\max\Big\{\tau\delta \exp\big(\delta \widehat \Lambda^*_\gamma\big)\Big),   \tau\delta \exp\big(\delta \Phi^{0}_\gamma(0)\big)\Big\}<1,$$ there exists $0<\kappa_\gamma\doteq \kappa_\gamma\big(V^0_\gamma\big)<1$ such that  for 
 \begin{align}\label{eq-widekappa}\kappa_{\gamma,\max}\doteq\max\Big\{\kappa_\gamma,\Big(1- \tau\delta \exp\big(\delta \widehat \Lambda^*_\gamma\big)\Big),  \Big(1- \tau\delta \exp\big(\delta \Phi^0_\gamma(0)\big)\Big) \Big\}<\widetilde \kappa_\gamma<1,\end{align}
 we have 
\begin{align}\label{eq-V-gamma-semi-contract}
\vvvert{ \Phi^n_\gamma}_{\Upu,\sV}&\leq  (\kappa_\gamma)^n\vvvert{ \Phi^0_\gamma}_{\Upu,\sV}\\\label{eq-lambda-gamma-semi-contract}
\big|\Phi^n_\gamma(0)\big|&\leq  (\widetilde \kappa_{\gamma})^n\bigg(| \Phi^0_\gamma(0)| +\frac{C_{\gamma,0}}{e\kappa_{\gamma,\max} \log (\widetilde \kappa_\gamma \kappa_{\gamma,\max}^{-1})}  \vvvert{\Phi^0_\gamma}_{\Upu,\sV}\bigg),
\end{align}
for some $C_{\gamma,0}>0$.

\end{theorem}
The proof is deferred to Section~\ref{sec-thm-rvi-gamma}.
\begin{corollary}
Under the conditions of Theorem~\ref{thm-rvi-gamma} and with the constants there, we have 
\begin{align}\label{eq-contraction-2} \|\Phi_\gamma^n\|_{\Upu,\sV}\leq C_\gamma (\widetilde  \kappa_\gamma)^n\|\Phi_\gamma^0\|_{\Upu,\sV}, \end{align}
where, $C_\gamma\doteq  \Big( 3+\Upu\sV(0) +  \frac{C_{\gamma,0}}{e\kappa_{\max} \log (\widetilde \kappa \kappa_{\max}^{-1})}\Big)$.
\end{corollary}
We omit the proof of the corollary as it follows along the same lines as the proof of Corollary~\ref{cor-norm-conv}.
\begin{remark}\label{rem-gamma-constants}
For the sake of keeping the statement of the above theorem short, we only stated key qualitative properties of the constants $\beta$, $\kappa_\gamma$, $\widetilde \kappa_\gamma$ and $C_{\gamma,0}$. However, just as in Theorem~\ref{thm-rvi-0}, much more quantitative description is in fact, possible. Here, we briefly sketch how to obtain this quantitative description. The constants (possibly depending on $V^0_\gamma$) $(\varrho,K_{\sup})$ and  ($ \alpha,\alpha_{\inf})$ are obtained from  from Lemma~\ref{lem-exp-conv}, and Lemma~\ref{lem-exp-minor}, respectively. Then, $\beta=\beta (V^0_\gamma)>0$ and $0<\kappa_\gamma=\kappa_\gamma(V^0_\gamma)<1$ are given by 
\begin{align}\label{eq-prop-const}  \Upu\doteq   \frac{\alpha_{\inf}}{K_{\sup}} ,\quad\text{and}\quad \kappa_\gamma  \doteq  \Big(1-\frac{\alpha( \beta)}{2}\Big)\vee  \Big(\frac{2(1-\varrho( \beta))+\alpha ( \beta) (1+\varrho( \beta))}{2(1-\varrho( \beta))+ 2\alpha( \beta)}\Big)<1\,.\end{align}
 \end{remark}

\medskip

\section{Proof of Theorem~\ref{thm-rvi-gamma}}\label{sec-thm-rvi-gamma}

 The proof is a highly non-trivial adaptation of the proof of  Theorem~\ref{thm-rvi-0} and  is split into two parts:  proof of~\eqref{eq-V-gamma-semi-contract} - this is highly non-trivial,  and proof of~\eqref{eq-lambda-gamma-semi-contract} - this follows along the same lines with minor modifications as the proof of Lemma~\ref{lem-L-conv}, once~\eqref{eq-V-gamma-semi-contract} is established. Hence, we focus only the sketch of the proof of~\eqref{eq-V-gamma-semi-contract} which can be split into five key parts. 
 
\subsubsection*{(i) Additive versions of~\eqref{eq-hjb-gamma} and~\eqref{eq-rvi-mod-gamma} and the Lyapunov condition~\eqref{a-main-gamma-Lg}}  
To begin with, we apply an exponential transformation to \eqref{eq-hjb-gamma} and~\eqref{eq-rvi-mod-gamma}, and obtain their additive versions ~\eqref{eq-hjb-gamma-m} and~\eqref{eq-rvi-mod-gamma-m}. We also obtain the additive version~\eqref{eq-lyap-gamma-m} for the Lyapunov condition \eqref{a-main-gamma-Lg} in Assumption \ref{a-main-gamma-conv}. 
  In this form,~\eqref{eq-hjb-gamma-m} and~\eqref{eq-rvi-mod-gamma-m} resemble~\eqref{eq-hjb-0} and~\eqref{eq-rvi-mod-0}, respectively, with three key differences: (a) we have an additional maximization operation,     (b) we have $r(x,u)-\frac{1}{2}\|w\|^2$ in place of $r(x,u)$,  and (c) we have $\bar \Lg^{u,w}$ in place of $\Lg^{u}$; see~\eqref{eq-gen-ext} for its definition.

\subsubsection*{(ii) Using the diffusion associated with $\bar \Lg^{u,w}$ under Markov controls}
We then consider the  `extended' diffusion $Z$ associated with $\bar \Lg^{u,w}$, for two particular choices of Markov controls $(v=v(t,x),w=w(t,x))$. These particular choices of $(v,w)$ are (a) minimizer  associated with~\eqref{eq-rvi-mod-gamma-m} (denoted by $\widehat v$ with $n$ dependence suppressed) and maximizer  associated with~\eqref{eq-hjb-gamma-m} (denoted by $w^*$), and (b) minimizer  associated with~\eqref{eq-hjb-gamma-m} (denoted by $ v^*$)  and maximizer associated with~\eqref{eq-rvi-mod-gamma-m} for every $n$ (denoted by $\widehat w$ with $n$ dependence suppressed).    
Then an application of It\^o's formula to $V^*_\gamma$ and $ V^n_\gamma$ up to $t=\tau$ with $Z$ under $(v*,\widehat w)$, and subtracting two equations gives 
\begin{align*}
  V_\gamma^n(x)-V_\gamma^*(x)\leq  {\gamma\tau}{}\big(\Lambda_\gamma^*-\exp\big( \delta V_\gamma^n(0)\big)\big)  +  \E_x^{\widehat v, w^*} \Big[ V_\gamma^{n-1}(Z_\tau) -V_\gamma^*(Z_\tau)\Big]\,.
\end{align*}
Similarly,  an application of It\^o's formula to $V^*_\gamma$ and $V^n_\gamma$ up to $t=\tau$ with $Z$ under $(\widehat v, w^*)$, and subtracting two equations gives 
\begin{align*}
  V_\gamma^n(x)-V_\gamma^*(x)\geq  {\gamma\tau}{}\big(\Lambda_\gamma^*-\exp\big( \delta V_\gamma^n(0)\big)\big)  +  \E_x^{v^*, \widehat w} \Big[ V_\gamma^{n-1}(Z_\tau) -V_\gamma^*(Z_\tau)\Big]\,.
\end{align*}

\subsubsection*{(iii)  Difficulty in analyzing iterates $V^n_\gamma$}
Once we establish that the Markov kernel $\sQ^{v,w}(x,\D y)$ associated with  the extended diffusion $Z$ under the control pair  $(v^*,\widehat w)$ or $(\widehat v,w^*)$ satisfies~\eqref{eq-drift}  and~\eqref{eq-minor} of Proposition~\ref{prop-contract}, we can follow the arguments in the proof of Lemma~\ref{lem-V-conv} and derive the one-step contraction of $\vvvert{V_\gamma^n-V_\gamma^*}_{\beta,\sV}$, and consequently, also~\eqref{eq-V-gamma-semi-contract}. Unfortunately, this is a highly non-trivial task for the following reason: from  Lemma~\ref{lem-lyap-gamma-mod} (in the light of the argument in the proof of Corollary~\ref{cor-lyap-drift}),   to prove that $\sQ^{v,w}(x,\D y)$ (for $(v^*,\widehat w)$ or $(\widehat v,w^*)$) satisfies~\eqref{eq-drift}, we require the knowledge of  appropriate  bounds on  $\E_x^{v,w}\big[\int_0^\tau \|w(t,Z_t)\|^2 \D t\big]$ (for $(v^*,\widehat w)$ or $(\widehat v,w^*)$).  However, a priori these bounds are difficult to obtain. 

\subsubsection*{(iv) Resolution of the difficulty in Part (iii)}
To overcome this difficulty, we make use of the following crucial observation: from the definition of  $\vvvert{\cdot}_{\Upu,\sV}$, it is clear that for any real-valued sequence $\{\phi^n:n\geq 0\}$,  
$\vvvert{\Phi^n_\gamma}_{\Upu,\sV}=\vvvert{\Phi^n_\gamma+\phi^n}_{\Upu,\sV}$. This further means that for any $0<\kappa_\gamma<1$,
\begin{align}\label{eq-conv-fn} \vvvert{\Phi^n_\gamma+\phi^n}_{\Upu,\sV}\leq  \kappa_\gamma\vvvert{\Phi_\gamma^{n-1}+\phi^{n-1}}_{\Upu,\sV}\,,
\end{align}
if and only if $ \vvvert{\Phi^n_\gamma}_{\Upu,\sV}\leq  \kappa_\gamma\vvvert{\Phi_\gamma^{n-1}}_{\Upu,\sV}\,.$

We then define  an intermediate family of iterates $\overline V_\gamma^n$ such that $\overline V_\gamma^n-V_\gamma^n$ is constant (only depending on $n$ and possibly unbounded in $n$). These iterates are defined according to~\eqref{eq-widehat} and consequently, $\phi^n$ is given by~\eqref{eq-const-series}.   From the above observation, it suffices to prove that 
 $ \vvvert{\overline V_\gamma^n-V^*_\gamma}_{\Upu,\sV}\leq  \kappa_\gamma\vvvert{\overline V_\gamma^{n-1}-V^*_\gamma}_{\Upu,\sV}\,.$   The reason behind choosing another family of iterates $\overline V_\gamma^n$ is that it is now easier to compute desired upper and lower bounds of $\overline V_\gamma^n$. 
 We then use these  bounds to obtain the upper bound on $\E_x^{v,w}\big[\int_0^\tau \|w(t,Z_t)\|^2 \D t\big]$ (for $(v^*,\bar  w)$ or $(\bar  v,w^*)$). Now $\bar v$ and $\bar  w$ are the associated  minimizer and the maximizer  associated with~\eqref{eq-rvi-mod-gamma-int}, respectively.   To proceed from here, we obtain the bounds on $\overline V^n_\gamma$ under the assumption that  $\|\overline V_\gamma^{n-1}-V^*_\gamma\|_{\widetilde \beta,\sV}<1$, for some $0<\widetilde \beta<1-\theta$. As a consequence, all the bounds on $\E_x^{v,w}\big[\int_0^\tau \|w(t,Z_t)\|^2 \D t\big]$ (for $(v^*,\bar  w)$ or $(\bar  v,w^*)$) depend on $\|\overline V_\gamma^{n-1}-V^*_\gamma\|_{\widetilde \beta,\sV}$. We later choose $\widetilde \beta$ small enough (and denote that particular $\widetilde \beta$ by $\beta$) so that this condition holds for every $n$.

\subsubsection*{(v) Verifying the extended diffusion $Z$ satisfies~\eqref{eq-drift}  and~\eqref{eq-minor} of Proposition~\ref{prop-contract}.} 
We first conclude from the bounds above that the extended diffusion $Z$ under the control pair  $(v^*,\bar  w)$ or $(\bar  v,w^*)$
has a unique strong solution. The drift condition in \eqref{eq-drift} follows from the bounds established for  $\E_x^{v,w}\big[\int_0^\tau \|w(t,Z_t)\|^2 \D t\big]$ (for $(v^*,\bar  w)$ or $(\bar  v,w^*)$) mentioned above. 
The minorization condition in \eqref{eq-minor}  turns out to be much trickier than its counterpart in the case of CEC problem (Corollary~\ref{cor-minor}). 
This is because in the case of $\sQ^{v,w}(x,\D y)$ for $(v^*,\bar  w)$ or $(\bar  v,w^*)$ (or equivalently, $Z$ under   $(v^*,\bar  w)$ or $(\bar  v,w^*)$), to use an existing result \cite[Theorem 1.2]{menozzi2021}, we need to ensure that $b(x,u)+ \Sigma(x)w(t,x)$ has at most linear growth in $x$ (uniformly in $u$), which is a priori not known due to the unavailability of   relevant bounds on  $w(t,x)$. Therefore, we proceed with the proof differently. Instead of treating $Z$ as a diffusion being driven by Brownian motion $W$ in the original probability measure (where $W$ is the Brownian motion), we treat $Z$ as a diffusion driven by $W+\int_0^\cdot w(t,Z_t)\D t$ whose law is
the  same as the Brownian motion  in a new probability measure (given by the Girsanov's theorem).  
This now  helps us to use the existing result which is \cite[Theorem 1.2]{menozzi2021}. We subsequently prove that $\sQ^{v,w}(x,\D y)$ (for $(v^*,\bar  w)$ or $(\bar  v,w^*)$) satisfies~\eqref{eq-minor}.

  From here, following the arguments in the proof of Lemma~\ref{lem-V-conv}, we can derive the one-step contraction of $\vvvert{\overline V_\gamma^n-V_\gamma^*}_{\beta,\sV}$ with $\beta<\beta^*$. Here,  $\beta^*$ and the one-step contraction coefficient depend on constants from Lemmas~\ref{lem-exp-conv} and~\ref{lem-exp-minor}, 
  hence also   on $\|\overline V_\gamma^{n-1}-V^*_\gamma\|_{\widetilde \beta,\sV}$ through $\E_x^{v,w}\big[\int_0^\tau \|w(t,Z_t)\|^2 \D t\big]$ (for $(v^*,\bar  w)$ or $(\bar  v,w^*)$). This is the most important distinction from the  CEC case  - recall that in the CEC case, the analogous constants are independent of the $n$th iterate.  As a consequence, in order for the one-step contraction coefficient to be independent of the iterates, it is necessary to show that  $\|\overline V_\gamma^{n}-V^*_\gamma\|_{ \widetilde\beta,\sV}<1$, for the same $\widetilde\beta$. This follows from observing that the constants  obtained in Lemmas~\ref{lem-exp-conv} and~\ref{lem-exp-minor} are appropriately uniformly bounded, as $\widetilde \beta \to 0$.  This implies that $\beta^*$ is uniformly bounded away from zero and hence, choosing $\widetilde \beta$ small enough ensures that $\|\overline V_\gamma^{n}-V^*_\gamma\|_{ \widetilde\beta,\sV}<1$, for every $n$. This in turn means that the contraction coefficient is independent of the iterates and~\eqref{eq-V-gamma-semi-contract} follows.

\subsection{Additive versions of the HJB equations~\eqref{eq-hjb-gamma} and~\eqref{eq-rvi-mod-gamma}}\label{sec-add-eq-hjb-gamma}
 In what follows, we consider the logarithm of $\widetilde V^*_\gamma$ and the corresponding PDE it satisfies. This PDE is an additive version of~\eqref{eq-hjb-gamma}.   To begin with, we introduce a new generator that we encounter extensively from hereon. For $u\in \bU$ and $w\in \RR^d$, 
\begin{align}\label{eq-gen-ext} \bar\Lg^{u,w} f(x)\doteq \Lg^{u} f (x)+ \big(\Sigma(x) w\big)\cdot \grad f(x)\,.
\end{align} 
Recall $\Lg^u$ from~\eqref{eq-gen}. Next, we introduce an `extended' diffusion associated with  $\bar \Lg^{u,w}$. Let $\Wm$ be the set of maps $w:\RR_+\times \RR^d\rightarrow \RR^d$ that are Borel measurable and $\Wsm\subset \Wm$ be the set of maps $w\in \Wm$ that are of the form $w(t,\cdot)=\widetilde  w(\cdot)$, for some Borel measurable map $\widetilde  w:\RR^d\rightarrow \RR^d$.  The `extended' diffusion is now defined as the solution to the following stochastic equation: for $v\in \Um$ and $w\in \Wm$, 
\begin{align}\label{eq-Z-cont}
Z_t= \int_0^t \Big( b\big(Z_s,v(s,Z_s)\big) + \Sigma(Z_s) w(s,Z_s) \Big) \D s +\int_0^t\Sigma(Z_s) \D W_s \quad \text{ with $Z_0=x$.}\end{align}
The expectation involving the above process $Z$ for $Z_0=x$, $v\in \Um$ and $w\in \Wm$ is denoted by $\E_x^{v,w}$. We omit the dependence of $Z$ on $v$ and $w$ and it is expressed through $\E_x^{v,w}$. 

The following result (taken from \cite[Lemma 4.3]{anugu24ergodicgen} and stated for the special case of $v\in \Um$ and $w\in \Wm$)  gives us the existence of  the process $Z$ under a certain moment condition. Let the set of $\RR^d$--valued continuous functions on $[0,T]$ be denoted by $\frC([0,T],\RR^d)$.
\begin{proposition}\label{prop-Z-ext}
For $v\in \Um$ and $w\in \Wm$,  
there exists a unique $\frC([0,T],\RR^d)$--valued process ${ Z}^M$ that satisfies  
\begin{align*}
 {Z}^M_{t\wedge \tau^M} &=  \int_0^{t\wedge \tau^M}b\big({ Z}^M_s, v(s,Z^M_s)\big)\D s + \int_0^{t\wedge \tau^M}\Sigma({Z}^M_s)  {w}( s,Z^M_s)\D s+ \int_0^{t\wedge \tau^M}\Sigma({ Z}^M_s) \D { W}_s,
\end{align*}
where $ \tau^M\doteq \inf\big\{t\geq 0: \int_0^t \|w( s,Z^M_s)\|^2 \D s>M\big\}.$ 
Additionally, for $T>0$, if  $$\sup_{M>0}\E\Big[\int_0^T\|w(t, Z^M_t)\|^2 \D t\Big]<\infty,$$ 
then \eqref{eq-Z-cont} admits a unique strong solution on $[0,T]$.
\end{proposition}

By simple substitution, we obtain the aforementioned additive version of~\eqref{eq-hjb-gamma} below. 

\begin{lemma}\label{lem-eq-mod-gamma} Define $V^*_\gamma\doteq \log \widetilde V^*_\gamma$, where $\widetilde V^*_\gamma$ satisfies~\eqref{eq-hjb-gamma}. Then, $V^*_\gamma$ satisfies 
\begin{align}\label{eq-hjb-gamma-m}
\min_{u\in \bU}\max_{w\in \RR^d}\Big[\bar \Lg^{u,w}V^*_\gamma(x) +{\gamma}{} r(x,u) -\frac{1}{2}\|w\|^2\Big]&={\gamma}{} \Lambda_\gamma^*. 
\end{align}

\end{lemma}

We also obtain the following additive version of the optimality characterization  of  \eqref{eq-minimizer} in  Theorem \ref{thm-hjb-gamma}.

\begin{corollary}\label{cor-eq-min} $v^*\in \Usm$ satisfies 
\begin{align}\label{eq-min-3}  \min_{u\in \bU}\max_{w\in \RR^d}\Big[\bar \Lg^{u,w} V^*_\gamma(x)+ {\gamma}{} r(x,u)-\frac{1}{2}\|w\|^2\Big]=  \max_{w\in \RR^d}\Big[\bar \Lg^{v^*,w} V^*_\gamma(x)+ {\gamma}{} r\big(x,v^*(x)\big)-\frac{1}{2}\|w\|^2\Big]\end{align}
if and only if it satisfies 
\begin{align*}\min_{u\in \bU}\big[\Lg^u\widetilde V^*_\gamma(x)+ \gamma r(x,u)\widetilde V^*_\gamma(x)\big]= \Lg^{v^*}\widetilde V^*_\gamma(x)+ \gamma r\big(x,v^*(x)\big)\widetilde V^*_\gamma(x)\,.\end{align*}

\end{corollary}
From hereon, any generic $v\in \Usm$ that satisfies~\eqref{eq-min-3} is denoted by $ v^*$ and we set $ w^*(\cdot)\doteq  \Sigma(\cdot)\transp\grad  V^*_\gamma(\cdot)$ that is the unique $w\in \Wsm$ that  satisfies 
\begin{align}\label{eq-w-*}  \min_{u\in \bU}\max_{w\in \RR^d}\Big[\bar \Lg^{u,w} V^*_\gamma(x)+{\gamma}{} r(x,u)-\frac{1}{2}\|w\|^2\Big]=  \min_{u\in \bU}\Big[\bar \Lg^{u, w^*} V_\gamma(x)+ {\gamma}{} r(x,u)-\frac{1}{2}\| w^*(x)\|^2\Big]\,.\end{align}

\begin{lemma}\label{lem-eq-mod-gamma-rvi} Let   $V_\gamma\doteq  \log \widetilde V_\gamma$ with $V^n_\gamma=\log \widetilde V^n_\gamma$. Then, for $n\geq 1$, $V_\gamma$  satisfies 
\begin{align}\label{eq-rvi-mod-gamma-m}
\partial_t  V_\gamma(t,x)= \min_{u\in \bU}\max_{w\in \RR^d}\Big[\bar \Lg^{u,w} V_\gamma(t,x)+{\gamma}{} r(x,u)-\frac{1}{2}\|w\|^2\Big]- \exp\big( \delta V_\gamma^n(0)\big),
\end{align}
for $0\leq t\leq \tau$  with  $ V_\gamma(0,x)=  V_\gamma^n(x)$. Moreover,   $V_\gamma^{n+1} (x)= V_\gamma(\tau,x)$. 
\end{lemma}

\subsection{Additive versions of the Lyapunov condition~\eqref{a-main-gamma-Lg} in Assumption~\ref{a-main-gamma-conv} }\label{sec-add-a-main-gamma}

For further analysis, we also use a `transformed version' of  Assumption~\ref{a-main-gamma-conv} which we give below.
\begin{lemma} \label{lem-lyap-gamma-mod}Suppose  Assumption~\ref{a-main-gamma-conv}(i) holds. Then, for any $u\in \bU$ and $w\in \RR^d$, we have 
\begin{align}\label{eq-lyap-gamma-m}
\bar \Lg^{u,w} \sV(x)\leq \bar l-l(x)+\frac{1}{2}\|w\|^2,\quad \text{for $x\in \RR^d$.} 
\end{align}  
\end{lemma}
\begin{proof}
Since $\widetilde \sV= \exp(\sV)$, we have
\begin{align*}
\Lg^u\widetilde \sV(x)= \Lg^u\exp\big(\sV(x)\big)= \exp\big( \sV(x)\big) \Lg^u \sV(x) + \frac{\exp\big( \sV(x)\big)}{2} \|\Sigma(x)\transp \grad \sV(x)\|^2 \,.
\end{align*}
From Assumption~\ref{a-main-gamma-conv}, we have 
\begin{align*}
 \exp\big( \sV(x)\big) \Lg^u \sV(x)& + \frac{\exp\big( \sV(x)\big)}{2} \|\Sigma(x)\transp \grad \sV(x)\|^2  \leq \big(\bar l-l(x)\big)\exp\big( \sV(x)\big)\\
\implies  \Lg^u \sV(x) & \le   \big(\bar l-l(x)\big) - \frac{ 1}{2} \|\Sigma(x)\transp \grad \sV(x)\|^2. 
\end{align*}
Hence, from the definition of $\bar \Lg^{u,w}$, we obtain 
\begin{align*}
 \bar \Lg^{u,w}\sV(x) & =\Lg^u \sV(x)+  \big(\Sigma(x) w\big)\cdot \grad \sV(x)  \\
 &  \leq \bar l-l(x) - \frac{1}{2} \|\Sigma(x)\transp \grad \sV(x)\|^2 +\big(\Sigma(x)\transp \grad \sV(x)\big)\cdot w\\
   &\leq \bar l-l(x)+ \frac{1}{2} \|w\|^2 \,.
 \end{align*}
In the above, to get the second line, we use~\eqref{a-main-gamma-Lg} and  to get the last line, we apply Young's inequality:  for $z\in \RR^d$, $|w\cdot z| \leq \frac{1}{2} \|z\|^2 +\frac{1}{2} \|w\|^2$ with $z=\Sigma(x)\transp \grad \sV(x) $. This gives us the desired result.
\end{proof}
\begin{lemma}\label{lem-lyap-rev} Suppose Assumption~\ref{a-main-gamma-conv}(i) holds. Then, for $\xi>0$, $u\in \bU$, we have
\begin{align*}
\Lg^{u}\big(\widetilde \sV(x)\big)^{-\xi }  \geq \xi \big(l(x)-\bar l\big)\big(\widetilde \sV(x)\big)^{-\xi},\quad \text{for $x\in \RR^d$.}
\end{align*}  
\end{lemma}
\begin{proof}
We have
\begin{align*}
  \Lg^{u} \big(\widetilde \sV(x)\big)^{-\xi} & = -\frac{\xi}{\big(\widetilde \sV(x)\big)^{1+\xi}} \Lg^{u} \widetilde \sV(x) + \frac{\xi(1+\xi)}{\big(\widetilde \sV(x)\big)^{2+\xi}} \|\Sigma(x)\transp \grad \widetilde \sV(x)\|^2 \\
  &\geq -\frac{\xi}{\big(\widetilde \sV(x)\big)^{1+\xi}} \Lg^{u} \widetilde \sV(x)\\
&\geq  -\frac{\xi}{\big(\widetilde \sV(x)\big)^{\xi}} \big(\bar l-l(x)\big) \,. 
\end{align*}
In the third line, we use Assumption~\ref{a-main-gamma-conv}(i). This proves the result.
\end{proof}
From the above lemma, we have the following immediate corollary.
\begin{corollary}\label{cor-lyap-rev}Suppose Assumption~\ref{a-main-gamma-conv}(i) holds. Then,  for every $\tau>0$, $\xi> 0$ and $v\in \Um$, the following hold:
\begin{enumerate}
\item [(i)]
$$\E_x^{v}\Big[\big(\widetilde \sV(X_\tau)\big)^{-\xi}\Big]\geq e^{-\bar l\xi\tau}\big(\widetilde \sV(x)\big)^{-\xi} \,. $$
\item [(ii)] For $L>0$ and an open ball $\frB$ such that $\inf_{y\in \frB^c} \big(l(y)- \bar l-L\xi^{-1}\big)\geq 0$, we have
\begin{align*}
\E_x^{v}\Big[e^{-L\widehat \tau_\frB} \Big]\geq \big(\widetilde \sV(x)\big)^{-\xi}, \text{ for $x\in \frB^c$.}
\end{align*}
Recall that $\widehat \tau_\frB$ is the first hitting time of set $\frB$, by the process $X$.
\end{enumerate}
\end{corollary}
\begin{proof} Fix $v\in \Um$. To prove part (i), we use Lemma~\ref{lem-lyap-rev} and apply It\^o's formula to $e^{\bar l\xi t}\big(\sV(X_t)\big)^{-\xi}$, up to $t=\tau$ to get
\begin{align*}
\E_x^{v}\Big[e^{\bar l \xi \tau}\big(\widetilde \sV(X_\tau)\big)^{-\xi}\Big]&=  \big(\widetilde \sV(x)\big)^{-\xi}+ \E_x^{v}\Big[\int_0^\tau  \big(e^{\bar l \xi t}\Lg^{u} \big(\widetilde \sV(X_t)\big)^{-\xi}+\bar l \xi e^{\bar l {\xi} t} \big(\widetilde \sV(X_t)\big)^{-\xi}\big)\D t\Big]\\
&\geq \big(\widetilde \sV(x)\big)^{-\xi }  +\E_x^{v}\Big[\int_0^\tau e^{\bar l\xi t} \frac{\big(-\bar l\xi +\xi l(X_t)+\bar l \xi \big)}{\big(\widetilde \sV(X_t)\big)^{\xi}}  \D t\Big]\\
&\geq \big(\widetilde \sV(x)\big)^{-\xi}  \,.
\end{align*}
In the above, to get the second line, we use Lemma~\ref{lem-lyap-rev} and to get the last line, we use the non-negativity of $l$ and $\widetilde \sV$, and the fact that $\xi> 0$.

To prove part (ii), we apply It\^o's formula to $e^{-Lt}\big(\widetilde \sV(X_t)\big)^{-\xi }$, up to $t=\widehat \tau_\frB\wedge T$ with $x\in \frB^c$ to get
\begin{align*}
\E_x^{v}\Big[e^{-L( \widehat \tau_\frB\wedge T)}\big(\widetilde \sV(X_ {\widehat \tau_\frB\wedge T})\big)^{-\xi }\Big]&=  \big(\widetilde \sV(x)\big)^{-\xi }+ \E_x^{v}\Big[\int_0^{ \widehat \tau_\frB\wedge T}  \big(e^{-L t}\Lg^{v} \big(\widetilde \sV(X_t)\big)^{-\xi }-L e^{-Lt} \big(\widetilde \sV(X_t)\big)^{-\xi }\big)\D t\Big]\\
&\geq \big(\widetilde \sV(x)\big)^{-\xi }  +\E_x^{v}\Big[\int_0^{ \widehat \tau_\frB} e^{-L t} \frac{\xi  \big(-\bar l + l(X_t)-L\xi ^{-1}\big)}{\big(\widetilde \sV(X_t)\big)^{\xi}}  \D t\Big]\\
&\geq \big(\widetilde \sV(x)\big)^{-\xi }  \,.
\end{align*}
We obtain the last line from the choice of $\frB$ and $L>0$. From above, using the fact that $\big(\widetilde \sV(x)\big)^{-\xi }\leq 1$, we get
\begin{align}
\E_x^{v}\Big[e^{-L( \widehat \tau_\frB\wedge T)}\Big]\geq \E_x^{v}\Big[e^{-L( \widehat \tau_\frB\wedge T)}\big(\widetilde \sV(X_ {\widehat \tau_\frB\wedge T})\big)^{-\xi }\Big]\geq \big(\widetilde \sV(x)\big)^{-\xi }\,.
\end{align}
Now taking $T\to\infty$, we obtain the desired result.
\end{proof}

 \subsection{An intermediate family of iterates}
In this section, we  introduce the intermediate  family of iterates as mentioned earlier.  Let the new family of iterates $\{ {\overline V}_\gamma^n: n\geq 1\}$ be defined as follows:  let $\widetilde {\overline V}_\gamma^0 \doteq   \widetilde V_\gamma^0$ and $ \widetilde {\overline V}_\gamma$ be the solution to 
\begin{align}\label{eq-widehat}
\partial_t \widetilde {\overline V}_\gamma (t,x)= \min_{u\in \bU}\big[\Lg^u \widetilde {\overline V} _\gamma(t,x)+ \gamma r(x,u)\widetilde {\overline V}_\gamma (t,x)\big]- \gamma \Lambda^*_{\gamma}  \widetilde {\overline V}_\gamma (t,x),\quad   \widetilde {\overline V}_\gamma (0,x)= \widetilde {\overline V}_\gamma^n(x)
\end{align}
and  $\widetilde {\overline V}_\gamma^{n+1}(x)=\widetilde {\overline V}_\gamma(\tau, x)$. Set $\overline V_\gamma^n(x)\doteq \log \big(\widetilde {\overline V}_\gamma^n(x)/\widetilde {\overline V}_\gamma^n(0)\big)$. Clearly, $\overline V^n_\gamma(0)=0$ for $n\geq 1$, and 
  from the above definition, we see that $ \overline  V_\gamma^n= V_\gamma^n+\phi^n$ with 
\begin{align}\label{eq-const-series} \phi^n\doteq  \sum_{i=1}^{n-1} \big(V_\gamma^n(0)- \Lambda^*_{\gamma}\big)\,.\end{align}

Note that we have replaced $\big( \widetilde V^n_\gamma (0)\big)^\delta$ in \eqref{eq-rvi-mod-gamma}, by $\gamma \Lambda^*_{\gamma}$ in \eqref{eq-widehat} for the iterates $\{ {\overline V}_\gamma^n:n\geq 1\}$. Due to this, the estimates in Lemmas~\ref{lem-b-n} and~\ref{lem-w-bound} do not depend on  $\widetilde V^n_{\gamma}(0)$ (and thereby $V^n_\gamma(0)$) explicitly. This ensures that a priori bounds on $\widetilde V^n_{\gamma}(0)$ (and thereby $V^n_\gamma(0)$) are not required in the proofs of  Lemmas~\ref{lem-b-n} and~\ref{lem-w-bound}. This consequently, simplifies the verification of  conditions~\eqref{eq-drift} and~\eqref{eq-minor} through Lemmas~\ref{lem-exp-conv} and~\ref{lem-exp-minor}, respectively.

\begin{remark}\label{rem-choice-cont}
Just as in Lemma~\ref{lem-eq-mod-gamma-rvi}, we also have the following: $\overline V_\gamma\doteq  \log \widetilde {\overline V}_\gamma$ with $\overline V^n_\gamma=\log \widetilde {\overline V}^n_\gamma$, for $n\geq 1$, satisfies 
\begin{align}\label{eq-rvi-mod-gamma-int}
\partial_t  \overline V_\gamma(t,x)= \min_{u\in \bU}\max_{w\in \RR^d}\Big[\bar \Lg^{u,w} \overline V_\gamma(t,x)+ {\gamma}{} r(x,u)-\frac{1}{2}\|w\|^2\Big]- {\gamma}{}  \Lambda^*_{\gamma} ,
\end{align}
 for $0\leq t\leq \tau$ with $ \overline V_\gamma(0,x)=  \overline V_\gamma^n(x)$. Moreover,  $\overline V_\gamma^{n+1} (x)= \overline V_\gamma(\tau,x)$.
  It is clear that $ \bar v\in \Um$ satisfies 
\begin{align}\label{eq-min-1}  \min_{u\in \bU}\max_{w\in \RR^d}\Big[\bar \Lg^{u,w} \overline V_\gamma(t,x)+ {\gamma}{} r(x,u)-\frac{1}{2}\|w\|^2\Big]=  \max_{w\in \RR^d}\Big[\bar \Lg^{ \bar v,w} \overline V_\gamma(t,x)+ {\gamma}{} r\big(x, \bar v (t,x)\big)-\frac{1}{2}\|w\|^2\Big]\end{align}
if and only if it satisfies 
\begin{align*}\min_{u\in \bU}\big[\Lg^u\widetilde {\overline V}_\gamma(t,x)+ \gamma r(x,u)\widetilde {\overline V}_\gamma(t,x)\big]= \Lg^{\bar v}\widetilde {\overline V}_\gamma(t,x)+ \gamma r\big(\bar v(t,x)\big)\widetilde {\overline V}_\gamma(t,x)\,.\end{align*}
Moreover, $  \bar w (t, \cdot)\doteq  \Sigma(\cdot)\transp\grad  \overline V_\gamma(t,\cdot)$ is the unique $w\in \Wm$ that  satisfies 
\begin{align*}  \min_{u\in \bU}\max_{w\in \RR^d}\Big[\bar \Lg^{u,w} \overline V_\gamma(t,x)+ {\gamma}{} r(x,u)-\frac{1}{2}\|w\|^2\Big]=  \min_{u\in \bU}\Big[\bar \Lg^{u, \bar w} \overline V_\gamma(t,x)+ {\gamma}{} r(x,u)-\frac{1}{2}\| \bar w(t,x)\|^2\Big]\,.\end{align*}
From hereon, any generic $v\in \Um$ that satisfies~\eqref{eq-min-1}  is denoted by $ \bar v$ and  $ \bar  w(t,x) \doteq \Sigma(x)\transp\grad  \overline V_\gamma(t,x)\,.$  Even though $\bar v$ and $\bar w$ depend on $n$ (\emph{via.} $\overline V_\gamma (0,x)=\overline V_\gamma^n(x)$), we suppressed the dependence as we always fix $n$ whenever $\bar v$ and $ \bar w$ are involved. 

\end{remark}

\subsection{Key a priori bounds on $V^*_\gamma$ and $\overline V^n_\gamma$}
We give important a priori estimates involving $V_\gamma^* =\log {\widetilde V}^*_{\gamma}$ and $\overline V_\gamma^n = \log \widetilde {\overline V}^n_\gamma$.

\begin{lemma}\label{lem-b-*} Under Assumptions~\ref{a-regularity} and~\ref{a-main-gamma-conv}, the following estimates hold: for every  $\xi >0$, there exist constants $\overline C^*_{\gamma,\theta} , \underline C^*_{\gamma,\xi }>0$  such that
\begin{align}\label{eq-ub-*}
V_\gamma^*(x)&\leq  \theta  \sV(x)+\overline C^*_{\gamma,\theta}\,, \\\label{eq-lb-*}
V_\gamma^*(x)&\geq -\xi \sV(x) -\underline C^*_{\gamma,\xi } \,.
\end{align}

\end{lemma}

\begin{proof}
The proof of the first part  uses Theorem~\ref{thm-hjb-gamma}(ii). From Assumption~\ref{a-main-gamma-conv}(ii), we choose an open ball $\frB\subset \RR^d$ such that for $x\in \frB^c$, 
\begin{align}\label{eq-choice-B} \inf_{x\in \frB^c}\big\{\gamma\max_{u\in \bU} r(x,u)-\gamma \Lambda_\gamma^* - \theta  l(x) \big\}\leq 0\,.\end{align}

Applying Theorem~\ref{thm-hjb-gamma}(ii) with open ball $\frB$ (and $\overline \frB$ denoting its closure) as chosen above, $x\in \frB^c$ and $v^*\in \Usm$ that satisfies~\eqref{eq-minimizer}, we get
\begin{align*}
\widetilde V_\gamma^*(x)&= \E_x^{v^*}\Big[\exp\Big( \int_0^{\widehat \tau_\frB} \gamma \big( r(X_t, v^*(X_t))-\Lambda_\gamma^*\big)\D t\Big)\widetilde V_\gamma^*(X_{\widehat \tau_\frB})\Big]\\
&\leq \E_x^{ v^*}\Big[\exp\Big(\theta  \int_0^{\widehat \tau_\frB} l(X_t)\D t\Big) \widetilde V_\gamma^*(X_{\widehat \tau_\frB})\Big] \\
&\leq \sup_{y\in \overline \frB} \frac{\widetilde V_\gamma^{*}(y)}{\widetilde \sV(y)^{\theta}} \E_x^{ v^*}\Big[\exp\Big(\theta  \int_0^{\widehat \tau_\frB} l(X_t)\D t\Big) \widetilde \sV(X_{\widehat \tau_\frB})^{\theta }\Big]\\
&= \sup_{y\in \overline \frB} \frac{\widetilde V_\gamma^{*}(y)}{\widetilde \sV(y)^{\theta}} \E_x^{v^*}\Big[\Big(\exp\Big(  \int_0^{\widehat \tau_\frB} l(X_t)\D t\Big) \widetilde \sV(X_{\widehat \tau_\frB})\Big)^{\theta}\Big] \\
&\leq \sup_{y\in \overline  \frB} \frac{\widetilde V_\gamma^{*}(y)}{\widetilde \sV(y)^{\theta}}\Big(\E_x^{v^*}\Big[\exp\Big( \int_0^{\widehat \tau_\frB}l(X_t)\D t\Big) \widetilde \sV(X_{\widehat \tau_\frB})\Big]\Big)^{\theta} \\
&\leq\sup_{y\in \overline \frB} \frac{\widetilde V_\gamma^{*}(y)}{\widetilde \sV(y)^{\theta}} \widetilde \sV(x)^{\theta} \,.
\end{align*}
Recall $\theta$ from Assumption~\ref{a-main-gamma-conv}(ii). In the above, to get the second line, we use~\eqref{eq-choice-B}; to get the fifth line, we use Jensen's inequality and to get the last line, we use Remark~\ref{rem-lyap-mult}.

Taking logarithm on both sides, we get
\begin{align}\nonumber
 V_\gamma^*(x)&\leq \log\Big( \sup_{y\in \overline \frB} \frac{\widetilde V_\gamma^{*}(y)}{\widetilde \sV(y)^{\theta}}\Big) +\theta \sV(x) = \sup_{y\in \overline \frB}\big(V_\gamma^*(y)- \theta \sV(y)\big) +\theta \sV(x)\,.
\end{align} 
To get the above equality, we use the definitions of $V_\gamma^*$ and $\sV$. This proves~\eqref{eq-ub-*} with $\overline C^*_{\gamma,\theta}\doteq  \sup_{y\in\overline  \frB}\big(V_\gamma^*(y)- \theta\sV(y)\big)$.

Next, we move on to prove~\eqref{eq-lb-*}. Here, we fix $\xi>0$ and choose an open ball $\frB$ such that $\inf_{y\in \frB^c}\big(l(y)-\bar l-\gamma \Lambda_\gamma^* \xi^{-1}\big)\geq 0$. To that end, for any $x\in \frB^c$, we have
\begin{align*}\widetilde V_\gamma^*(x) &= \E_x^{v^*}\Big[\exp\Big( \int_0^{\widehat \tau_\frB}\gamma \big(r(X_t,\bar v(t,X_t))-\Lambda_\gamma^* \big)\D t\Big) \widetilde V_\gamma^*(X_{\widehat \tau_\frB})\Big] \\
&\geq\E_x^{\bar v}\Big[ e^{-\gamma \widehat \tau_\frB\Lambda_\gamma^*}\widetilde V_\gamma^{*}(X_{\widehat \tau_\frB})\Big]\\
&\geq  \big(\inf_{y\in \overline \frB} \widetilde V_\gamma^{*}(y)\big)\E_x^{ v^*}\Big[e^{-\gamma \widehat \tau_\frB\Lambda_\gamma^*}\Big]\\
&\geq  \big(\inf_{y\in \overline \frB} \widetilde V_\gamma^{*}(y)\big) \big(\widetilde \sV(x)\big)^{-\xi}\,.  \end{align*} 
Here, $\overline \frB$ again denotes the closure of $\frB$. To get the second line, we use the non-negativity of $r(\cdot,\cdot)$ and to get the last line, we use Corollary~\ref{cor-lyap-rev}(ii) with $K=\gamma\Lambda_\gamma^*$. From above, using the definition of $V_\gamma^*$ and $\sV$,  we have $V_\gamma^*(x)\geq \inf_{y\in \overline  \frB}V_\gamma^{*}(y)  - \xi \sV(x)$. This proves~\eqref{eq-lb-*} (and completes the proof) with $\underline C^*_{\gamma,\xi}\doteq  \inf_{y\in \overline \frB}V_\gamma^{*}(y) $ with dependence on $\xi$ through $\overline \frB$.  
\end{proof}

\begin{lemma}\label{lem-b-n} Suppose Assumptions~\ref{a-regularity} and~\ref{a-main-gamma-conv} hold and for a fixed $n\geq 1 $, suppose that $\widetilde {\overline V}_\gamma$ is the solution to~\eqref{eq-widehat} such that $\widetilde {\overline V}_\gamma(0,x)=\widetilde {\overline V}_\gamma^{n-1}(x)$.  Also, suppose that $\|\overline V_\gamma^{n-1}-V^*_\gamma\|_{\widetilde \beta,\sV}<1$, for some $0<\widetilde \beta<1-\theta$. Then, the following estimates hold: for some $ C_{\sup},  C_{\inf}>0$ (independent of $n$),
\begin{align}\label{eq-ub-n}
\overline V_\gamma(t,x)&\leq  (\theta +\widetilde \beta) \sV(x)+\overline C^*_{\gamma,\theta} + {C_{\sup}}{} \,, \\\label{eq-lb-n}
\overline V_\gamma^{n}(x)&\geq -(\xi+\widetilde \beta)\sV(x) -\underline C^*_{\gamma,\xi}- {C_{\inf}}{},
\end{align}
for $0\leq t\leq \tau$. In particular, the estimate in~\eqref{eq-ub-n} holds for $\overline V_\gamma^{n}=\overline V_\gamma(\tau,x)$.
\end{lemma}
\begin{proof} From the hypothesis on $\overline V^n_\gamma$, the definition of $\|\cdot\|_{\widetilde \beta,\sV}$ and the upper bound on $V^*_\gamma$ (from~\eqref{eq-ub-*}), we have 
$$ \overline V^{n-1}_\gamma(x) \leq (\theta+\widetilde \beta) \sV(x)+\overline C^*_{\gamma,\theta}\,. $$

To prove part (i), we consider  $\widetilde {\overline V}_\gamma$ which is the solution to~\eqref{eq-widehat} such that $\widetilde {\overline V}_\gamma(0,x)=\widetilde {\overline V}_\gamma^{n-1}(x)$. It is clear that  $\widetilde{\overline V}_\gamma(t,x)$, for $0\leq t\leq \tau$ satisfies
$$\widetilde {\overline V}_\gamma(\tau-t,x)\leq  \E_x^{v}\Big[\exp\Big(\int_0^{t} \gamma \big( r(X_s, v(X_s))- \Lambda^*_{\gamma}\big)\D s\Big)\widetilde {\overline V}_\gamma^{n-1}(X_t) \Big],$$
for every $v\in \Usm$.  Dividing the above display by $\big({\widetilde \sV}(x)\big)^{\theta+\widetilde \beta}$ on both sides, we have 
\begin{align*}
\frac{\widetilde {\overline V}(\tau-t,x)}{\big({\widetilde \sV}(x)\big)^{\theta+\widetilde \beta}}&\leq \frac{1}{\big(\widetilde \sV(x)\big)^{\theta+\widetilde \beta}}\E_x^{v}\Big[\exp\Big( \int_0^{t}  \big( \gamma  r(X_s, v(X_s))-\gamma\Lambda^*_{\gamma}\big)\D s \Big)\widetilde {\overline V}_\gamma^{n-1}(X_t) \Big]\\
&\leq \frac{1}{\big({\widetilde \sV}(x)\big)^{\theta+\widetilde \beta}}\E_x^{v}\Big[\exp\Big( \int_0^{t}  \big( \theta l(X_s) + k_r\big)\D s \Big)\widetilde {\overline V}_\gamma^{n-1}(X_t) \Big]\\
&\leq \frac{e^{k_r t +{\bar l t(\theta+\widetilde \beta)}{} }}{\big({\widetilde \sV}(x)\big)^{\theta+\widetilde \beta}} \sup_{y\in \RR^d} \frac{\widetilde {\overline V}_\gamma^{n-1}(y)}{\big({\widetilde \sV}(y)\big)^{\theta+\widetilde \beta}}\E_x^{v}\Big[\exp\Big((\theta+\widetilde \beta) \int_0^{t} \big(l(X_s)-\bar l\big)\D s \Big)\big(\widetilde { \sV}(X_t)\big)^{\theta+\widetilde \beta} \Big]\\
&\leq  \frac{e^{k_r\tau +{\bar l \tau (\theta+\widetilde \beta) }{} }}{\big({\widetilde \sV}(x)\big)^{\theta+\widetilde \beta}}  \sup_{y\in \RR^d} \frac{\widetilde {\overline V}_\gamma^{n-1}(y)}{\big({\widetilde \sV}(y)\big)^{\theta+\widetilde \beta}}\E_x^{v}\Big[\exp\Big( \int_0^{t} \big(l(X_s)-\bar l\big)\D s \Big)\widetilde \sV(X_t) \Big]^{(\theta+\widetilde \beta) } \\
&\leq e^{k_r \tau +{\bar l \tau (\theta+\widetilde \beta)}{} }  \sup_{y\in \RR^d} \frac{\widetilde {\overline V}_\gamma^{n-1}(y)}{\big({\widetilde \sV}(y)\big)^{\theta+\widetilde \beta}}\,.
\end{align*}
In the above, to get the second line, we use~\eqref{eq-inf-c-1}; to get the third line, we multiply and divide by $\big({\widetilde \sV}(X_t)\big)^{\theta+\widetilde \beta}$ inside the expectation and bound the expectation; to get the fourth line, we use Jensen's inequality and to get the last line, we use~\eqref{eq-lyap-ineq}. Taking the logarithm on both sides and using the definitions of $\overline V_\gamma$ and $\sV$, we get~\eqref{eq-ub-n} with $C_{\sup}\doteq k_r \tau +{\bar l \tau (\theta+\widetilde \beta)}$.  

To prove~\eqref{eq-lb-n}, we  choose $v\in \Um$ which is a minimizer for~\eqref{eq-widehat}. Then, we have 
$$ \widetilde {\overline V}^{n+1}_\gamma(x)=  \E_x^{v}\Big[\exp\Big(\int_0^{\tau} \gamma \big( r(X_t, v(X_t))- \Lambda^*_{\gamma}\big)\D t\Big)\widetilde {\overline V}_\gamma^{n-1}(X_\tau) \Big]\,.$$
From here, using Corollary~\ref{cor-lyap-rev}(i) and arguing as above, we can obtain~\eqref{eq-lb-n}. This completes the proof.
\end{proof}

\subsection{Verifying the extended diffusion $Z$  satisfies~\eqref{eq-drift}  and~\eqref{eq-minor} of Proposition~\ref{prop-contract} }\label{sec-add-exp-ergodicity}

We first  establish the following boundedness result associated with the auxiliary controls in the extended diffusion $Z$, for $(v,w)= (\bar v,w^*)$ or $(v^*, \bar w)$  for which we use the bounds from Lemmas~\ref{lem-b-*} and~\ref{lem-b-n}.  Recall $v^*$ and $\bar v$ from~\eqref{eq-min-3} and~\eqref{eq-min-1}, respectively. Also, recall that \begin{align}\label{def-w}w^*(\cdot)=   \Sigma(\cdot)\transp\grad  V^*_\gamma(\cdot) \quad \text{ and } \quad  \bar w (t, \cdot)= \Sigma(\cdot)\transp\grad  \overline V_\gamma(t,\cdot)\end{align}

\begin{lemma}\label{lem-w-bound} Suppose Assumptions~\ref{a-regularity} and~\ref{a-main-gamma-conv} hold.  Let   $\theta$ as in Assumption~\ref{a-main-gamma-conv} and $\|\overline V_\gamma^{n-1}-V^*_\gamma\|_{\widetilde \beta,\sV}<1$, for some $0<\widetilde \beta<1-\theta$.  Then, for $n\geq 1$, $0\leq t\leq \tau$, $\vt>0$, $\xi>0$ and $(v,w)= (\bar v,w^*)$ or $(v^*, \bar w)$, we have 
\begin{align*} 
\frac{1}{2}\E_x^{v, w}\Big[\int_0^t e^{\vt s} \|w(s,Z_s)\|^2\D s\Big]&\leq  {\gamma}  \E_x^{v,w}\Big[\int_0^t e^{\vt s} \Big(r\big(Z_s,v(s,Z_s)\big)-\Lambda^*_\gamma +\vt (\theta+\widetilde \beta) \sV(Z_s)\Big)\D s\Big] \\
& \quad +(  \theta +\widetilde \beta) e^{\vt t} \E_x^{v, w}\big[\sV(Z_t)\big] + (\xi+\widetilde \beta) \sV(x) + H(\vt)\,.
\end{align*}
Here, \begin{align*}H (\vt)&\doteq   \big(C_{\sup}+\overline C^*_{\gamma,\theta}\big) + \big(C_{\inf}+ \underline C^*_{\gamma,\xi}\big)\,.\end{align*}
\end{lemma}
\begin{proof} Let $(v,w)=(v^*, \bar w)$. We apply It\^o's formula to $  e^{\vt t}\overline V_\gamma(t -u,Z_u) $ with $(v,w)=(v^*, \bar w)$ and $\overline V_\gamma(0,x)= \overline V_\gamma^{n-1}(x)$ up to $u=t$ to obtain 
\begin{align*}
\overline V_\gamma^n(x)  &\leq  \E_x^{v^*,\bar w}\Big[ \int_0^t  e^{\vt s}\Big( {\gamma} r\big(Z_s,v^*(Z_s)\big) - \gamma \Lambda^*_{\gamma} -\frac{1}{2}\| \bar w(s,Z_s)\|^2 +\vt \overline V_\gamma (t-s, Z_s)\Big)\D s \Big] +  \E_x^{ v^*,\bar w}\Big[ e^{\vt t}   \overline V_\gamma^{n-1}(Z_t)\Big] \\
&= \E_x^{v^*,\bar w}\Big[ \int_0^t e^{\vt s} \Big({\gamma}r\big(Z_s,v^*(Z_s)\big) - {\gamma}\Lambda^*_{\gamma}   -\frac{1}{2}\|\bar w(s,Z_s)\|^2 +\vt \overline V_\gamma (t-s,Z_s) \Big)\D t \Big] + \E_x^{ v^*,\bar w}\Big[ e^{\vt t}  \overline V_\gamma^{n-1}(Z_t)\Big] \,.
\end{align*}
From here, rearranging the last line above, we get 
\begin{align}\nonumber
&\frac{1}{2 } \E_x^{v^*,\bar w}\Big[\int_0^t e^{\vt s}\|\bar w(s,Z_s)\|^2\D s \Big]\\\nonumber
&\leq \E_x^{v^*,\bar w}\Big[ \int_0^t  e^{\vt s}\Big({\gamma} r\big(Z_s,v^*(Z_s)\big) - \gamma \Lambda^*_{\gamma} +\vt \overline V_\gamma(t-s,Z_s)   \Big)\D s \Big]   + \E_x^{ v^*,\bar w}\Big[ e^{\vt t}  \overline V_\gamma^{n-1}(Z_t)\Big]- \overline V_\gamma^n(x)\\\nonumber
&\leq \E_x^{v^*,\bar w}\Big[ \int_0^t e^{\vt s}\Big({\gamma} r\big(Z_s,v^*(Z_s)\big) - \gamma \Lambda^*_{\gamma}+\vt (\theta+\widetilde \beta) \sV(Z_s)\Big) \D s \Big] \\\label{eq-vw-1}
&\qquad\qquad +  (\theta +\widetilde \beta) e^{\vt t}\E_x^{ v^*,\bar w}\Big[   \sV(Z_t)\Big]+ (\xi+\widetilde \beta) \sV(x) + \big( \underline C^*_{\gamma,\xi}+ { C_{\inf}}{}\big)
 + \big(\overline C^*_{\gamma,\theta} +{ C_{\sup}}{}\big)  \,.
\end{align}
In the above, to get the second inequality, we use~\eqref{eq-ub-n} to bound  $\E_x^{ v^*,\bar w}\Big[   \overline V_\gamma^{n-1}(Z_t)\Big]$ and $\overline V_\gamma (t-s,Z_s)$ from above and~\eqref{eq-lb-n} to bound $\overline V_\gamma^n(x)$ from below.
 
Similarly, letting $(v,w)=(\bar v,w^*)$, 
applying It\^o's formula to $ V_\gamma^*(Z_t)$ and arguing as above,  we have 
\begin{align*}
\frac{1}{2}\E_x^{\bar v,w^*}\Big[\int_0^t e^{\vt s}\|w^*(Z_s)\|^2\D s \Big]
&\leq \E_x^{\bar v,w^*}\Big[ \int_0^t e^{\vt s}\Big({\gamma} r\big(Z_s,\bar v(Z_s)\big) - \gamma \Lambda^*_{\gamma}+\vt  V^*_\gamma (Z_s)\Big) \D s \Big]  \\
&\qquad +  \theta  e^{\vt t}\E_x^{ \bar v, w^*}\Big[   \sV(Z_t)\Big]+ \xi \sV(x) + \overline C_{\gamma,\theta}^*+\underline {C}_{\gamma,\xi}^*\,.
\end{align*}
In the above, the constants $\overline C^*_{\gamma,\theta}$ and $\underline C^*_{\gamma,\xi}$ are from~\eqref{eq-ub-*} and~\eqref{eq-lb-*}, respectively. From~\eqref{eq-vw-1} and the above display, we obtain the result.
\end{proof}

\begin{corollary}Suppose  Assumptions~\ref{a-regularity} and~\ref{a-main-gamma-conv} hold. Then,   for $(v,w)= (\bar v,w^*)$ or $(v^*, \bar w)$, and for every $T>0$, there exists a unique strong solution $Z$ to~\eqref{eq-Z-cont} on $[0,T]$.
\end{corollary}
\begin{proof} The corollary immediately follows from Lemma~\ref{lem-w-bound} and Proposition~\ref{prop-Z-ext}.
\end{proof}

We now verify  that under $(v,w)= (\bar v,w^*)$ or $(v^*, \bar w)$ the extended diffusion $Z$  in~\eqref{eq-Z-cont} 
satisfies the drift condition~\eqref{eq-drift} of Proposition~\ref{prop-contract}.  For any $v\in \Um$ and $w\in \Wm$, let  $$\sQ^{v,w}_\tau(x,A)\doteq \E_x^{v,w}\big[\Ind_{A}(Z_\tau)\big],$$
for $x\in \RR^d$ and Borel measurable set $A\subset \RR^d$. For a Borel measurable function $f:\RR^d\rightarrow \RR$, $\sQ_\tau^{v,w}(f)(x)\doteq \int_{\RR^d}f(y) \sQ_\tau^{v,w}(x,\D y)\,. $

\begin{lemma}\label{lem-exp-conv}Suppose the  hypothesis of Lemma~\ref{lem-w-bound} holds and $(v,w)= (\bar v,w^*)$ or $(v^*, \bar w)$.
 Then, for  every $\xi>0$, we have
\begin{align}\label{eq-fl-cond-gamma} \sQ_\tau ^{v,w} \big(\sV\big)(x) \leq \varrho(\widetilde \beta) \sV(x) + K(\widetilde \beta),\end{align}
where, $\vt \doteq \frac{(1-\theta) \lambda}{1+\theta+\widetilde \beta}$, \begin{align}\label{eq-constants}\varrho(\widetilde \beta)\doteq  \frac{e^{-\vt \tau }(1+\xi +\widetilde \beta )}{1-\theta-\widetilde \beta}, \text{ and }  K(\widetilde \beta) \doteq   \frac{e^{-\vt\tau}}{1-\theta-\widetilde \beta} \Big(\big({\bar l}{}+{k_r}{} +  {(1-\theta) k_\sV}{}\big) \frac{(e^{\vt \tau}-1)}{\vt} + H(\vt )\Big)\,.\end{align}
Moreover, \begin{enumerate}
\item [(i)] For $\tau >\widetilde \tau(\widetilde\beta) \frac{1}{\vt} \log \Big(\frac{1-\theta-\widetilde \beta}{1+\xi +\widetilde \beta }\Big)$, we have $0<\varrho(\widetilde \beta)<1$. 
\item [(ii)] $K_{\sup}\doteq \limsup_{\widetilde \beta\downarrow 0} K(\widetilde \beta)<\infty$. 
\end{enumerate}
\end{lemma}

\begin{proof} We fix $(v,w)$ as per the hypothesis of the lemma.  To begin with, recall~\eqref{eq-lyap-gamma-m} and constants $k_r$ and $k_\sV$ from Remark~\ref{rem-k-inf-c}. We have
\begin{align}\nonumber
\bar \Lg^{u,w} \sV(x)&\leq \bar l -l(x) +\frac{1}{2}\|w\|^2 \\\nonumber
&= \bar l-\theta l(x) +\frac{1}{2}\|w\|^2 - {(1-\theta)}{}l(x)\\ \label{eq-inf-c-2}
&\leq  \bar l-\theta l(x) +\frac{1}{2}\|w\|^2 - (1-\theta) \lambda\sV(x) +{(1-\theta) k_\sV}{}\,. 
\end{align}
To get the third line, we use~\eqref{eq-inf-c-1} and the definition of $\sV$. 
  Applying It\^o's formula to $e^{\vt t}\sV(Z_t)$ up to $t$ with $v\in \Um$ and $w\in \Wm$, we obtain
\begin{align*}
&\E_x^{v,w}\Big[e^{\vt t}\sV(Z_t)\Big]\\
&= \sV(x) + \E_x^{v,w}\Big[\int_0^t e^{\vt s} \big( \bar \Lg^{v, w} \sV(Z_s)  + \vt \sV(Z_s)\big) \D s \Big]\\
&\leq \sV(x) + \E_x^{v, w}\Big[\int_0^te^{\vt s}  \big( {\bar l}{}-{\theta l(Z_s)}{}+\frac{1}{2}\| w(s,Z_s)\|^2 + {(1-\theta) k_\sV}{} -(1-\theta)\lambda\sV(Z_s) +\vt \sV(Z_s)\big)\D t \Big]\\
&\leq \sV(x) + \E_x^{v, w}\Big[\int_0^t e^{\vt s} \big({\bar l}{}-{\theta l(Z_s)}{} + {(1-\theta) k_\sV}{}  -(1-\theta)\lambda\sV(Z_s)  +\vt \sV(Z_s)  \big) \D s \Big]\\
&\quad+{\gamma}  \E_x^{v,w}\Big[\int_0^t  e^{\vt s}\Big(r\big(Z_s,v(s,Z_s)\big)-\Lambda^*_\gamma +\vt (\theta+\widetilde \beta) \sV(Z_s)\Big)\D s\Big] +(  \theta +\widetilde \beta) e^{\vt t} \E_x^{v, w}\big[\sV(Z_t)\big]  \\
& \quad + (\xi+\widetilde \beta) \sV(x) + H(\vt)\\
&\leq \sV(x) + \E_x^{v, w}\Big[\int_0^t  e^{\vt s}\big({\bar l}{} + k_r  + {(1-\theta) k_\sV}{}   \big) \D s \Big]\\
&\quad+(  \theta +\widetilde \beta) e^{\vt t} \E_x^{v, w}\big[\sV(Z_t)\big] + (\xi+\widetilde \beta) \sV(x) + H(\vt) \,. 
\end{align*}
In the above, to get the second inequality, we use Lemma~\ref{lem-w-bound} and~\eqref{eq-inf-c-1} and to get the third inequality, we use the definition of $\vt$.
Rearranging, we get
 \begin{align*}
 \E_x^{v, w}\Big[\sV(Z_t)\Big]&\leq  \frac{e^{-\vt t}(1+\xi+\widetilde \beta)}{(1-\theta-\widetilde \beta)}\sV(x) + \frac{e^{-\vt t}}{(1-\theta-\widetilde \beta)}\E_x^{v, w}\Big[\int_0^t  e^{\vt s}\big({\bar l}{}+k_r  + (1-\theta) k_\sV\big)\D t \Big]\\
&\quad+\frac{e^{-\vt t}}{(1-\theta-\widetilde \beta)} {H(\vt )}{} \,.
 \end{align*}
 
This proves~\eqref{eq-fl-cond-gamma} with $t=\tau$ . As the rest of proof is straightforward, we omit it.
\end{proof}

We next verify that under $(v,w)= (\bar v,w^*)$ or $(v^*, \bar w)$ the extended diffusion $Z$  in~\eqref{eq-Z-cont} satisfies the minorization condition~\eqref{eq-minor} of Proposition~\ref{prop-contract}. 

\begin{lemma}\label{lem-exp-minor}Suppose the  hypothesis of Lemma~\ref{lem-w-bound} holds and $(v,w)= (\bar v,w^*)$ or $(v^*, \bar w)$. 
 Define $\sV_R\doteq \{x\in \RR^d: \sV(x)\leq R\}$. Then, for every  $R>0$ and for a large enough compact set $\mathfrak{K}=\mathfrak{K}(R)\subset \RR^d$, there exists a constant $\alpha = \alpha(\widetilde \beta) >0$ such that 
\begin{align} \label{eq-p-minor-gamma}\inf_{x\in \sV_R} \sQ_\tau ^{v,w}(x,A)\geq \alpha \nu_{\mathfrak{K}}(A),\end{align}
where $\nu_{\mathfrak{K}}(\cdot)\doteq \frac{m(\cdot \cap \mathfrak{K})}{m(\mathfrak{K})}$ and $m$ is the Lebesgue measure on $\RR^d$. Moreover, the following hold: 
\begin{enumerate} \item [(i)]The above constants are independent of either $(\bar v,w^*)$ or $(v^*, \bar w)$, and  $\sQ^{v,w}_\tau$ satisfies~\eqref{eq-minor} of Proposition~\ref{prop-contract}.
\item [(ii)]$\alpha_{\inf}\doteq \liminf_{\widetilde \beta\downarrow 0}\alpha(\widetilde \beta)>0$.
\end{enumerate}
\end{lemma} 
\begin{proof} To begin with, recall $Z$ which is the unique strong solution to~\eqref{eq-Z-cont} for $(v,w)= (\bar v,w^*)$ or $(v^*, \bar w)$. 
Let $f:\RR^d\rightarrow \RR^+$ be a non-negative Borel measurable function and also, let the pair $(v,w)$ be as per the hypothesis of the lemma. 
 Define the probability measure $\widetilde \PP$ as follows:  let 
\begin{align}
\Pi_t = \exp\Big( \int_0^t w(s,Z_s) \D W_s - \frac{1}{2} \int_0^t \|w(s,Z_s)\|^2 \D s\Big)
\end{align}
be a $\cF_t$-adapted process and for a $B\in \cF$,  $\widetilde \PP(B)\doteq \int_B \Pi_t (\omega) \D \PP(\omega)$. Then, from Girsanov's theorem \cite[Theorem 5.1]{karatzas_brownian_1998}, we have
\begin{align*}
\big(\sQ_\tau^{v,w} e^{f}\big)(x)&= \E^{v,w}_x\Big[ e^{f(X_\tau)} \Pi_\tau \Big]\\
&= \exp\Big(\log \E^{v,w}_x\Big[ e^{f(X_\tau) + \int_0^\tau w(s,Z_s) \D W_s - \frac{1}{2} \int_0^\tau \|w(s,Z_s)\|^2 \D s}  \Big]\Big)\\
&\geq \exp\Big(\E^{v,w}_x\Big[f(X_\tau) + \int_0^\tau w(s,Z_s) \D W_s - \frac{1}{2} \int_0^\tau \|w(s,Z_s)\|^2 \D s  \Big] \Big)\\
&\geq \exp\Big(\E^{v,w}_x\Big[f(X_\tau) - \frac{1}{2} \int_0^\tau \|w(s,Z_s)\|^2 \D s  \Big] \Big).
\end{align*} 
In the above, to get the first inequality, we use Jensen's inequality. Letting 
\[\widehat \Pi(R)\doteq \inf_{x\in \sV_R} \exp\Big(- \frac{1}{2}\E^{v,w}_x\Big[ \int_0^\tau \|w(s,Z_s)\|^2 \D s  \Big] \Big),\] the above display becomes
\begin{align*}
\big(\sQ_\tau^{v,w} e^{f}\big)(x) \geq  \widehat \Pi(R) \exp\Big(\big(\tP_\tau ^v f\big)(x)\Big).
\end{align*}
Now we choose $f(x)= \Ind _A(x)$ for a Borel set $A$ such that $\sup_{x\in \sV_R}\sQ^{v,w}_\tau(x,A^c)\leq  \widetilde \Pi(R)$. From here, adding  $-\widetilde \Pi(R)$ on both sides of the above display (for $x\in B_R$) gives us
\begin{align*}
e \sQ_\tau^{v,w} (x,A) +\sQ_\tau^{v,w} (x, A^c)-\widetilde \Pi(R) &\geq  \widehat \Pi(R) \exp\Big(\tP_\tau ^v(x,A)\Big)-\widetilde \Pi(R)\\
\implies e \sQ_\tau^{v,w} (x,A) &\geq \widehat \Pi(R) \tP_\tau ^v(x,A)\,.
\end{align*}
To get the second line, we use the fact that $\sup_{x\in \sV_R}\sQ^{v,w}_\tau(x,A^c)\leq  \widetilde \Pi(R)$ on the left hand side and the fact that $1+x\leq e^x$ for $x\in \RR$ on the right hand side. To summarize, we have shown that for a Borel set $A\subset \RR^d$ such that $\sup_{x\in \sV_R}\sQ^{v,w}_\tau(x,A^c)\leq  \widetilde \Pi(R)$, we have
$$ \inf_{x\in \sV_R}\sQ_\tau^{v,w} (x,A) \geq \frac{\widehat \Pi(R)}{e} \inf_{x\in \sV_R}\tP_\tau ^v(x,A)\geq  \frac{\widehat \Pi(R)\alpha_{R,\tau} m(\mathfrak{K})}{e} \nu_{\mathfrak{K}}(A), $$
for every compact set $\mathfrak{K}\subset \RR^d$. To get the second inequality, we use  Corollary~\ref{cor-minor} as it holds under Assumption~\ref{a-regularity}. Recall that $\nu_{\mathfrak{K}}(\cdot)= \frac{m(\cdot \cap \mathfrak{K})}{m(\mathfrak{K})}$. By choosing $\mathfrak{K}$ large enough such that $ \sup_{x\in \sV_R}\sQ_\tau^{v,w} (x,\mathfrak{K}^c)< \widetilde \Pi(R) $, we can ensure that 
$$  \inf_{x\in \sV_R}\sQ_\tau^{v,w} (x,A) \geq   \frac{\widehat \Pi(R)}{e} \alpha_{R,\tau} m(\mathfrak{K})\nu_{\mathfrak{K}}(A),$$
for every Borel set $A\subset \RR^d$.  Therefore, the proof of~\eqref{eq-p-minor-gamma}  is complete with $\alpha (\widetilde \beta)\doteq   \frac{\widehat \Pi(R)}{e} \alpha_{R,\tau} m(\mathfrak{K})  $. The rest of the proof is straightforward and is hence, omitted.
\end{proof}

\subsection{ Convergence of $\Phi_\gamma^n$ to $0$ in $\vvvert{\cdot}_{\Upu,\sV}$}\label{sec-V-gamma-semi-contract} In this section, for a small enough $\Upu>0$,  we analyze the convergence of $\Phi_\gamma^n$ to $0$ in the norm $\vvvert{\cdot}_{\Upu,\sV}$. 
\begin{lemma}\label{lem-V-conv-gamma}
Suppose Assumptions~\ref{a-regularity} and~\ref{a-main-gamma-conv} hold and $\|\Phi^0\|_{\beta,\sV}\leq 1$.  Then, for small enough $\Upu>0$  and for $\tau>\widetilde \tau( \beta)$ (see Lemma~\ref{lem-exp-conv} for its definition)  and $0<\kappa_\gamma
<1$ (independent of $n$) ,  
$$ \vvvert{\Phi^n_\gamma}_{\Upu,\sV}\leq  \kappa_\gamma\vvvert{\Phi^{n-1}_\gamma}_{\Upu,\sV}\,.$$
In particular,~\eqref{eq-V-gamma-semi-contract} holds with $\kappa_\gamma$: $\vvvert{ \Phi^n_\gamma}_{\Upu,\sV}\leq  (\kappa_\gamma)^n\vvvert{ \Phi^0_\gamma}_{\Upu,\sV}$.
\end{lemma}

\begin{proof}
To begin with, recall the two pairs $(\bar v,w^*)$ and $(v^*,\bar w)$ where $v^*$ and $\bar v$ are defined via~\eqref{eq-min-3} and~\eqref{eq-min-1}, respectively, and $w^*$ and $\bar w$ via~\eqref{def-w}.  We note that for either of these pairs, the  statements of Lemmas~\ref{lem-exp-conv}--\ref{lem-exp-minor} hold which we will use below to invoke Proposition~\ref{prop-contract}.   For the pair $(\bar v,w^*)$, using~\eqref{eq-hjb-gamma-m}, we get  
\begin{align*}
\bar \Lg^{\bar v,w^*}V_\gamma^*(x) +{\gamma} r(x,\bar v(t,x)) -\frac{1}{2}\|w^*(x)\|^2&\geq {\gamma} \Lambda_\gamma^* \,,
\end{align*}
and using~\eqref{eq-rvi-mod-gamma-int}, we get
\begin{align*}
\partial_t  \overline V_\gamma(t,x)\geq \bar \Lg^{ \bar v,w^*} \overline V_\gamma(t,x)+ {\gamma} r(x, \bar v(t,x))-\frac{1}{2}\|w^*(x)\|^2- {\gamma} \Lambda^*_{\gamma}\,. 
\end{align*}
Now applying It\^o's formula to $V_\gamma^*(Z_t)$ and $\overline V_\gamma(\tau-t,Z_t)$ under $(v,w)= (\bar v,w^*)$ up to $t=\tau$, we have
\begin{align*}
V_\gamma^*(x)\leq \E_x^{\bar v,w^*}\Big[\int_0^\tau \Big({\gamma}r(Z_t, \bar v(t,Z_t)\big) - \frac{1}{2}\|w^*(Z_t)\|^2- {\gamma}\Lambda_\gamma^*\Big)\D t\Big] + \E_x^{\bar v,w^*} \Big[V_\gamma^*(Z_\tau)\Big] \,,
\end{align*} 
and
\begin{align*}
\overline V_\gamma^n(x)\geq \E_x^{\bar v,w^*}\Big[\int_0^\tau \Big({\gamma}r(Z_t, \bar v(t,Z_t)\big) - \frac{1}{2}\|w^*(Z_t)\|^2- {\gamma} \Lambda^*_{\gamma}\big)\D t\Big] + \E_x^{\bar v,w^*} \Big[\overline V_\gamma^{n-1}(Z_\tau)\Big]\,. 
\end{align*} 
From the above two displays, we have
\begin{align}\label{eq-mk-exp-ub} \overline V_\gamma^n(x)-V_\gamma^*(x)\geq   \E_x^{\bar v,w^*} \Big[\overline V_\gamma^{n-1}(Z_\tau) -V_\gamma^*(Z_\tau)\Big]=\sQ_\tau^{\bar v,w^*}\big(V_\gamma^{n-1} -V_\gamma^*\big)(x) \,.\end{align}
The equality above follows from the definition of $\sQ_\tau^{\bar v,w^*}$. Similarly, for the pair $( v^*,\bar w)$, using~\eqref{eq-hjb-gamma-m}, we get  
\begin{align*}
\bar \Lg^{v^*,\bar w}V_\gamma^*(x) +{\gamma} r\big(x,v^*(x)\big) -\frac{1}{2}\|\bar w(t,x)\|^2&\leq {\gamma} \Lambda_\gamma^*\,,
\end{align*}
and using~\eqref{eq-rvi-mod-gamma-int}, we get
\begin{align*}
\partial_t  \overline V_\gamma(t,x)\leq \bar \Lg^{v^*,\bar w} \overline V_\gamma(t,x)+ {\gamma} r\big(x,v^*(x)\big)-\frac{1}{2}\|\bar w(t,x)\|^2- {\gamma}\Lambda^*_{\gamma}\,. 
\end{align*}
Now applying It\^o's formula to $V_\gamma^*(Z_t)$ and $\overline V_\gamma(\tau-t,Z_t)$ under $(v,w)=(v^*,\bar w)$ up to $t=\tau$, we have
\begin{align*}
V_\gamma^*(x)\geq \E_x^{ v^*,\bar w}\Big[\int_0^\tau \Big({\gamma}r\big(Z_t, v^*(Z_t)\big) - \frac{1}{2}\|\bar w(t, Z_t)\|^2- {\gamma}\Lambda_\gamma^*\Big)\D t\Big] + \E_x^{v^*,\bar w} \Big[V_\gamma^*(Z_\tau)\Big]\,,
\end{align*} 
and 
\begin{align*}
\overline V_\gamma^n(x)\leq \E_x^{ v^*,\bar w}\Big[\int_0^\tau \Big({\gamma}r\big(Z_t, v^*(Z_t)\big) - \frac{1}{2}\|\bar w(t, Z_t)\|^2- {\gamma}\Lambda^*_{\gamma} \Big)\D t\Big] + \E_x^{ v^*,\bar w} \Big[\overline V_\gamma^{n-1}(Z_\tau)\Big]\,. 
\end{align*} 
From the above two displays, we have 
\begin{align}\label{eq-mk-exp-lb} \overline V_\gamma^n(x)-V_\gamma^*(x)\leq   \E_x^{v^*,\bar w} \Big[\overline V_\gamma^{n-1}(Z_\tau) -V_\gamma^*(Z_\tau)\Big]=\sQ_\tau^{ v^*,\bar w}\big(V_\gamma^{n-1} -V_\gamma^*\big)(x)\,.\end{align}
The equality above follows from the definition of $\sQ_\tau^{ v^*,\bar w}$.
Using Lemmas~\ref{lem-exp-conv} and~\ref{lem-exp-minor},  following the arguments of in the proof of Lemma~\ref{lem-V-conv} and applying Proposition~\ref{prop-contract} (in conjunction with Remark~\ref{rem-special-beta}), we get
\begin{align}\label{eq-singlestep-contract} \vvvert{\overline V_\gamma^n-V_\gamma^*}_{\Upu,\sV}\leq \kappa_\gamma(\widetilde \beta) \vvvert{\overline V_\gamma^{n-1}-V_\gamma^*}_{\Upu,\sV}\,.\end{align}
Here,  $0<\beta\leq  \frac{\alpha(\widetilde \beta)}{2K(\widetilde \beta)}$  and  $$\kappa_\gamma (\widetilde\beta) \doteq  \Big(1-\frac{\alpha(\widetilde \beta)}{2}\Big)\vee  \Big(\frac{2(1-\varrho(\widetilde \beta))+\alpha (\widetilde \beta) (1+\varrho(\widetilde \beta))}{2(1-\varrho(\widetilde \beta))+ 2\alpha(\widetilde \beta)}\Big)\,.$$ In particular, as $\alpha_{\inf}>0$ and $K_{\sup}<\infty$, ~\eqref{eq-singlestep-contract} also holds for  $0<\beta< \frac{\alpha_{\inf}} {K_{\sup}} $ (independent of $\widetilde \beta$).   Since  $\overline V^{n+1}_\gamma(0)=0$, we have  $$\vvvert{\overline V^{n+1}_\gamma-V^*_\gamma}_{\beta,\sV}=\|\overline V^{n+1}_\gamma-V^*_\gamma\|_{\beta,\sV}\leq 1\,.$$
Therefore, if $\|\Phi^0_\gamma\|_{\beta,\sV}\leq 1$  for $\beta \leq \frac{\alpha_{\inf}} {K_{\sup}} $, then we can iterate the argument $n$ times and obtain the result.
\end{proof}

\begin{remark}
From the repeated application of Lemma~\ref{lem-V-conv-gamma} and the definition of $\vvvert{\cdot}_{\beta,\sV}$, we immediately obtain 
\begin{align}\label{eq-est-v-v^*-gamma}
 \sup_{0\neq x\in \RR^d}\frac{|\Phi_\gamma^n(x)- \Phi_\gamma^n(0)|}{2+\Upu \sV(x)+\Upu \sV(0)}\leq \sup_{x\neq y} \frac{|\Phi_\gamma^n(x)- \Phi_\gamma^n(y)|}{2+\Upu \sV(x)+\Upu \sV(y)}=\vvvert{\Phi_\gamma^n}_{\Upu,\sV}\leq (\kappa_\gamma)^n \vvvert{\Phi_\gamma^0}_{\Upu,\sV}\,.
\end{align}

\end{remark}

\subsection{ Convergence of $\Phi_\gamma^n (0)$ to $0$}\label{sec-lambda-gamma-semi-contract}
In the following,  $\widehat v\in \Um$ denotes the minimizer of  that satisfies~\eqref{eq-rvi-mod-gamma-m} and $\widehat w(t,x)\doteq \Sigma(x)\transp\grad  V_\gamma(t,x)$ when $V_\gamma(0,x)=V_\gamma^{n-1}(x)$. We again suppress the dependence as we always fix $n$ whenever $\widehat v$ and $ \widehat w$ are involved. 
\begin{lemma}\label{lem-lambda-gamma-semi-contract}Suppose Assumptions~\ref{a-regularity} and~\ref{a-main-gamma-conv} hold. Then, for $\widetilde \kappa_\gamma $ in~\eqref{eq-widekappa} and $C_{\gamma,0}\doteq 2+\Upu\varrho\sV(0)+\Upu K$,~\eqref{eq-lambda-gamma-semi-contract} holds.

\end{lemma}
\begin{proof}   Consider the pairs $(v^*,\widehat  w)$ and $(\widehat v,w^*)$  and also recall that $V_\gamma^{n}(x)= V_\gamma(\tau,x)$ with $V_\gamma$ satisfying \begin{align*}
\partial_t   V_\gamma(t,x)= \min_{u\in \bU}\max_{w\in \RR^d} \Big[\bar \Lg^{u,w}  V_\gamma(t,x)+ {\gamma} r(x,u)-\frac{1}{2}\|w\|^2\Big]- \exp\big( \delta V_\gamma^{n-1}(0)\big), \quad V_\gamma(0,x)= V_\gamma^{n-1}(x)\,.
\end{align*} Following the similar arguments as used in the proof of Lemma~\ref{lem-V-conv-gamma} for $V_\gamma^*$ and $V_\gamma$, instead of $V_\gamma^*$ and $\overline V_\gamma$ and using the definition for $\widehat\Lambda^*_\gamma$ and $\Phi^n_\gamma$, we can obtain 
\begin{align*}
\big(\sQ_\tau ^{\widehat v,w^*}\Phi^{n-1}_\gamma\big)(x) &\leq \Phi^n_\gamma (x) - {\tau}{}\big(\gamma\Lambda_\gamma^*-\exp\big( \delta V^{n-1}(0)\big)\big) \leq \big(\sQ_\tau ^{ v^*,\widehat  w}\Phi^{n-1}_\gamma\big)(x) \,,\\
\big(\sQ_\tau ^{\widehat  v,w^*}\Phi^{n-1}_\gamma\big)(x) &\leq \Phi^n_\gamma (x) - {\tau}{}\big(\exp\big(\delta  \widehat \Lambda^*_\gamma\big)-\exp\big( \delta V^{n-1}(0)\big)\big) \leq \big(\sQ_\tau ^{ v^*,\widehat  w}\Phi^{n-1}_\gamma\big)(x)\,,\\
\big(\sQ_\tau ^{\widehat  v,w^*}\Phi^{n-1}_\gamma\big)(x) &\leq \Phi^n_\gamma (x) - {\tau}{}\widehat \Phi^{n-1}_\gamma \leq \big(\sQ_\tau ^{ v^*,\widehat  w}\Phi^{n-1}_\gamma\big)(x),
\end{align*}
where $\widehat \Phi^n_\gamma\doteq \exp\big(\delta  \widehat \Lambda^*_\gamma\big)-\exp\big( \delta V^{n}(0)\big) $. From here, the rest of the proof follows along the same line as the proof of Lemma~\ref{lem-L-conv}. Using~\eqref{eq-est-v-v^*-gamma}, for any $x\in \RR^d$ we have 
 \begin{align*}
\Phi^{n-1}_\gamma(0) - &(\kappa_\gamma)^{n-1}\Big(2+\Upu \sV(x)+\Upu \sV(0)\Big)\vvvert{\Phi_\gamma^0}_{\Upu,\sV}\\
&\leq \Phi^{n-1}_\gamma(x) \leq \Phi^{n-1}_\gamma(0) + (\kappa_\gamma)^{n-1}\Big(2+\Upu \sV(x)+\Upu \sV(0)\Big)\vvvert{\Phi_\gamma^0}_{\Upu,\sV}\,.
 \end{align*}
Define $\bar \sV_\Upu(x)\doteq 2+\Upu \sV(x)+\Upu \sV(0)$. From the above two displays, we obtain
\begin{align*}
\Phi_\gamma^{n-1}(0)-(\kappa_\gamma)^{n-1}\vvvert{\Phi_\gamma^0}_{\Upu,\sV}\E_0^{\widehat  v,w^*} \big[\bar \sV_\Upu (Z_\tau)\big]&\leq \Phi_\gamma^{n}(0)  +  {\tau}{}\big(\gamma\Lambda_\gamma^*-\exp\big( \delta V^{n-1}(0)\big)\big)\\
& \leq \Phi_\gamma^{n-1}(0)+(\kappa_\gamma)^{n-1}\vvvert{\Phi_\gamma^0}_{\Upu,\sV}\E_0^{ v^*,\widehat w} \big[\bar \sV_\Upu (Z_\tau)\big]\,.
\end{align*}
Using Lemma~\ref{lem-exp-conv} and It\^o's formula, we can conclude that  for  $(v,w)= (\widehat  v,w^*)$ or $(v^*,\widehat  w)$,
\begin{align*} 
\E_0^{v,w}\big[\bar \sV(Z_\tau)\big] & = 2+\Upu \E_0^{v,w}\big[\sV(Z_\tau)\big] +\Upu \sV(0) \doteq  C_{\gamma,0}\,. \end{align*}
From here, we obtain
\begin{align*}
\Phi^{n-1}_\gamma(0)+ &{\tau}{}\widehat \Phi^{n-1}_\gamma- C_{\gamma,0}(\kappa_\gamma)^{n-1}\vvvert{\Phi_\gamma^0}_{\Upu,\sV}\leq \Phi_\gamma^n(0)\leq \Phi_\gamma^{n-1}(0)+ {\tau}{}\widehat \Phi^{n-1}_\gamma+ C_{\gamma,0}(\kappa_\gamma)^{n-1}\vvvert{\Phi_\gamma^0}_{\Upu,\sV}\,.
\end{align*}
Using the inequality: $e^x \geq 1+x$ for $x\in \RR$ and the definition of $\widehat \Phi^{n-1}_\gamma$, we can immediately conclude that 
$$ - \delta \exp\big(\delta \Phi^{n-1}_\gamma(0)\big) \Phi^{n-1}_\gamma (0)\leq  \widehat \Phi^{n-1}_\gamma \leq -  \delta \exp\big(\delta \widehat \Lambda^*_\gamma\big) \Phi^{n-1}_\gamma (0)\,.$$
From the above two displays, we have
\begin{align*}
 \Phi_\gamma^n(0)&\leq\Big(1- \tau\delta \exp\big(\delta \widehat \Lambda^*_\gamma\big)\Big) \Phi_\gamma^{n-1}(0)+ C_{\gamma,0}(\kappa_\gamma)^{n-1}\vvvert{\Phi_\gamma^0}_{\Upu,\sV}\,.
\end{align*} 
Similarly, we also have 
\begin{align*}
 \Phi_\gamma^n(0)&\geq  \Big(1- \tau\delta \exp\big(\delta \Phi^{n-1}_\gamma(0)\big)\Big) \Phi_\gamma^{n-1}(0)+ C_{\gamma,0}(\kappa_\gamma)^{n-1}\vvvert{\Phi_\gamma^0}_{\Upu,\sV}\,.
\end{align*} 
Therefore, from the definition of $\widetilde \kappa_\gamma$, we have 
$$ |\Phi_\gamma ^n(0)|\leq\widetilde \kappa_\gamma |\Phi_\gamma^{n-1}(0)|+ C_{\gamma,0}(\kappa_\gamma)^{n-1}\vvvert{\Phi_\gamma^0}_{\Upu,\sV} $$
This completes the proof.
\end{proof}
\begin{proof}[\bf Completing the proof of Theorem~\ref{thm-rvi-gamma}] Combining Lemmas~\ref{lem-V-conv-gamma} and~\ref{lem-lambda-gamma-semi-contract}, we immediately obtain Theorem~\ref{thm-rvi-gamma}. 
\end{proof}

\medskip

\appendix
\section{A  contraction property of Markov kernels in general state spaces }
We now recall a result from \cite{hairer2011} for Markov chains in general state spaces. Most of the content of this  section  is taken from   \cite{hairer2011}, where the authors give a new and simpler proof of  Harris' ergodic theorem of Markov chains in general state spaces; see Theorem 1.2 of that paper.  Their proof of Harris' ergodic theorem critically hinges on the contraction property of  the associated Markov kernel with respect to a particular weighted semi-norm (see Section~\ref{sec-semi-norm} below for its introduction). In this work, we extensively make use of this contraction property,  in the context of state space being $\RR^d$.  Before we proceed further, let us introduce a few important notions. We fix a measurable space $(\cX,\cB)$ for the entire section.

\subsection{ A weighted norm and a corresponding weighted semi-norm}\label{sec-semi-norm} Recall the notation $ \|f\|_{\Upu,\mathscr{F}}$ from~\eqref{def-norm} and $\cC_{\beta,\mathscr{F}}(\cX)$.    For any $f\in \cC_{\Upu,\mathscr{F}}(\cX)$, we also define a semi-norm $\vvvert{\cdot}_{\Upu,\mathscr{F}}$ that is very closely related to the norm $\|\cdot\|_{\Upu,\mathscr{F}}$:
\begin{align}\label{def-norm-w} \vvvert{f}_{\Upu,\mathscr{F}}\doteq  \sup_{x\neq y\in \cX}\frac{|f(x)-f(y)|}{2+\Upu\mathscr{F}(x)+\Upu\mathscr{F}(y)}\,.\end{align}
It is easy to verify that $\vvvert{\cdot}_{\Upu,\mathscr{F}}$ is a semi-norm.  To the best of our knowledge, the semi-norm $\vvvert{\cdot}_{\Upu,\mathscr{F}}$ is first introduced in \cite{hairer2011}.  
The  result below (which is \cite[Lemma 2.1]{hairer2011}) shows explicitly the aforementioned relation between  $\|\cdot\|_{\Upu,\mathscr{F}}$   and $\vvvert{\cdot}_{\Upu,\mathscr{F}}$. 
\begin{lemma}\label{lem-compare} The following identity holds:
$$ \vvvert{f}_{\Upu,\mathscr{F}}=\inf_{c\in \RR} \|f+c\|_{\Upu,\mathscr{F}}\,.$$
Moreover, we have \[ \inf_{c\in \RR} \|f+c\|_{\Upu,\mathscr{F}}= \|f+c^*_f\|_{\Upu,\mathscr{F}},\]
for \begin{align}\label{def-c*}c^*_f\doteq \inf_{x\in \cX}\big(\vvvert{f}_{\Upu,\mathscr{F}}\big(1+\Upu\mathscr{F}(x)\big)-f(x)\big).\end{align}
\end{lemma}

\subsection{Contraction property with respect to the semi-norm in~\eqref{def-norm-w}}
We now state the main result of this section which is Theorem 3.1 of \cite{hairer2011}.
\begin{proposition}\label{prop-contract}
Let  $\tP:\cX\times\cB\rightarrow [0,1]$ be a Markov kernel on $(\cX,\cB)$. 
Suppose that $\tP$ satisfies the following two properties: 
\begin{enumerate}
\item[(i)] There exist an inf-compact function $\mathfrak{V}:\cX\rightarrow [0,\infty)$,  and constants $0<\eta<1$ and $K>0$ such that  for every $x\in \cX$,
\begin{align}\label{eq-drift} \big(\tP\mathfrak{V}\big)(x)\leq \eta \mathfrak{V}(x) + K\,. \end{align}
\item[(ii)] There exists constants $0<\alpha<1$, $R>\frac{2K}{1-\eta}$ and a probability measure $\nu\in \calP(\cX)$ such that
\begin{align}\label{eq-minor} \inf_{x\in \cC_R} \tP(x,\cdot)\geq \alpha \nu (\cdot),\end{align}
for $\cC_R\doteq \{x\in \cX: \mathfrak{V}(x)\leq R\}$.  Here, $K$ and $\eta$ are the same as in (i) above.
\end{enumerate}
Then, for 
\begin{align}\label{eq-constants-contract} \eta_0\doteq \eta+\frac{2K}{R}<1,\quad 0<\alpha_0<\alpha,\quad \Upu\doteq   \frac{\alpha_0}{K},\quad\text{and}\quad \kappa \doteq \big(1-\alpha+\alpha_0\big)\vee \big(\frac{2+\Upu R\eta_0}{2+\Upu R}\big)<1,\end{align}
we have
\begin{align}\label{eq-contract}
\vvvert{\tP f}_{\Upu,\mathfrak{V}}\leq \kappa \vvvert{f}_{\Upu,\mathfrak{V}}\,.
\end{align}
\end{proposition}
\begin{remark}\label{rem-special-beta} In the above proposition, $\eta_0<1$ follows from the choice of $R$, and $\kappa<1$ follows from the choice of $\alpha_0$ and the fact that $\eta_0<1$\,. For the sake of illustrating the contraction in a more explicit manner (without focusing on the optimal contraction rate), we set $R=\frac{4K}{1-\gamma}$, $\alpha_0=\frac{\alpha}{2}$. Then, from~\eqref{eq-constants-contract}, we have $\eta_0= \frac{1+\eta}{2}$,
 $\beta= \frac{\alpha}{2K}$ and $$\kappa= \Big(1-\frac{\alpha}{2}\Big)\vee  \Big(\frac{2(1-\eta)+\alpha (1+\eta)}{2(1-\eta)+ 2\alpha}\Big)\,.$$
\end{remark}
\begin{remark} The value of $\Upu>0$ is irrelevant for $\vvvert{\cdot}$ to be well-defined. However, $\Upu$ is chosen sufficiently small for the above result to hold. It should also be pointed out that in \cite[Theorem 3.1]{hairer2011}, $\Upu$  is chosen to be only $\frac{\alpha_0}{K}$. However, one can easily check that the above proposition holds for any $\Upu\leq \frac{\alpha_0}{K}$ with the only caveat being the `worsening' of $\kappa$ (that is, as $\Upu$ decreases, $\kappa$ increases). Hence, $\Upu=\frac{\alpha_0}{K}$ is the optimal choice of $\Upu$ for which the above proposition holds. 
\end{remark}
\begin{remark} As far as this work is concerned, the strength of Proposition~\ref{prop-contract} lies in the  fact that it provides us with explicit values of constants involved like $\Upu$ and $\kappa$.  We remark that  the explicit values of $\Upu$ and $\kappa$ are not important in the case of the CEC problem; only the existence of $\Upu$ and that fact that $\kappa<1$ are sufficient for the exponential convergence in Theorem~\ref{thm-rvi-0}. On the other hand, the explicit values of $\Upu$ and $\kappa$ become important in the proof of Theorem~\ref{thm-rvi-gamma}.
\end{remark}

\medskip

\bibliographystyle{abbrv}
\bibliography{RVI}

@article{hordijk1974convergence,
 ISSN = {00251909, 15265501},
 URL = {http://www.jstor.org/stable/2629972},
 author = {Arie Hordijk and Henk Tijms},
 journal = {Management Science},
 number = {11},
 pages = {1432--1438},
 publisher = {INFORMS},
 title = {Convergence Results and Approximations for Optimal $(s, {S})$ Policies},
 urldate = {2026-05-08},
 volume = {20},
 year = {1974}
}

@article{schweitzer1988contraction,
  title={Contraction mappings underlying undiscounted {M}arkov decision problems -- {II}},
  author={Schweitzer, PJ},
  journal={Journal of Mathematical Analysis and Applications},
  volume={132},
  number={1},
  pages={154--170},
  year={1988},
  publisher={Elsevier}
}

@article{federgruen1978contraction,
  title={Contraction mappings underlying undiscounted {M}arkov decision problems},
  author={Federgruen, Awi and Schweitzer, Paul J and Tijms, Hendrik Cornelis},
  journal={Journal of Mathematical Analysis and Applications},
  volume={65},
  number={3},
  pages={711--730},
  year={1978},
  publisher={Elsevier}
}

@article{cavaozs1996value1,
author = {Cavazos-Cadena, Rolando},
title = {Value Iteration in a Class of Communicating {M}arkov Decision Chains with the Average Cost Criterion},
journal = {SIAM Journal on Control and Optimization},
volume = {34},
number = {6},
pages = {1848-1873},
year = {1996},
doi = {10.1137/S1064827590192863},

URL = { 
    
        https://doi.org/10.1137/S1064827590192863
    
    

},
eprint = { 
    
        https://doi.org/10.1137/S1064827590192863
    
    

}
}

@article{cavazos1996value,
 ISSN = {00219002},
 URL = {http://www.jstor.org/stable/3214980},
 author = {Rolando Cavazos-Cadena and Emmanuel Fernández-Gaucherand},
 journal = {Journal of Applied Probability},
 number = {4},
 pages = {986--1002},
 publisher = {Applied Probability Trust},
 title = {Value Iteration in a Class of Average Controlled {M}arkov Chains with Unbounded Costs: Necessary and Sufficient Conditions for Pointwise Convergence},
 urldate = {2026-05-06},
 volume = {33},
 year = {1996}
}

@article{aviv1999value,
title = {The value iteration method for countable state {M}arkov decision processes},
journal = {Operations Research Letters},
volume = {24},
number = {5},
pages = {223-234},
year = {1999},
issn = {0167-6377},
doi = {https://doi.org/10.1016/S0167-6377(99)00015-2},
url = {https://www.sciencedirect.com/science/article/pii/S0167637799000152},
author = {Yossi Aviv and Awi Federgruen}
}

@article{arapostathis2019open,
author = {Arapostathis, Ari},
title = {Open Problem{--}Convergence and Asymptotic Optimality of the Relative Value Iteration in Ergodic Control},
journal = {Stochastic Systems},
volume = {9},
number = {3},
pages = {292-294},
year = {2019}
}

@article{cavazos2003value,
author = {Cavazos-Cadena, Rolando and Montes-de-Oca, Ra\'{u}l},
title = {The Value Iteration Algorithm in Risk-Sensitive Average {M}arkov Decision Chains with Finite State Space},
journal = {Mathematics of Operations Research},
volume = {28},
number = {4},
pages = {752-776},
year = {2003}
}

@article{anugu2025new,
  title={Jacobi-like relative value iteration algorithms for the ergodic risk-sensitive control of {M}arkov chains},
  author={Anugu, Sumith Reddy and Pang, Guodong and Sassone, Nicola},
 journal={submitted}
}

@article{bielecki1999risk,
  title={Risk sensitive control of finite state {M}arkov chains in discrete time, with applications to portfolio management},
  author={Bielecki, Tomasz and Hern{\'a}ndez-Hern{\'a}ndez, Daniel and Pliska, Stanley R},
  journal={Mathematical Methods of Operations Research},
  volume={50},
  number={2},
  pages={167--188},
  year={1999},
  publisher={Springer}
}

@article{bertsekas1998new,
author = {Bertsekas, Dimitri P.},
title = {A New Value Iteration method for the Average Cost Dynamic Programming Problem},
journal = {SIAM Journal on Control and Optimization},
volume = {36},
number = {2},
pages = {742-759},
year = {1998}}

@article{white1963dynamic,
title = {Dynamic programming, {M}arkov chains, and the method of successive approximations},
journal = {Journal of Mathematical Analysis and Applications},
volume = {6},
number = {3},
pages = {373-376},
year = {1963},
issn = {0022-247X},
doi = {https://doi.org/10.1016/0022-247X(63)90017-9},
url = {https://www.sciencedirect.com/science/article/pii/0022247X63900179},
author = {D.J White}
}

@article{borkar2002risk,
  title={Risk-sensitive optimal control for {M}arkov decision processes with monotone cost},
  author={Borkar, Vivek S and Meyn, Sean P},
  journal={Mathematics of Operations Research},
  volume={27},
  number={1},
  pages={192--209},
  year={2002},
  publisher={INFORMS}
}

@article{biswas2022ergodic,
     author = {Biswas, Anup and Pradhan, Somnath},
     title = {Ergodic risk-sensitive control of {M}arkov processes on countable state space revisited},
     journal = {ESAIM: Control, Optimisation and Calculus of Variations},
     year = {2022},
     publisher = {EDP-Sciences},
     volume = {28, article number 26},
     doi = {10.1051/cocv/2022018},
     mrnumber = {4429406},
     zbl = {1493.90218},
     language = {en},
     url = {https://www.numdam.org/articles/10.1051/cocv/2022018/}
}

@ARTICLE{meyn1997policy,
  author={Meyn, S.P.},
  journal={IEEE Transactions on Automatic Control}, 
  title={The policy iteration algorithm for average reward {M}arkov decision processes with general state space}, 
  year={1997},
  volume={42},
  number={12},
  pages={1663-1680},
  doi={10.1109/9.650016}}

@article{arapostathis2021policy, title={On the policy improvement algorithm for ergodic risk-sensitive control}, volume={151}, DOI={10.1017/prm.2020.61}, number={4}, journal={Proceedings of the Royal Society of Edinburgh: Section A Mathematics}, author={Arapostathis, Ari and Biswas, Anup and Pradhan, Somnath}, year={2021}, pages={1305–1330}}

@Inbook{arapostathis2012,
author="Arapostathis, Ari",
editor="Hern{\'a}ndez-Hern{\'a}ndez, Daniel
and Minj{\'a}rez-Sosa, J. Adolfo",
title="On the Policy Iteration Algorithm for Nondegenerate Controlled Diffusions Under the Ergodic Criterion",
bookTitle="Optimization, Control, and Applications of Stochastic Systems: In Honor of On{\'e}simo Hern{\'a}ndez-Lerma",
year="2012",
publisher="Birkh{\"a}user Boston",
address="Boston",
pages="1--12"
}

@article{arapostathis1993discrete,
author = {Arapostathis, Aristotle and Borkar, Vivek S. and Fern\'{a}ndez-Gaucherand, Emmanuel and Ghosh, Mrinal K. and Marcus, Steven I.},
title = {Discrete-Time Controlled {M}arkov Processes with Average Cost Criterion: A Survey},
journal = {SIAM Journal on Control and Optimization},
volume = {31},
number = {2},
pages = {282-344},
year = {1993}
}

@book{arapostathis2012ergodic, place={Cambridge}, series={Encyclopedia of Mathematics and its Applications}, title={Ergodic Control of Diffusion Processes}, publisher={Cambridge University Press}, author={Arapostathis, Ari and Borkar, Vivek S. and Ghosh, Mrinal K.}, year={2011}, collection={Encyclopedia of Mathematics and its Applications}}

@article{FM95,
	author = {Fleming, Wendell H. and McEneaney, William M.},
	doi = {10.1137/S0363012993258720},
	eprint = {https://doi.org/10.1137/S0363012993258720},
	journal = {SIAM Journal on Control and Optimization},
	number = {6},
	pages = {1881-1915},
	title = {Risk-Sensitive Control on an Infinite Time Horizon},
	url = {https://doi.org/10.1137/S0363012993258720},
	volume = {33},
	year = {1995}}

@article{biswas2022survey,
	author = {Anup Biswas and Vivek S. Borkar},
	journal = {Annual Reviews in Control},
	pages = {118-141},
	title = {Ergodic risk-sensitive control---{A} survey},
	volume = {55},
	year = {2023}}

@article{AB20,
	author = {Arapostathis, Ari and Biswas, Anup},
	journal = {SIAM Journal on Control and Optimization},
	number = {1},
	pages = {85--103},
	publisher = {SIAM},
	title = {A variational formula for risk-sensitive control of diffusions in {$\RR^d$}},
	volume = {58},
	year = {2020}}

@article{ABBK20,
	author = {Arapostathis, Ari and Biswas, Anup and Borkar, Vivek S and Kumar, K Suresh},
	journal = {SIAM Journal on Control and Optimization},
	number = {6},
	pages = {3785--3813},
	publisher = {SIAM},
	title = {A Variational Characterization of the Risk-Sensitive Average Reward for Controlled Diffusions on {$\RR^d$}},
	volume = {58},
	year = {2020}}

@book{karatzas_brownian_1998,
	address = {New York, NY},
	series = {Graduate {Texts} in {Mathematics}},
	title = {Brownian {Motion} and {Stochastic} {Calculus}},
	volume = {113},
	copyright = {http://www.springer.com/tdm},
	isbn = {9780387976556 9781461209492},
	url = {http://link.springer.com/10.1007/978-1-4612-0949-2},
	language = {en},
	urldate = {2026-03-26},
	publisher = {Springer},
	author = {Karatzas, Ioannis and Shreve, Steven E.},
	year = {1998},
	doi = {10.1007/978-1-4612-0949-2}
}

@article{arapostathis2017correction,
author = {Arapostathis, Ari and Borkar, Vivek S.},
title = {A Correction to “A Relative Value Iteration Algorithm for Nondegenerate Controlled Diffusions''},
journal = {SIAM Journal on Control and Optimization},
volume = {55},
number = {3},
pages = {1711--1715},
year = {2017}

}

@Inbook{arapostathis2013relative,
author="Arapostathis, Ari
and Borkar, Vivek S.
and Kumar, K. Suresh",
editor="K{\v{r}}ivan, Vlastimil
and Zaccour, Georges",
title="Relative Value Iteration for Stochastic Differential Games",
bookTitle="Advances in Dynamic Games: Theory, Applications, and Numerical Methods",
year="2013",
publisher="Springer International Publishing",
address="Cham",
pages="3--27",
isbn="978-3-319-02690-9",
doi="10.1007/978-3-319-02690-9_1",
url="https://doi.org/10.1007/978-3-319-02690-9_1"
}

@article{hmedi2023global,
title = {On the global convergence of relative value iteration for infinite-horizon risk-sensitive control of diffusions},
journal = {Systems \& Control Letters},
author = {Hassan Hmedi and Ari Arapostathis and Guodong Pang},
volume = {171},
pages = {105413},
year = {2023},
issn = {0167--6911}

}

@article{arapostathis2020relative,
author = {Arapostathis, Ari and Borkar, Vivek},
year = {2020},
month = {01},
pages = {9-24},
title = {On the relative value iteration with a risk-sensitive criterion},
volume = {122},
journal = {Banach Center Publications},
doi = {10.4064/bc122-1}
}

@article{arapostathis2014convergence,
author = {Arapostathis, Ari and Borkar, Vivek S. and Kumar, K. Suresh},
title = {Convergence of the Relative Value Iteration for the Ergodic Control Problem of Nondegenerate Diffusions under Near-Monotone Costs},
journal = {SIAM Journal on Control and Optimization},
volume = {52},
number = {1},
pages = {1--31},
year = {2014}
}

@article{anugu24ergodicgen,
      title={Ergodic Risk Sensitive Control of Diffusions under a General Structural Hypothesis}, 
      author={Sumith Reddy Anugu and Guodong Pang},
      year={2025},
      eprint={2511.01100},
     journal={arxiv:2511.01100},
      url={https://arxiv.org/abs/2511.01100}
}

@article{arapostathis2012relative,
	author = {Arapostathis, Ari and Borkar, Vivek S.},
	doi = {10.1137/110850529},
	eprint = {https://doi.org/10.1137/110850529},
	journal = {SIAM Journal on Control and Optimization},
	number = {4},
	pages = {1886-1902},
	title = {A Relative Value Iteration Algorithm for Nondegenerate Controlled Diffusions},
	url = {https://doi.org/10.1137/110850529},
	volume = {50},
	year = {2012}}

@inproceedings{hairer2011,
	address = {Basel},
	author = {Hairer, Martin and Mattingly, Jonathan C.},
	booktitle = {Seminar on Stochastic Analysis, Random Fields and Applications VI},
	isbn = {978-3-0348-0021-1},
	pages = {109--117},
	publisher = {Springer Basel},
	title = {Yet Another Look at {H}arris' Ergodic Theorem for {M}arkov Chains},
	year = {2011}}

@article{menozzi2021,
	author = {S. Menozzi and A. Pesce and X. Zhang},
	issn = {0022-0396},
	journal = {Journal of Differential Equations},
	pages = {330-369},
	title = {Density and gradient estimates for non degenerate {B}rownian {SDE}s with unbounded measurable drift},
	volume = {272},
	year = {2021}
}

@article{arapostathis2019strict,
  title={Strict monotonicity of principal eigenvalues of elliptic operators in $\mathbb{R}^d$ and risk-sensitive control},
  author={Arapostathis, Ari and Biswas, Anup and Saha, Subhamay},
  journal={Journal de Math{\'e}matiques Pures et Appliqu{\'e}es},
  volume={124},
  pages={169--219},
  year={2019}
}

\end{document}